\documentclass[11pt]{amsart}
\usepackage[utf8]{inputenc}

\usepackage{geometry}

\usepackage[nocompress]{cite}

\makeatletter
\@namedef{subjclassname@2020}{%
  \textup{2020} Mathematics Subject Classification}
\makeatother

\usepackage{amsmath, amssymb, bbm}
\usepackage{amsthm, thmtools, mathtools}

\usepackage[hidelinks,pdfencoding=auto]{hyperref}
\usepackage[capitalise,nosort]{cleveref}

\declaretheoremstyle[bodyfont=\normalfont]{defnstyle}

\declaretheorem[numberwithin=section,refname={Theorem,Theorems},name=Theorem]{theorem}

\declaretheorem[name=Theorem,numbered=yes]{ltheorem}

\declaretheorem[numberlike=theorem,refname={Corollary,Corollaries},name=Corollary]{corollary}

\declaretheorem[numberlike=theorem,refname={Lemma,Lemmmas},name=Lemma]{lemma}

\declaretheorem[numberlike=theorem,refname={Proposition,Propositions},name=Proposition]{proposition}

\declaretheorem[numberlike=theorem,refname={Definition,Definitions},style=defnstyle,name=Definition]{definition}

\declaretheorem[numberlike=theorem,style=remark]{remark}

\declaretheorem[numberlike=theorem,style=remark,refname={Notation,Notations}]{notation}

\declaretheorem[numbered=no,style=definition]{question}

\declaretheorem[numbered=no,style=remark,qed=$\blacksquare$]{proof}

\usepackage{enumitem}

\newlist{rome}{enumerate}{7}
\setlist[rome]{label=(\roman*)}

\newlist{trome}{enumerate}{1}
\setlist[trome]{label=(\roman*), ref=\thetheorem~(\roman*)}
\crefalias{tromei}{theorem}

\newlist{drome}{enumerate}{1}
\setlist[drome]{label=(\roman*),ref=\thedefinition~(\roman*)}
\crefalias{dromei}{definition}

\newlist{rrome}{enumerate}{1}
\setlist[rrome]{label=(\roman*),ref=\theremark~(\roman*)}
\crefalias{rromei}{remark}

\usepackage{tikz}

\usepackage{xparse}
\usepackage{ifmtarg}
\newlength{\diagrampunctdist}
\newsavebox{\diagrampunct}
\makeatletter
\DeclareDocumentEnvironment{diagram*}{ O{} O{} }
{\xdef\basenode{\@ifmtarg{#2}{current bounding box.south}{#2.base}}\savebox{\diagrampunct}{#1}\begin{center}\begin{tikzpicture}}{\@ifnotmtarg{#1}{\node at ($(\basenode-|current bounding box.east)+(\diagrampunctdist,0)$) {\usebox{\diagrampunct}};\node at ($(\basenode-|current bounding box.west)-(\diagrampunctdist,0)$) {\hphantom{\usebox{\diagrampunct}}};}\end{tikzpicture}\end{center}}
\makeatother

\newcounter{diagram}
\newenvironment{diagram}{\setcounter{diagram}{\value{theorem}}\refstepcounter{theorem}\refstepcounter{diagram}\begin{center}\hfill\begin{tikzpicture}}
    {\end{tikzpicture}\hfill\llap{\normalfont(\thediagram)}\end{center}}

\crefname{diagram}{}{}
\Crefname{diagram}{Diagram}{Diagram}
\numberwithin{diagram}{section} 

\newlength{\ProArLineHeight}
\newlength{\ProArThickness}

\usetikzlibrary{arrows.meta,shapes,arrows,calc,
  decorations.markings,decorations.pathreplacing,
  positioning}

\tikzset{node distance=1.5cm, auto,
  baseline=(current bounding box.center),
  shrink/.style={outer sep=2pt,inner sep=0, minimum size=0},
  squeeze/.style={auto=false,outer sep=3pt,inner sep=0, minimum size=0},
  dot/.style={inner sep=1pt, minimum
    size=0pt,outer sep=5pt,circle,fill},
  a/.style={->}, e/.style={->>},
  d/.style={double, double equal sign distance,-},
  u/.style={dashed,->},
  n/.style={double equal sign distance, -implies},
  ne/.style={double equal sign distance, -},
  t/.style={double distance=2.5pt, -implies, postaction={draw,-}},
  te/.style={double distance=2.5pt, -, postaction={draw,-}},
  cover/.style={preaction={draw=white, -,line width=6pt}},
  over/.style={auto=false,fill=white,inner sep=1.5pt, minimum size=0, outer sep=0},
  s/.style={shorten >=3pt, shorten <=3pt},
  pro/.style={decoration={markings,
      mark=at position 1/2 with {
        \draw[-,line width=\ProArThickness]
        (0,{-\ProArLineHeight/2}) -- (0,{\ProArLineHeight/2});
      }},
    inner sep=\ProArLineHeight,
    postaction=decorate,
  },
  proarrow/.style={->,pro},
  prodotted/.style={->,dotted,pro},
  proequal/.style={double distance=1.5pt,pro},
  la/.style={scale=0.8}}

\def\cellslide{0.5}
\def\celllength{.2cm}

\NewDocumentCommand{\cell}{ O{} O{n} O{\cellslide} O{\celllength} m m m }{
  \coordinate (mid) at ($({#5})!{#3}!({#6})$);
  \coordinate (start) at ($(mid)!{#4}!({#5})$);
  \coordinate (end) at ($(mid)!{#4}!({#6})$);
  \draw[#2] (start) to node
  [inner sep=4pt,outer sep=0,minimum size=0,#1]{{#7}} (end);
}

\NewDocumentCommand{\celli}{ O{} O{n} O{\cellslide} O{\celllength} m m m }{
  \coordinate (mid) at ($({#5})!{#3}!({#6})$);
  \coordinate (start) at ($(mid)!{#4}!({#5})$);
  \coordinate (end) at ($(mid)!{#4}!({#6})$);
  \draw[#2] (start) to node(label)[inner sep=4pt,outer sep=0,minimum size=0,#1]{{#7}} (end);
  \coordinate (far) at ($(end)+(mid)-(label)$);
  \node[] at ($(end)!7pt!(far)$) {$\scriptscriptstyle\iso$} ;
}

\newcommand{\inv}{{\hspace{1pt}\scriptscriptstyle\text{-}1}}

\newdimen{\pullbackmarkerlength}
\def\pullbackangle{30}

\pgfmathsetlengthmacro{\pullbackradius}{\pullbackmarkerlength*sin(45)/sin(45-\pullbackangle)}
\newcommand{\pullback}[1]{
  \draw ({#1}.center) ++(360-\pullbackangle:\pullbackradius) -- +(0,-\pullbackmarkerlength);
  \draw ({#1}.center) ++(270+\pullbackangle:\pullbackradius) -- +(\pullbackmarkerlength,0);
}

\newcommand{\sq}[5]{{#1}\colon({#4} \; ^{{#2}}_{\substack{{#3}}} \; {#5})}

\DeclareMathOperator{\id}{id}
\DeclareMathOperator{\vid}{e}
\DeclareMathOperator{\op}{op}
\DeclareMathOperator{\ar}{Ar}

\newcommand{\Ar}{\ensuremath{\mathrm{Ar}_{*}}}
\newcommand{\evY}{\ensuremath{\mathcal{E}}}
\newcommand{\YC}{\ensuremath{\mathbf C^{\mathcal Y}}}
\newcommand{\Yoneda}{\ensuremath{\mathcal Y}}
\newcommand{\sY}{\ensuremath{\overline{\mathcal{Y}}}}

\newcommand{\slice}[2]{\ensuremath{{#1}\mathbin{\downarrow}{#2}}}

\newcommand{\pslice}[2]{\ensuremath{{#1}\mathbin{\downarrow}{#2}}}
\newcommand{\dblslice}[2]{\ensuremath{{#1}\mathbin{\downarrow\downarrow}{#2}}}

\newlength{\arrowlength}
\newcommand{\vcong}{\rotatebox{270}{$\cong$}}

\newcommand{\threecellarr}{\mathrel{\begin{tikzpicture}[baseline=(A.base)]
      \node(A)[inner sep=0,outer sep=0,minimum size=0] at (0,0) {\vphantom{a}};
      \draw[t](A) -- ++(\arrowlength,0);
    \end{tikzpicture}}}

\newcommand{\varrow}{\mathrel{\begin{tikzpicture}
    [baseline=(current bounding box.south)]
    \draw[proarrow] (0,0) -- (\arrowlength,0);
  \end{tikzpicture}}}

\newcommand{\vequal}{\mathrel{\begin{tikzpicture}
    [baseline=(current bounding box.south)]
    \draw[proequal] (0,0) -- (\arrowlength,0);
  \end{tikzpicture}}}

\newcommand{\iso}{\mathrel{\cong}}

\newcommand{\eqv}{\mathrel{\simeq}}

\newcommand{\xarr}[1]{
  \xrightarrow{\lower2pt\hbox{\smash{\ensuremath{\scriptscriptstyle #1}}\hspace{1pt}}}
}

\newcommand{\xnat}[1]{
  \xRightarrow{\hbox{\smash{\ensuremath{\scriptscriptstyle #1}}\hspace{2pt}}}
}

\newcommand{\ainl}[3]{\ensuremath{ {#2}\colon{#1}\rightarrow{#3} }}
\newcommand{\aeinl}[3]{\ensuremath{ {#2}\colon{#1}\xarr{\eqv}{#3} }}
\newcommand{\aiinl}[3]{\ensuremath{ {#2}\colon{#1}\xarr{\iso}{#3} }}

\newcommand{\ninl}[3]{\ensuremath{ {#2}\colon{#1}\Rightarrow{#3} }}
\newcommand{\niinl}[3]{\ensuremath{ {#2}\colon{#1}\xnat{\iso}{#3} }}
\newcommand{\neinl}[3]{\ensuremath{ {#2}\colon{#1}\xnat{\eqv}{#3} }}
\newcommand{\minl}[3]{\ensuremath{ {#2}\colon{#1}\threecellarr{#3} }}

\newcommand{\vainl}[3]{\ensuremath{{#2}\colon{#1}\varrow{#3} }}
\newcommand{\vidinl}[1]{{#1}\vequal{#1} }

\newcommand{\el}{\ensuremath{\mathbf{el}}}
\newcommand{\eell}{\ensuremath{\mathbbm{el}}}
\newcommand{\mor}{\ensuremath{\mathbf{mor}}}

\DeclareMathOperator{\Psd}{\mathbf {Ps}}
\DeclareMathOperator{\PPsd}{\mathbb {P}s}

\newcommand{\SV}{\ensuremath{\mathcal V}\xspace}

\usepackage{xspace}
\newcommand{\Cat}{\textnormal{\textsf{\textbf{Cat}}}\xspace}

\newcommand{\TwoCat}{\textnormal{\textsf{2Cat\textsubscript{nps}}}\xspace}

\newcommand{\DblCat}{\textnormal{\textsf{DblCat\textsubscript{h,nps}}}\xspace}

\newcommand{\onecat}[1]{\ensuremath{\mathcal{#1}}}

\makeatletter
\newcommand{\@bbify}[1]{
  \ifcsname b#1\endcsname
  \message{WARNING: Overwriting b#1 with blackboard letter!}
  \fi
  \expandafter\edef\csname b#1\endcsname
  {\noexpand\ensuremath{\noexpand\mathbb #1}\noexpand\xspace}}
\newcommand{\@calify}[1]{
  \ifcsname c#1\endcsname
  \message{WARNING: Overwriting c#1 with calligraphic letter!}
  \fi
  \expandafter\edef\csname c#1\endcsname
  {\noexpand\ensuremath{\noexpand\mathbf #1}\noexpand\xspace}}
\newcounter{@letter}\stepcounter{@letter}
\loop\@bbify{\Alph{@letter}}\@calify{\Alph{@letter}}
\ifnum\the@letter<26\stepcounter{@letter}\repeat
\makeatother

\title{Bi-initial objects and bi-representations are not so different}

\author[\lowercase{t.\ clingman}]{\lowercase{tslil clingman}}
\address{Johns Hopkins University, 3400 N. Charles St., Baltimore, MD, USA}
\email{tslil@jhu.edu}

\author[L.\ Moser]{Lyne Moser}
\address{Max Planck Institute for Mathematics, Vivatsgasse 7, 53111 Bonn, Germany}
\email{moser@mpim-bonn.mpg.de}


\begin{document}

\begin{abstract}
  We introduce a functor $\mathcal V\colon\textsf{DblCat\textsubscript{h,nps}}\to \textsf{2Cat\textsubscript{h,nps}}$ extracting from a double category a $2$-category whose objects and morphisms are the vertical morphisms and squares. We give a characterisation of bi-representations of a normal pseudo-functor $F\colon \mathbf C^{\operatorname{op}}\to \textsf{\textbf{Cat}}$ in terms of double bi-initial objects in the double category $\mathbbm{el}(F)$ of elements of $F$, or equivalently as bi-initial objects of a special form in the $2$-category~$\mathcal V\mathbbm{el}(F)$ of \emph{morphisms} of $F$. Although not true in general, in the special case where the $2$-category~$\mathbf C$ has tensors by the category $\mathbbm 2=\{0\to 1\}$ and $F$ preserves those tensors, we show that a bi-representation of $F$ is then precisely a bi-initial object in the $2$-category~$\mathbf{el}(F)$ of \emph{elements} of $F$. We give applications of this theory to bi-adjunctions and weighted bi-limits.
\end{abstract}

\maketitle

\section{Introduction}
\label{sec:introduction}

In ordinary category theory, properties of a categorical object are often formulated as questions of \emph{representability} of a \emph{presheaf}. By a presheaf, we mean a functor $\ainl {\onecat C^{\op}}F{\textsf{Set}}$, where $\onecat C$ is a category and $\textsf{Set}$ is the category of sets and functions. A representation of a presheaf comprises the data of an object $I\in \onecat C$ together with a natural isomorphism $\onecat C(-,I)\xnat{\iso} F$. In particular, this gives isomorphisms of sets $\onecat C(C,I)\iso FC$, for each object $C\in \onecat C$. A classical theorem, which we shall refer to as the ``Representation Theorem'', establishes that a presheaf $F$ has a representation precisely when its \emph{category of elements} has an initial object; see for example \cite[Proposition III.2.2]{MacLane} or \cite[Proposition 2.4.8]{Riehl}\footnote{Riehl defines the category of elements by hand along with a projection functor to $\onecat C$, rather than $\onecat C^{\op}$; therefore, representations correspond to terminal objects in this setting.}. This category of elements of $F$ is defined as the slice category $\slice{\{*\}}F$, where $\{*\}$ denotes the singleton set.

Two main examples of properties that can be rephrased in terms of representability are the existence of limits and adjoints for a functor. Indeed, asking that a functor $\ainl {\onecat I}F{\onecat C}$ admits a limit amounts to asking whether the presheaf $\ainl{\onecat C^{\op}}{[\onecat I,\onecat C](\Delta (-),F)}{\textsf{Set}}$ has a representation. Therefore, by the Representation Theorem, this is equivalent to requiring the presence of a terminal object in the slice category $\slice \Delta F$ of cones over~$F$. Similarly, the existence of a right adjoint to a functor $\ainl {\onecat C}L{\onecat D}$ may equivalently be reformulated as the existence of a representation of the presheaf $\ainl {\onecat C^{\op}}{\onecat C(L-,D)}{\textsf{Set}}$, for each object~$D\in \onecat D$, or equivalently of the existence of a terminal object in the slice category $\slice L D$, for each object $D\in \onecat D$.

In passing from ordinary categories to $2$-categories, we may seek to elevate discussions of representations of ordinary presheaves to their $2$-dimensional counter-parts. By a $2$-dimensional presheaf, we mean a normal pseudo-functor $\ainl {\cC^{\op}}F{\Cat}$, where $\cC$ is a $2$-category and $\Cat$ is the $2$-category of categories, functors, and natural transformations. The data of a $2$-dimensional representation is once more an object $I\in\cC$, but when it comes to comparing the \emph{categories} $\cC(C,I)$ and $FC$ we may retain the idea that this comparison is mediated by an isomorphism of categories, or we may require only the presence of an equivalence of categories. The former choice leads to the notion of a \emph{$2$-representation}, while the latter leads to the more general notion of a \emph{bi-representation}.

Recall that an object $I$ in a category $\onecat C$ is initial if we have an isomorphism of sets $\onecat C(I,C)\iso\{*\}$ for all objects $C\in\onecat C$. If we wish to formulate the $2$-dimensional definition analogously for an object $I$ in a $2$-category $\cC$, as before we now have the option of retaining the idea that the universal property should be governed by an isomorphism of categories $\cC(I,C)\iso\mathbbm 1$, for all objects $C\in \cC$, where $\mathbbm 1$ is the terminal category, or instead asking that the universal property is governed by an equivalence of categories. The former requirement leads to the notion of a \emph{$2$-initial object}, while the latter leads to the more general notion of a \emph{bi-initial object}.

Now that the players are ready the game is afoot. The question underpinning the most general $2$-dimensional version of the Representation Theorem is this:
\begin{question}\hypertarget{question}
Can bi-representations of a normal pseudo-functor $F$ be characterised as \emph{certain} bi-initial objects in \emph{some} $2$-category?
\end{question}\newcommand{\quest}{\hyperlink{question}{question}}
As a first guess, based on the Representation Theorem, we might expect that the \emph{$2$-category of elements $\el(F)$ of $F$} would be the correct setting for an affirmative answer. The $2$-category $\el(F)$ is defined as the \emph{pseudo-slice} $2$-category $\slice {\mathbbm 1} F$, where by pseudo-slice we mean a relaxation of the slice $2$-category where the triangles of morphisms commute up to a general $2$-isomorphism, rather than an identity.

Although we hate to disappoint the reader, to see that this is not the case we will turn our interest to a specific kind of bi-representation. Generalising ordinary limits, bi-representations of the $2$-presheaf $[\cI,\cC](\Delta(-),F)$, known as \emph{bi-limits}, were first introduced by Street in \cite{Street1980,Street1980Corr} and further studied by Kelly in \cite{Kelly1989}. The comparatively stronger special case -- $2$-representations of the above $2$-presheaf, known as \emph{$2$-limits} -- had previously been introduced, independently, by Auderset \cite{Auderset} and Borceux-Kelly \cite{BorKel}, and was further developed by Street~\cite{Street1976}, Kelly \cite{Kelly1989,KellyBook} and Lack in \cite{Lack2010}. As $\el([\cI,\cC](\Delta(-),F))$ is the opposite of the pseudo-slice $2$-category $\slice \Delta F$ of cones over $F$, the \quest{} now becomes whether bi-limits may be characterised as bi-terminal objects in the pseudo-slice $2$-category of cones.

Unfortunately, as the authors show in \cite{us}, such a characterisation is not possible in general. The failure stems from the fact that the data of a bi-limit is not wholly captured by a bi-terminal object in the pseudo-slice $2$-category of cones (see {\cite[\S 5]{us}}). In fact, such a failure is actually illustrated by an example of a $2$-terminal object in a slice $2$-category of cones which is not a $2$-limit; see \cite[Counter-example 2.12]{us}.

We have thus eliminated our first guess from the possible affirmative answers to our \quest{} above. To cast further doubt on any positive resolution, correct characterisations of $2$-limits as some form of $2$-dimensional terminal objects that are present in the literature are all phrased in the language of \emph{double categories}; such results are explored by Grandis \cite{Grandis}, Grandis-Par\'e \cite{GraPar1999,GraPar2019}, and Verity \cite{Verity}. These results may all be seen to share the following approach: the slice $2$-category of cones does not capture enough data to successfully characterise $2$-limits, and so instead more data must be necessarily added in the form of a slice double category of cones. Indeed, in \cite{GraPar2019} Grandis and Par\'e write:
\begin{quote}
        ``On the other hand, there seems to be no natural way of expressing the $2$-dimensional universal property of weighted (strict or pseudo) limits by terminality in a $2$-category.''
\end{quote}

The state of the art thus seems to suggest that our \quest{} admits no positive answer in general. However, our main contribution in response to Grandis and Par\'e above, is a successful and purely $2$-categorical characterisation of $2$-limits as certain $2$-terminal objects in a ``shifted'' slice $2$-category of cones. In fact, we obtain this result as an application of a purely $2$-categorical formulation of a generalisation of the Representation Theorem in the case of bi-representations. To do this, we extend the results of Grandis, Par\'e, and Verity to general bi-representations of normal pseudo-functors, and obtain in this fashion a double-categorical characterisation of bi-representations in terms of ``bi-type'' double-initial objects. From this work and some new methods we are able to extract our results.

Let us explore these results in greater detail. Recall that a double category has two sorts of morphisms between objects -- the \emph{horizontal} and \emph{vertical} morphisms -- and $2$-dimensional morphisms called \emph{squares}. So as to distinguish double categories from $2$-categories, we will always be careful to name the former by double-struck letters \bA, \bB, \bC, \ldots whereas the latter will always appear as named by bold letters \cA, \cB, \cC, \ldots

A $2$-category $\cA$ can always be seen as a horizontal double category $\bH\cA$ with only trivial vertical morphisms. This construction extends functorially to an assignment on normal pseudo-functors $F\mapsto \bH F$. Associated to each such normal pseudo-functor is a \emph{double category of elements} $\eell(F)$ given by the pseudo-slice double category $\dblslice {\mathbbm 1}{\bH F}$, where these pseudo-slices are double-categorical analogues of pseudo-slice $2$-categories. Furthermore, we introduce a new notion of \emph{double bi-initial} objects $I$ in a double category $\bA$; objects $I\in \bA$ for which the projection $\dblslice I\bA\to \bA$ is given by an appropriate equivalence of double categories. Note that, in the case of \emph{double-initial} objects as defined by Grandis and Par\'e in \cite[\S 1.8]{GraPar1999}, the projection is required to be only an isomorphism. With these notions in hand, we are able to formulate the following double categorical characterisation of bi-representations. This appears as the first part of our main theorem, \cref{thm:birep}.

\begin{ltheorem}\label{theoremA}
Let $\cC$ be a $2$-category, and $\ainl {\cC^{\op}} F\Cat$ be a normal pseudo-functor. The following statements are equivalent.
\begin{rome}
  \item The normal pseudo-functor $F$ has a bi-representation $(I,\rho)$.
  \item There is an object $I\in \cC$ together with an object $i\in FI$ such that $(I,i)$ is double bi-initial in~$\eell(F)$.
\end{rome}
\end{ltheorem}

By applying this result to the $2$-presheaf $[\cI,\cC](\Delta(-),F)$, where $F\colon \cI\to \cC$ is a normal pseudo-functor, we derive a generalisation in \cref{thm:conicalbilim} of results by Grandis, Par\'e, and Verity, characterising bi-limits as double bi-terminal objects in the pseudo-slice double category $\dblslice \Delta F$. This follows from the fact that the double category $\eell([\cI,\cC](\Delta(-),F))$ is isomorphic to the horizontal opposite of the pseudo-slice double category $\dblslice \Delta F$.

We now aim to extract a fully $2$-categorical statement from \cref{theoremA} above. For this, it is enough to characterise double bi-initial objects in a double category $\bA$ as \emph{certain} bi-initial objects in \emph{some} $2$-category. A first guess for the $2$-category in question would be given by the underlying horizontal $2$-category $\cH\bA$ of objects, horizontal morphisms, and squares with trivial vertical boundaries of $\bA$. However, the general vertical structure of the double category \bA is not captured by this operation, and therefore the $2$-category $\cH\bA$ alone does not suffice for our purposes. To remedy this issue, we introduce a functor \SV which extracts from a double category $\bA$ a $2$-category $\SV\bA$ whose objects and morphisms are the vertical morphisms and squares of $\bA$, respectively. This captures precisely the additional data that was lacking in $\cH\bA$ in our application, and allows us to prove the below result. In fact, the functor $\cH$ is as a retract of $\SV$ and so we may leverage $\SV$ alone to characterise double bi-initial objects. The below appears as \cref{thm:initial} in the paper.

\begin{ltheorem} \label{theoremB}
Let $\bA$ be a double category, and $I\in \bA$ be an object. The following statements are equivalent.
\begin{rome}
  \item The object $I$ is double bi-initial in $\bA$.
  \item The object $I$ is bi-initial in $\cH\bA$ and the vertical identity $\vid_I$ is bi-initial in $\SV\bA$.
  \item The vertical identity $\vid_I$ is bi-initial in $\SV\bA$.
\end{rome}
\end{ltheorem}

As a direct application of this result to the double category of elements $\eell(F)$ of a normal pseudo-functor $\ainl {\cC^{\op}} F\Cat$, we obtain our fully $2$-categorical characterisation of bi-representations. Note that the underlying horizontal $2$-category of $\eell(F)$ is precisely the $2$-category of elements $\el(F)$ of $F$, but new here is $\SV\eell(F)$ which we refer to as the \emph{$2$-category of morphisms of $F$}, denoted by $\mor(F)$. Indeed, while the objects in $\el(F)$ are pairs $(C,x)$ of an object $C\in \cC$ and an object $x\in FC$, the objects of $\mor(F)$ are pairs $(C,\alpha)$ of an object $C\in \cC$ and a morphism $\ainl x\alpha y$ in $FC$, which justifies the terminology. The following result extends \cref{theoremA} and appears as the second part of our main theorem, \cref{thm:birep}.

\begin{ltheorem} \label{theoremC}
Let $\cC$ be a $2$-category, and $\ainl {\cC^{\op}} F\Cat$ be a normal pseudo-functor. The following statements are equivalent.
\begin{rome}
  \item The normal pseudo-functor $F$ has a bi-representation $(I,\rho)$.
  \item There is an object $I\in \cC$ together with an object $i\in FI$ such that $(I,i)$ is bi-initial in~$\el(F)$ and $(I,\id_i)$ is bi-initial in $\mor(F)$.
  \item There is an object $I\in \cC$ together with an object $i\in FI$ such that $(I,\id_i)$ is bi-initial in $\mor(F)$.
\end{rome}
\end{ltheorem}

The equivalence of (i) and (iii) in the above theorem gives a satisfying answer to our original \quest{}. In particular, to respond Grandis and Par\'e, we specialise the above theorem to the case of bi-limits to see that bi-limits are equivalently certain types of bi-terminal objects in a $2$-category whose objects are given by the \emph{morphisms} of cones -- known as modifications -- as we will see in \cref{thm:conicalbilim}. Thus the counter-examples of \cite{us} for bi-limits show the presence of $\mor(F)$ in (ii) is necessary in general.

Although the correct characterisation of bi-limits in a $2$-category \cC above depends on taking morphisms of cones as objects, in the presence of \emph{tensors} in \cC these can be simply seen as cones whose summit is a tensor by the category $\mathbbm 2=\{0\to 1\}$. In the case of $2$-limits, Kelly observed in \cite[\S 3]{Kelly1989} that the presence of tensors by $\mathbbm 2$ causes the $1$-dimensional aspect of the universal property of a $2$-limit to imply the $2$-dimensional aspect. As a consequence, we showed in \cite[Proposition 2.13]{us} that a $2$-limit is precisely a $2$-terminal object in the slice $2$-category of cones under such a hypothesis.

This result is part of a far more general framework and we shall approach this in parts. A double categorical analogue of tensors by $\mathbbm 2$ is given by the notion of \emph{tabulators}\footnote{In fact these notions are somehow dual, but our applications all involve the horizontal double category associated to the opposite of a $2$-category and so tabulators there coincide with tensors by $\mathbbm 2$.} of vertical morphisms; these are defined by Grandis and Par\'e in \cite[\S 5.3]{GraPar1999} as double limits of vertical morphisms seen as double functors. By \cref{theoremA}, bi-representations correspond to double bi-initial objects in a certain double category, and it is at this level that we seek a simplification of \cref{theoremB} in the presence of tabulators. This is the content of \cref{prop:tabulators-biinitial}.

\begin{ltheorem} \label{theoremD}
Let $\bA$ be a double category with tabulators, and $I\in \bA$ be an object. Then the following statements are equivalent.
\begin{rome}
  \item The object $I$ is double bi-initial in $\bA$.
  \item The object $I$ is bi-initial in $\cH\bA$.
\end{rome}
\end{ltheorem}

We now aim to simplify the characterisation of bi-representations given in \cref{theoremC} when the $2$-category $\cC$ has tensors by $\mathbbm 2$, which can be seen as tabulators in the double category $\bH\cC^{\op}$. For this simplification, we further need the normal pseudo-functor $\ainl{\cC^{\op}}F{\Cat}$ to preserve these tensors so that overall the double category of elements $\eell(F)$ admits tabulators. Although we were not able to use the $2$-category of elements $\el(F)$ to give an answer to our \quest{} in general, in this special case we may apply \cref{theoremD} to recover the following verbatim translation of the Representation Theorem to the $2$-categorical setting, which appears as \cref{thm:eeellll-tensors}.

\begin{ltheorem}\label{theoremE}
  Let $\cC$ be a $2$-category with tensors by $\mathbbm 2$, and $\ainl{\cC^{\op}}F{\Cat}$ be a normal pseudo-functor which preserves these tensors. The following statements are equivalent.
  \begin{rome}
    \item The normal pseudo-functor $F$ has a  bi-representation $(I,\rho)$.
    \item There is an object $I\in\cC$ together with an object $i\in FI$ such that $(I,i)$ is bi-initial in $\el(F)$.
  \end{rome}
\end{ltheorem}

This applies to the case of bi-limits, and we formulate in \cref{cor:bilimtensor} a more general version of \cite[Proposition 2.13]{us}: a bi-limit is precisely a bi-terminal object in the pseudo-slice $2$-category of cones when the ambient $2$-category admits tensors by $\mathbbm 2$. This application provides the promised proof of \cite[Proposition 5.5]{us}.

While we have only mentioned the case of bi-limits so far, in this paper the different theorems characterising bi-representations are first specialised to the case of \emph{weighted bi-limits}, which were introduced by Street \cite{Street1976} and Kelly \cite{Kelly1989}. The cone of a weighted limit is of a special shape, determined by the \emph{weight} -- a normal pseudo-functor $W$ taking values in $\Cat$ -- and a bi-limit can be seen as a weighted limit with \emph{conical} weight $W=\Delta \mathbbm 1$, i.e., a constant weight at the terminal category. More still, when the weight is conical the pseudo-slice of cones is isomorphic to the opposite of the pseudo-slice of \emph{weighted cones}. Since weighted bi-limits can also be seen as bi-representations of a normal pseudo-functor of a special kind, we also obtain characterisations in \cref{thm:weighted,thm:tensors-weighted} of weighted bi-limits in terms of double bi-initial and bi-initial objects. From these we extract the characterisations of bi-limits in terms of double bi-terminal and bi-terminal objects mentioned above.

Another application of the Representation Theorem is to the existence of a right adjoint to a given functor. Going one dimension up, we can define an analogous notion of \emph{bi-adjunction} between $2$-categories \cC and \cD. A bi-adjunction comprises the data of a pair of normal pseudo-functors $\ainl \cC L \cD$ and $\ainl \cD R \cC$ together with a pseudo-natural equivalence $\cD(L-,-)\xnat{\eqv} \cC(-,R-)$. In order to apply our main results to the existence of a right bi-adjoint to a given normal pseudo-functor, there is first the delicate matter of reformulating such a question in terms of bi-representations.

\cref{lem:yonedastyle} states that a normal pseudo-functor $\ainl \cC L\cD$ has a right bi-adjoint if and only if there is a bi-representation of the normal pseudo-functor $\cD(L-,D)$ for each object $D\in \cD$. This shows that the pseudo-naturality of $\cD(L-,-)\xnat{\eqv} \cC(-,R-)$ in one of the variables is superfluous data, and may always be recovered from merely object-wise information -- in analogy with the corresponding result for ordinary adjunctions and representations. Although this result about bi-adjunctions is known and expected, we were unable to find even a statement of this theorem in the literature. Capitalising on this gap we provide a proof in \cref{sec:bi-adjunctions} using some cool $2$-dimensional Yoneda tricks rather than a direct construction.

This formulation of the existence of a right bi-adjoint is then amenable to our theorems about bi-representations above and we prove in \cref{thm:biadjoint} that $L$ has a right bi-adjoint if and only if there is a double bi-initial object in the pseudo-slice double category $\dblslice{L}{D}$ for each object $D\in\cD$. As before we derive a purely $2$-categorical statement by applying the functors $\cH$ and $\SV$ to $\dblslice L D$. The resulting $2$-categories are isomorphic to the pseudo-slice $2$-category $\slice L D$ and a ``shifted'' pseudo-slice $2$-category $\slice {\Ar L} D$, whose objects are $2$-morphisms between $LC$ and $D$. Finally bi-adjunctions also benefit from the presence of tensors and we prove in \cref{thm:biadjtensor} that, if the $2$-category $\cC$ has tensors by $\mathbbm 2$ which are preserved by $L$\footnote{Note that tensors by $\mathbbm 2$ are weighted colimits and that left bi-adjoints preserve those (LBAPWBC).}, then $L$ has a right bi-adjoint if and only if there is a  bi-initial object in the pseudo-slice $2$-category $\slice L D$ for each object $D\in\cD$. This special case gives a straightforward $2$-categorical version of the characterisation of the existence of a right adjoint to an ordinary functor.

As we saw throughout the introduction, there are also $2$-type versions of the bi-type notions considered here. All of the theorems given in this paper may also be proven in this stronger setting; the proofs are predictably less involved as there are less coherence conditions to check here. For example, \cref{theoremA} in this stronger setting can be formulated as follows: there is a  $2$-representation of a normal pseudo-functor $\ainl {\cC^{\op}}F\Cat$ if and only if there is a double-initial object in the double category of elements of $F$, defined here as the \emph{strict} slice double category $\dblslice{\mathbbm 1}\bH F$. Similarly, \cref{theoremC} for the $2$-type case would concern $2$-initial objects and stricter versions of $\el(F)$ and $\mor(F)$.

\subsection{Outline}

The paper is organised as follows. \cref{sec:background,sec:equivalences} present the setting of $2$-categories and double categories in which we will be working and the functors relating these two settings. We recall notions of trivial fibrations of $2$-categories and double categories, which we later use in the definition of bi-initiality.

Then \cref{sec:pseudo-comma,sec:initial} introduce pseudo-comma $2$-categories and double categories, as well as (double) bi-initial objects. We then compare all these $2$-categorical notions to their double categorical analogues. After establishing these comparisons, the remaining work to prove the main result is to characterise bi-representations in terms of double bi-initial objects, which we do in \cref{sec:bi-rep-pseudo}. Finally \cref{sec:applications} studies applications of our main theorem to bi-adjunctions and (weighted) bi-limits.

\subsection{Acknowledgements}

Both authors are grateful to anonymous referees for pointing out issues in previous versions. This work began when both authors were at the Mathematical Sciences Research Institute in Berkeley, California, during the Spring 2020 semester. The first-named author benefited from support by the National Science Foundation under Grant No. DMS-1440140, while at residence in MSRI. The second-named author was supported by the Swiss National Science Foundation under the project P1ELP2\_188039 and the Max Planck Institute of Mathematics. The first-named author was additionally supported by the National Science Foundation grant DMS-1652600, as well as the JHU Catalyst Grant.

\section{Background on \texorpdfstring{$2$}{2}-categories and double categories}
\label{sec:background}

To state and prove our main result, \Cref{thm:birep} below, we will make use of the languages of $2$-categories and of double categories. In particular we will employ the notions of normal pseudo-functor, pseudo-natural transformation, modification, as well as horizontal double categorical counterparts to these notions -- double functors and horizontal natural transformations which exhibit pseudo-type behaviour in the horizontal direction. To cement terminology and familiarise ourselves with these notions we will briefly recall the $2$-categorical and double categorical concepts at issue in \Cref{sec:2-categories,sec:double-categories} below. Readers comfortable with these definitions should skip ahead to \Cref{sec:equivalences}.

\subsection{\texorpdfstring{$2$}{2}-categories}
\label{sec:2-categories}

Recall that a $2$-category is a category enriched in categories. It comprises the data of objects and hom-categories between each pair of objects, together with a \emph{horizontal} composition operation. The objects of the hom-categories are called \emph{morphisms}, the morphisms therein are called \emph{$2$-morphisms}, and the composition operation therein is called \emph{vertical} composition of $2$-morphisms.

Morphisms between $2$-categories which preserve all the $2$-categorical structure strictly are called $2$-functors. However, in this paper, we consider the more general notion of morphisms of $2$-categories, namely \emph{normal pseudo-functors}.

\begin{definition} \label{defn:pseudo}
Given $2$-categories $\cA$ and $\cB$, a \textbf{pseudo-functor} $\ainl \cA{(F,\phi)}\cB$ comprises the data of
\begin{rome}
\item an assignment on objects $A \in\cA \mapsto FA\in\cB$,
\item functors $\ainl{\cA(A,A')}{F}{\cB(FA,FA')}$ for each pair of objects $A,A'\in\cA$,
\item $2$-isomorphisms $\niinl{\id_{FA}}{\phi_{A}}{F\id_{A}}$ in $\cB$ for each object $A\in\cA$, called \emph{unitors},
\item $2$-isomorphisms $\ninl{(Fa')(Fa)}{\phi_{a,a'}}{F(a'a)}$ in \cB for each pair of composable morphisms $\ainl Aa{A'}$ and $\ainl{A'}{a'}{A''}$ in \cA, called \emph{compositors},
\end{rome}
such that these data satisfy naturality, associativity, and unitality conditions. For details see, for example, \cite[Definition 4.1.2]{JohYau}.

If, for all $A\in \cA$, the unitor $\phi_A$ is given by the identity $2$-morphism $\id_{\id_{FA}}$, we say that the pseudo-functor $(F,\phi)$ is \textbf{normal}.
\end{definition}

As every pseudo-functor is appropriately isomorphic to a normal one by \cite[Proposition 5.2]{LackPauli}, we choose to simplify our arguments by forgoing the extra data and coherence associated to the former class by working only with normal pseudo-functors.

\begin{notation}
    We denote by $\TwoCat$ the category of $2$-categories and normal pseudo-functors.
\end{notation}

We now define a $2$-category whose objects are the normal pseudo-functors. For this, we first define its morphisms and $2$-morphisms.

\begin{definition}\label{def:pseudo-natural-transformation}
Given pseudo-functors $\ainl{\cA}{F,G}{\cB}$, a \textbf{pseudo-natural transformation} $\ninl{F}{\alpha}{G}$ comprises the data of
\begin{rome}
  \item morphisms $\ainl{FA}{\alpha_A}{GA}$ in \cB for each object $A\in\cA$,
  \item $2$-isomorphisms $\niinl{(Ga)\alpha_A}{\alpha_{a}}{\alpha_{A'}(Fa)}$ in \cB for each morphism $\ainl Aa{A'}$ in \cA as depicted below.
  \begin{diagram*}
    \node(1)[]{$FA$};
    \node(2)[right of= 1,xshift=.2cm]{$GA$};
    \node(3)[below of= 1]{$FA'$};
    \node(4)[below of= 2]{$GA'$};
    \draw[a](1) to node[la]{$\alpha_A$}(2);
    \draw[a](3) to node[la,swap]{$\alpha_{A'}$}(4);
    \draw[a](1) to node[la,swap]{$Fa$}(3);
    \draw[a](2) to node[la]{$Ga$}(4);
    \celli[la]{2}{3}{$\alpha_a$};
  \end{diagram*}
\end{rome}
such that the $2$-morphisms $\alpha_a$ above are natural with respect to $2$-morphisms in \cA, and compatible with the compositors and unitors of $F$ and $G$. For details see, for example, \cite[Definition 4.2.1]{JohYau}.

If, for all morphisms $\ainl Aa{A'}$ in $\cA$, the $2$-isomorphism component $\alpha_a$ is an identity, i.e., $(Ga)\alpha_A=\alpha_{A'}(Fa)$, then we say that $\alpha$ is \textbf{$2$-natural}.
\end{definition}

\begin{remark}
  As a consequence of the compatibility with the unitors above, a pseudo-natural transformation $\ninl F\alpha G$ for which $F$ and $G$ are both normal automatically satisfies $\alpha_{\id_A}=\id_{\alpha_A}$ for all $A\in \cA$.
\end{remark}

\begin{definition} \label{def:modification}
Given pseudo-natural transformations $\ninl F{\alpha,\beta}G$, a \textbf{modification} $\minl\alpha\Gamma\beta$ comprises the data of $2$-morphisms $\ninl{\alpha_A}{\Gamma_A}{\beta_A}$ in \cB for each object $A\in \cA$ which are compatible with the $2$-isomorphism components of $\alpha$ and $\beta$. For details see, for example, \cite[Definition 4.4.1]{JohYau}.
\end{definition}

\begin{definition}\label{def:inthom2Cat}
  Let $\cA$ and $\cB$ be $2$-categories. We define the $2$-category $\Psd(\cA,\cB)$ whose objects are normal pseudo-functors from $\cA$ to $\cB$, morphisms are pseudo-natural transformations, and $2$-morphisms are modifications.
\end{definition}

\subsection{Double categories}
\label{sec:double-categories}

In addition to the $2$-dimensional concepts above, we will make much use of the possibly less familiar notions of double categories and their morphisms. To prepare for this, we invite the reader to join us in recalling some of the early definitions.

\begin{definition}
A \textbf{double category} $\bA$ comprises the data of
\begin{rome}
\item objects $A,A',B,B',\ldots$,
\item horizontal morphisms $\ainl Aa{A'}$,
\item vertical morphisms $\vainl AuB$,
\item squares $\alpha$ with both horizontal and vertical sources and targets, written inline as $\sq{\alpha}{a}{b}{u}{u'}$ or drawn as
  \begin{diagram*}[,]
    \node[](1) {$A$};
    \node[below of=1](2) {$B$};
    \node[right of=1](3) {$A'$};
    \node[below of=3](4) {$B'$};
    \draw[proarrow] (1) to node[swap,la] {$u$} (2);
    \draw[proarrow] (3) to node[la] {$u'$} (4);
    \draw[a] (1) to node[la] {$a$} (3);
    \draw[a] (2) to node[swap,la] {$b$} (4);
    \node[la] at ($(3)!0.5!(2)$) {$\alpha$};
 \end{diagram*}
\item horizontal and vertical identity morphisms for each object $A$, written $\id_A\colon A=A$ and $\vid_A\colon\vidinl A$ respectively,
\item a horizontal identity square for each vertical morphism $u$ and a vertical identity square for each horizontal morphism $f$, written respectively as
  \begin{diagram*}
    \node[](1) {$A$};
    \node[below of=1](2) {$B$};
    \node[right of=1](3) {$A$};
    \node[below of=3](4) {$B$};
    \draw[proarrow] (1) to node[swap,la] {$u$} (2);
    \draw[proarrow] (3) to node[la] {$u$} (4);
    \draw[d] (1) to node[la] {$\id_A$} (3);
    \draw[d] (2) to node[swap,la] {$\id_B$} (4);
    \node[la] at ($(3)!0.5!(2)$) {$\id_u$};

    \node[right of=3,xshift=2cm](1) {$A$};
    \node[below of=1](2) {$A$};
    \node[right of=1](3) {$A'$};
    \node[below of=3](4) {$A'$};
    \draw[proequal] (1) to node[swap,la] {$\vid_A$} (2);
    \draw[proequal] (3) to node[la] {$\vid_{A'}$} (4);
    \draw[a] (1) to node[la] {$a$} (3);
    \draw[a] (2) to node[swap,la] {$a$} (4);
    \node[la] at ($(3)!0.5!(2)$) {$\vid_a$};
 \end{diagram*}
\item a horizontal composition operation on horizontal morphisms and squares along a shared vertical boundary,
\item a vertical composition operation on vertical morphisms and squares along a shared horizontal boundary,
\end{rome}
such that the composition operations are all appropriately associative and unital, and such that horizontal and vertical composition of squares obeys the interchange law. We direct the reader to \cite[Definition 3.1.1]{Grandis} for details.
\end{definition}

\begin{remark}
  Note that a $2$-category \cA can be seen as a \emph{horizontal} double category $\bH\cA$, with only trivial vertical morphisms; see \cref{defn:bH}. Dually, we can also see a $2$-category~\cA as a \emph{vertical} double category $\bV\cA$.

  The horizontal embedding is preferred in this document, as it corresponds to the inclusion of $2$-categories seen as internal categories to \textsf{Cat} whose category of objects is discrete, into general internal categories to \textsf{Cat}, which are precisely the double categories. This inclusion itself agrees with the inclusion $\textsf{Cat}=\textsf{IntCat}(\textsf{Set})\rightarrowtail \textsf{IntCat}(\textsf{Cat})=\textsf{DblCat}$ arising from $\textsf{Set}\rightarrowtail \textsf{Cat}$, when categories are seen as $2$-categories with only trivial $2$-morphisms.
\end{remark}

Much as in the case of $2$-categories above, we will be interested not in (strict) double functors, which preserve the double categorical structure strictly, but in certain pseudo-type ones. As we choose here to see $2$-categories as horizontal double categories, in order to extend this assignment on objects to a functor from $\TwoCat$, we need to require that our pseudo-double functors are pseudo in the horizontal direction.

\begin{definition} \label{defn:pseudodouble}
Given double categories \bA and \bB, a \textbf{(horizontally) pseudo-double functor} $\ainl\bA {(F,\phi)}\bB$ comprises the data of
\begin{drome}
  \item assignments sending respectively objects $A$, horizontal morphisms $\ainl Aa{A'}$, vertical morphisms $\vainl AuB$, and squares $\sq{\alpha}{a}{b}{u}{u'}$ in \bA to objects $FA$, horizontal morphisms $\ainl{FA}{Fa}{FA'}$, vertical morphisms $\vainl{FA}{Fu}{FB}$, and squares $\sq{F\alpha}{Fa}{Fb}{Fu}{Fu'}$ in \bB,
  \item\label{defn:pseudo-double-unit-squares} for each $A\in\bA$, a vertically invertible unitor square of \bB of the form
  \begin{diagram*}[,][2]
    \node[](1) {$FA$};
    \node[below of=1](2) {$FA$};
    \node[right of=1](3) {$FA$};
    \node[below of=3](4) {$FA$};
    \draw[proequal] (1) to (2);
    \draw[proequal] (3) to (4);
    \draw[d] (1) to node[la]{$\id_{FA}$}(3);
    \draw[a] (2) to node [la,swap]{$F\id_A$} (4);
    \node[la,xshift=-6pt] at ($(1)!0.5!(4)$) {$\phi_A$};
    \node[la,xshift=6pt] at ($(1)!0.5!(4)$) {$\vcong$};
\end{diagram*}
\item for each pair of composable horizontal morphisms $\ainl Aa{A'}$ and $\ainl {A'}{a'}{A''}$ in~\bA, a vertically invertible compositor square of \bB of the form
  \begin{diagram*}[,][2]
    \node[](1) {$FA$};
    \node[right of=1](5) {$FA'$};
    \node[below of=1](2) {$FA$};
    \node[right of=5](3) {$FA''$};
    \node[below of=3](4) {$FA''$};
    \draw[proequal] (1) to (2);
    \draw[proequal] (3) to (4);
    \draw[a] (1) to node[la]{$Fa$}(5);
    \draw[a] (5) to node[la]{$Fa'$}(3);
    \draw[a] (2) to node [la,swap]{$F(a'a)$} (4);
    \node[la,xshift=-8pt] at ($(1)!0.5!(4)$) {$\phi_{a,a'}$};
    \node[la,xshift=8pt] at ($(1)!0.5!(4)$) {$\vcong$};
\end{diagram*}
\end{drome}
such that
\begin{enumerate}
\item vertical compositions of vertical morphisms and squares, as well as vertical identities, are preserved strictly,
\item the unitor squares are natural with respect to vertical morphisms of \bA,
\item the compositor squares are natural with respect to vertical composition by squares of \bA,
\item the compositor squares are associative and unital with respect to the unitor squares
\end{enumerate}

If, for all $A\in \bA$, the unitor square $\phi_A$ is given by the vertical identity square $\vid_{\id_{FA}}$, we say that the pseudo-double functor $(F,\phi)$ is \textbf{normal}.
\end{definition}

\begin{notation}
We denote by $\DblCat$ the category of double categories and normal pseudo-double functors.
\end{notation}

We direct the reader to \cite[Definition 3.5.1]{Grandis} for a full elaboration of these conditions for lax-type double functors -- though we have interchanged the vertical and horizontal directions by comparison.

We also construct a double category whose objects are the normal pseudo-double functors. We now define its horizontal morphisms which are \emph{horizontal pseudo-natural transformations}. Its vertical morphisms and squares are the \emph{vertical natural transformations} and \emph{modifications}, but we will not have use of the details of these notions, and refer the curious reader to \cite[Definitions 3.8.1 and 3.8.3]{Grandis} for these. As ever, we caution the reader that our horizontal and vertical directions are interchanged.

\begin{definition}\label{def:horizontal-pseudo-natural-transformation}
Given pseudo-double functors $\ainl\bA{F,G}\bB$, a \textbf{horizontal pseudo-natural transformation} $\ninl F\alpha {G}$ comprises the data of
\begin{rome}
\item horizontal morphisms $\ainl {FA}{\alpha_{A}}{GA}$ for each $A\in\bA$,
\item squares $\sq{\alpha_{u}}{\alpha_{A}}{\alpha_{B}}{Fu}{Gu}$ for each vertical morphism $\vainl AuB$ of \bA,
\item vertically invertible squares
  \begin{diagram*}
    \node(1)[]{$FA$};
    \node(2)[right of= 1]{$GA$};
    \node(3)[right of= 2]{$GA'$};
    \node(4)[below of= 1]{$FA$};
    \node(5)[below of= 2]{$FA'$};
    \node(6)[below of= 3]{$GA'$};
    \draw[proequal](1)to(4);
    \draw[proequal](3)to(6);
    \draw[a](1)to node[la]{$\alpha_{A}$}(2);
    \draw[a](2)to node[la]{$Ga$}(3);
    \draw[a](4)to node[la,swap]{$Fa$}(5);
    \draw[a](5)to node[la,swap]{$\alpha_{A'}$}(6);
    \node[la,xshift=-6pt] at ($(3)!0.5!(4)$) {$\alpha_{a}$};
    \node[la,xshift=6pt] at ($(3)!0.5!(4)$) {$\vcong$};
  \end{diagram*}
  for each horizontal morphism $\ainl Aa{A'}$,
\end{rome}
such that the squares $\alpha_{u}$ are coherent with respect to vertical composition and identities, and together with the squares $\alpha_{a}$ and the compositors and unitors of $F$ and $G$ satisfy horizontal conditions of naturality and unitality. For a full expansion of these conditions, see \cite[Definition 3.8.2]{Grandis} -- though note again that our horizontal and vertical directions have been interchanged.

If, for all horizontal morphisms $a$, the square component $\alpha_a$ is an identity, we call $\alpha$ a \textbf{horizontal natural transformation}.
\end{definition}

\begin{remark}
  As a consequence of the axioms, a horizontal pseudo-natural transformation $\ninl F\alpha G$ for which $F$ and $G$ are normal is such that the vertically invertible square $\alpha_{\id_A}$ is given by the vertical identity square $\vid_{\alpha_A}$ for all $A\in \bA$.
\end{remark}

\begin{definition} \label{def:inthomDbl}
  Let $\bA$ and $\bB$ be double categories. We define the double category $\PPsd(\bA,\bB)$ in $\DblCat$ whose objects are normal pseudo-double functors from $\bA$ to $\bB$, horizontal morphisms are horizontal pseudo-natural transformations, vertical morphisms are vertical (strict-)natural transformations, and squares are modifications.
\end{definition}

\section{The functor \texorpdfstring{$\SV$}{V} and trivial fibrations}
\label{sec:equivalences}

A $2$-category $\cA$ can be seen as a horizontal double category $\bH\cA$ with only trivial vertical morphisms. This construction has a right adjoint, which extracts from a double category~$\bA$ its underlying horizontal $2$-category $\cH \bA$ of objects, horizontal morphisms, and squares with trivial vertical boundaries. Another $2$-category $\SV\bA$ that can be extracted from a double category $\bA$ has as objects the vertical morphisms of $\bA$, as morphisms the squares of $\bA$, and $2$-morphisms as described in \Cref{defn:SVA}. These $2$-categories $\cH\bA$ and $\SV\bA$ allow one to retrieve most of the structure of the double category $\bA$, except for composition of vertical morphisms.

We explore these constructions as a means of comparing a weak notion of initial objects in a double category $\bA$ with bi-initial objects in the $2$-categories $\cH\bA$ and $\SV\bA$. In \cref{sec:dblvs2biinitial} these notions of initiality for objects are defined, in analogy with the $1$-dimensional case, by requiring the projection from the slice over the considered object to be an appropriate equivalence -- in the case of $2$-categories, a bi-equivalence. Both in the $2$-categorical and double categorical case, these projection morphisms are always strict functors and appear as fibrations in certain model structures. As a consequence, initiality can be equivalently defined by requiring the projection to be a \emph{trivial fibration} -- a special case of a bi-equivalence or its double categorical analogue.

With this motivation, in \cref{subsec:trivfib} we recall the definition of a trivial fibration in Lack's model structure on $2$-categories and $2$-functors \cite{Lack}. An analogous notion of \emph{double trivial fibrations} is introduced in \cite{usbutnotyou} by the second-named author, Sarazola, and Verdugo as the trivial fibrations in a model structure on double categories and double functors. While the double trivial fibrations are defined as those double functors whose images under \cH and \SV are trivial fibrations of $2$-categories, we show in \cref{prop:dblvs2trivfib} that double trivial fibrations are precisely those double functors whose image under \SV alone is a trivial fibration. This theorem is the essential content of our main result, and is responsible for allowing us to formulate the universal property of bi-representations in terms of bi-initial objects in the corresponding ``$2$-category of morphisms''.

\subsection{The functors \texorpdfstring{$\bH$}{IH}, \texorpdfstring{$\cH$}{H}, and \texorpdfstring{$\SV$}{V}} \label{subsec:HHV}

Let us first introduce the horizontal full embedding functor from $\TwoCat$ to $\DblCat$.

\begin{definition} \label{defn:bH}
    We define a functor $\ainl{\TwoCat}{\bH}{\DblCat}$ which sends a $2$-category~$\cA$ to the horizontal double category $\bH\cA$ with the same objects as $\cA$, horizontal morphisms given by the morphisms of $\cA$, only trivial vertical morphisms, and squares $\sq{\alpha}{a}{b}{\vid_A}{\vid_{A'}}$ given by the $2$-morphisms $\ninl{a}{\alpha}{b}$ of $\cA$.

    Given a normal pseudo-functor $\ainl{\cA}{F}{\cB}$, the induced normal pseudo-double functor $\ainl{\bH\cA}{\bH F}{\bH\cB}$ acts as $F$ does on the corresponding data, and respects vertical identities. The compositor vertically invertible squares of $\bH F$ are the ones corresponding to the compositor $2$-isomorphisms of $F$.
\end{definition}

The functor $\bH$ has a right adjoint, given by the following functor.

\begin{definition} \label{defn:cH}
  The functor $\ainl{\DblCat}{\cH}{\TwoCat}$ sends a double category $\bA$ to its \textbf{underlying horizontal $2$-category} $\cH\bA$ with the same objects as $\bA$, morphisms given by the horizontal morphisms of $\bA$, and $2$-morphisms $\ninl{a}{\alpha}{b}$ given by the squares in $\bA$ of the form
  \begin{diagram*}[.][2]
    \node[](1) {$A$};
    \node[below of=1](2) {$A$};
    \node[right of=1](3) {$A'$};
    \node[below of=3](4) {$A'$};
    \draw[proequal] (1) to (2);
    \draw[proequal] (3) to (4);
    \draw[a] (1) to node[la] {$a$} (3);
    \draw[a] (2) to node[swap,la] {$b$} (4);
    \node[la] at ($(1)!0.5!(4)$) {$\alpha$};
  \end{diagram*}

  Given a normal pseudo-double functor $\ainl{\bA}{F}{\bB}$, the induced normal pseudo-functor ${\ainl{\cH\bA}{\cH F}{\cH\bB}}$ acts as $F$ does on the corresponding data, and the data of its compositor $2$-isomorphisms are given by the compositor squares of $F$.
\end{definition}

\begin{proposition} \label{lem:HHadjunction}
  The functors $\ainl{\TwoCat}{\bH}{\DblCat}$ and $\ainl{\DblCat}{\cH}{\TwoCat}$ form an adjunction $\bH\dashv \cH$ such that the unit $\ninl{\id_{\TwoCat}}{\eta}{\cH\bH}$ is an identity.
\end{proposition}

\begin{proof}
  Let $\cC$ be a $2$-category, and $\bA$ be a double category. By specialising \Cref{defn:pseudodouble} to the case where the source is a double category with only trivial vertical morphisms, we see that normal pseudo-double functors $\bH\cC\to\bA$ correspond precisely to normal pseudo-functors $\cC\to\cH\bA$, i.e., we have an isomorphism of sets
  \[ \DblCat(\bH\cC,\bA)\iso \TwoCat(\cC,\cH\bA)\ , \]
  natural in $\cC$ and $\bA$. Moreover, a straightforward computation shows that $\cH\bH\cC=\cC$.
\end{proof}

We want to extract another $2$-category from a double category which contains the data of all vertical morphisms and squares, and this is done through the following functor $\DblCat\to \TwoCat$. In \cite[Definition 2.10]{usbutnotyou}, the second-named author, Sarazola, and Verdugo give a similar definition but in a setting where the morphisms of $2$-categories and double categories are \emph{strict}. Under the inclusion of the appropriate subcategories into our weaker setting, our functor may be seen to restrict to theirs.

Although this functor appeared chronologically prior in the work of \cite{usbutnotyou}, the original motivation to isolate and treat its definition was our \cref{thm:initial} below. Indeed, as we shall see in \cref{REMARK}, from this context the below definition naturally emerges.

\begin{definition} \label{defn:SVA}
    Let $\bV\mathbbm 2$ denote the free double category on a vertical morphism. We define the functor $\ainl{\DblCat}{\SV}{\TwoCat}$ to be the composite
    \[\DblCat\xrightarrow{\PPsd(\bV\mathbbm 2,-)}\DblCat\stackrel{\cH}{\longrightarrow}\TwoCat\ . \]
    In particular, it sends a double category $\bA$ to the $2$-category $\SV\bA$ whose
\begin{drome}
  \item objects are the vertical morphisms of $\bA$,
  \item morphisms $\ainl{u}{\alpha}{u'}$ are squares in $\bA$ of the form
\begin{diagram*}[,][2]
    \node[](1) {$A$};
    \node[below of=1](2) {$B$};
    \node[right of=1](3) {$A'$};
    \node[below of=3](4) {$B'$};
    \draw[proarrow] (1) to node[swap,la] {$u$} (2);
    \draw[proarrow] (3) to node[la] {$u'$} (4);
    \draw[a] (1) to node[la] {$a$} (3);
    \draw[a] (2) to node[swap,la] {$b$} (4);
    \node[la] at ($(1)!0.5!(4)$) {$\alpha$};
\end{diagram*}
    \item \label{2morinSV}$2$-morphisms $\ninl{\alpha}{(\sigma_0,\sigma_1)}{\alpha'}$ are pairs of squares $\sq{\sigma_0}{a}{a'}{\vid_A}{\vid_{A'}}$ and $\sq{\sigma_1}{b}{b'}{\vid_B}{\vid_{B'}}$ satisfying the following pasting equality.
\begin{diagram*}
    \node(1)[]{$A$};
    \node(2)[below of= 1]{$A$};
    \draw[proequal](1)to (2);
    \node(3)[right of= 1]{$A'$};
    \node(4)[below of= 3]{$A'$};
    \draw[proequal](3)to (4);
    \node (5)[below of= 2]{$B$};
    \node (6)[right of= 5]{$B'$};
    \draw[a](1)to node[la]{$a$}(3);
    \draw[a](2)to node[la,over]{$a'$}(4);
    \node[la] at ($(1)!0.5!(4)$) {$\sigma_0$};
    \draw[proarrow](2) to node[la,swap]{$u$}(5);
    \draw[proarrow](4) to node[la]{$u'$}(6);
    \draw[a](5)to node[swap,la]{$b'$}(6);
    \node[la] at ($(2)!0.5!(6)$) {$\alpha'$};

    \node(1)[right of= 4]{$B$};
    \node at ($(4)!0.5!(1)$) {$=$};
    \node(2)[below of= 1]{$B$};
    \draw[proequal](1)to (2);
    \node(3)[right of= 1]{$B'$};
    \node(4)[below of= 3]{$B'$};
    \draw[proequal](3)to (4);
    \draw[a](1)to node[over,la]{$b$}(3);
    \draw[a](2)to node[swap,la]{$b'$}(4);
    \node[la] at ($(1)!0.5!(4)$) {$\sigma_1$};
    \node(5)[above of=1]{$A$};
    \node(6)[right of=5]{$A'$};
    \draw[proarrow](5) to node[swap,la]{$u$}(1);
    \draw[proarrow](6) to node[la]{$u'$}(3);
    \draw[a](5) to node[la]{$a$}(6);
    \node[la] at ($(5)!0.5!(3)$) {$\alpha$};
\end{diagram*}
\end{drome}
\end{definition}

\begin{remark}
  The functor $\ainl{\DblCat}{\SV}{\TwoCat}$ is also a right adjoint since it is the composite of two right adjoints $\cH$ and $\PPsd(\bV\mathbbm 2,-)$, and its left adjoint is given by $\bH(-)\times \bV\mathbbm 2$.
\end{remark}

\subsection{Trivial fibrations} \label{subsec:trivfib}

We now recall the definition of the \emph{trivial fibrations} in Lack's model structure on the category of $2$-categories and $2$-functors; see \cite{Lack}.

\begin{definition}
    Let $\cA$ and $\cB$ be $2$-categories. A $2$-functor $F\colon \cA\to \cB$ is a \textbf{trivial fibration} if
    \begin{rome}
      \item for every object $B\in \cB$, there is an object $A\in \cA$ such that $FA=B$,
      \item for every morphism $\ainl{FA}{b}{FA'}$ in $\cB$, there is a morphism $\ainl{A}{a}{A'}$ in~$\cA$ such that $Fa= b$,
      \item for every $2$-morphism $\ninl{Fa}{\beta}{Fb}$ in $\cB$, there is a unique $2$-morphism $\ninl{a}{\alpha}{b}$ in $\cA$ such that $F\alpha=\beta$.
    \end{rome}
\end{definition}

Similarly, we recall double trivial fibrations which were defined as the trivial fibrations in the model structure on double categories and double functors constructed by the second-named author, Sarazola, and Verdugo in \cite{usbutnotyou}.

\begin{definition} \label{def:doubleq}
  Let $\bA$ and $\bB$ be double categories. A double functor $\ainl{\bA}{F}{\bB}$ is a \textbf{double trivial fibration} if
\begin{drome}
  \item\label{cond:ess_surj_obj} for every object $B\in \bB$, there is an object $A\in \bA$ such that $FA=B$,
  \item\label{cond:ff_hor_mor} for every horizontal morphism $\ainl{FA}{b}{FA'}$ in $\bB$, there is a horizontal morphism $\ainl{A}{a}{A'}$ such that $Fa=b$,
  \item\label{cond:ess_surj_ver_mor} for every vertical morphism $\vainl{B}{v}{D}$ in $\bB$, there is a vertical morphism $\vainl{A}{u}{C}$ in $\bA$ such that $Fu=v$,
  \item\label{cond:ff_squares} for every square $\beta$ in $\bB$ of the form
    \begin{diagram*}[,][3]
      \node[](1) {$FA$};
      \node[right of=1](2) {$FA'$};
      \node[below of=1](3) {$FC$};
      \node[right of=3](4) {$FC'$};
      \draw[proarrow] (1) to node[la,swap] {$Fu$} (3);
      \draw[proarrow] (2) to node[la] {$Fu'$} (4);
      \draw[a] (1) to node[la] {$Fa$} (2);
      \draw[a] (3) to node[la,swap] {$Fc$} (4);
      \node[la] at ($(1)!0.5!(4)$) {$\beta$};
    \end{diagram*}
    there is a unique square $\sq{\alpha}{a}{c}{u}{u'}$ in $\bA$ such that $F\alpha=\beta$.
  \end{drome}
\end{definition}

Since fibrations and weak equivalences in the model structure on double categories are defined to be the double functors whose images under $\cH$ and $\SV$ are fibrations and weak equivalences in Lack's model structure (see \cite[Theorem 3.19]{usbutnotyou}), it is possible to characterise double trivial fibrations in terms of trivial fibrations of $2$-categories. However, even more is true: $\SV$ alone captures enough data to detect the entire model structure on double categories, and in particular the double trivial fibrations.

\begin{theorem} \label{prop:dblvs2trivfib}
 Let $\bA$ and $\bB$ be double categories and $\ainl{\bA}{F}{\bB}$ be a double functor. The following statements are equivalent.
 \begin{rome}
 \item The double functor $F$ is a double trivial fibration.
 \item The $2$-functors $\ainl{\cH\bA}{\cH F}{\cH\bB}$ and $\ainl{\SV\bA}{\SV F}{\SV\bB}$ are trivial fibrations.
 \item The $2$-functor $\ainl{\SV\bA}{\SV F}{\SV\bB}$ is a trivial fibration.
 \end{rome}
\end{theorem}

\begin{proof}
 The equivalence of (i) and (ii) follows from \cite[Corollary 3.14]{usbutnotyou}. The equivalence of (ii) and (iii) is a consequence of the facts that $\cH F$ is a retract of $\SV F$ and that trivial fibrations are closed under retracts.
\end{proof}

\section{Pseudo-comma \texorpdfstring{$2$}{2}-dimensional categories} \label{sec:pseudo-comma}

We now introduce pseudo-comma double categories and pseudo-comma $2$-categories, and show that they are related through the functors $\ainl{\DblCat}{\cH,\SV}{\TwoCat}$. We treat these objects in general so that we may later variously specialise the theory to pseudo-slices both over and under objects. We will use these results for the purposes of comparing double bi-initial objects in a double category with bi-initial objects in the induced $2$-categories obtained by applying \cH and \SV in \cref{sec:initial}, as well as for computing the double categories of elements in the case of bi-adjunctions and weighted bi-limits in \cref{sec:applications}.

Let us first define the pseudo-comma double category of a cospan of normal pseudo-double functors. With an eye to our applications of this theory in \cref{sec:initial,sec:applications}, we then give a more explicit description of the data in a \emph{pseudo-slice} double category: a pseudo-comma where one of the double categories involved is terminal.

\begin{definition}\label{def:double-comma}
Let $\ainl{\bC}{G}{\bA}$ and $\ainl{\bB}{F}{\bA}$ be normal pseudo-double functors. The \textbf{pseudo-comma double category} $\dblslice G F$ of $G$ and $F$ is defined as the following pullback in \DblCat,
\begin{diagram*}
  \node(1)[]{$\dblslice G F$};
  \node(2)[below of= 1]{$\bC\times \bB$};
  \node(3)[right of= 2, xshift=1cm]{$\bA\times \bA$};
  \node(4)[above of= 3]{$\PPsd(\bH\mathbbm{2},\bA)$};
  \draw[a](1)to node[swap,la] {$\Pi$} (2);
  \draw[a](2)to node[swap,la]{$(G,F)$}(3);
  \draw[a](1)to (4);
  \draw[a](4) to node[la]{$(s,t)$}(3);
  \pullback{1};
\end{diagram*}
where $\bH\mathbbm 2$ is the free double category on a horizontal morphism and $\PPsd(-,-)$ is the double category described in \cref{def:inthomDbl}.
\end{definition}

  Note that $\ainl{\dblslice GF}{\Pi}{\bC\times \bB}$ is a \emph{strict} double functor.

\begin{remark} \label{def:pseudoslicedouble}
 We give an explicit description of the pseudo-comma double category in the case where $\bC=\mathbbm 1$ is the terminal category and $\ainl{\mathbbm 1}{G=I}{\bA}$ is an object in $\bA$. This is the double category $\dblslice I F$, called \textbf{pseudo-slice double category}, whose
\begin{rome}
  \item objects are pairs $(B,f)$ of an object $B\in \bB$ and a horizontal morphism $\ainl{I}{f}{FB}$ in $\bA$,
  \item horizontal morphisms $\ainl{(B,f)}{(b,\psi)}{(B',f')}$ comprise the data of a horizontal morphism $\ainl{B}{b}{B'}$ in $\bB$ and a vertically invertible square $\psi$ in $\bA$ of the form
\begin{diagram*}[,][4]
    \node(1)[]{$I$};
    \node(2)[below of=1]{$I$};
    \node(3)[right of=2]{$FB$};
    \node(4)[right of=3]{$FB'$};
    \node(5)[above of=4]{$FB'$};
    \draw[proequal] (1) to (2);
    \draw[proequal] (5) to (4);
    \draw[a] (1) to node[la] {$f'$} (5);
    \draw[a] (2) to node[swap,la] {$f$} (3);
    \draw[a] (3) to node[swap,la] {$Fb$} (4);
    \node[la,xshift=-5pt] at ($(1)!0.5!(4)$) {$\psi$};
    \node[la,xshift=5pt] at ($(1)!0.5!(4)$) {$\vcong$};
\end{diagram*}
    \item vertical morphisms $\vainl{(B,f)}{(u,\gamma)}{(C,g)}$ comprise the data of a vertical morphism $\vainl{B}{u}{C}$ in $\bB$ and a square $\gamma$ in $\bA$ of the form
\begin{diagram*}[,][3]
    \node(1)[]{$I$};
    \node(2)[below of=1]{$I$};
    \node(3)[right of=2]{$FC$};
    \node(4)[above of=3]{$FB$};
    \draw[proequal] (1) to (2);
    \draw[proarrow] (4) to node[la] {$Fu$} (3);
    \draw[a] (1) to node[la] {$f$} (4);
    \draw[a] (2) to node[swap,la] {$g$} (3);
    \node[la] at ($(1)!0.5!(3)$) {$\gamma$};
\end{diagram*}
    \item squares
\begin{diagram*}
    \node(1)[]{$(B,f)$};
    \node(2)[below of=1]{$(C,g)$};
    \node(3)[right of=2,xshift=.5cm]{$(C',g')$};
    \node(4)[above of=3]{$(B',f')$};
    \draw[proarrow] (1) to node[swap,la] {$(u,\gamma)$} (2);
    \draw[proarrow] (4) to node[la] {$(u',\gamma')$} (3);
    \draw[a] (1) to node[la] {$(b,\psi)$} (4);
    \draw[a] (2) to node[swap,la] {$(c,\varphi)$} (3);
    \node[la] at ($(1)!0.5!(3)$) {$\beta$};
\end{diagram*}
    comprise the data of a square $\sq{\beta}{b}{c}{u}{u'}$ in $\bB$ such that the following pasting equality holds in $\bA$.
\begin{diagram*}
    \node(1)[]{$I$};
    \node(2)[below of=1]{$I$};
    \node(3)[right of=2]{$FB$};
    \node(4)[right of=3]{$FB'$};
    \node(5)[above of=4]{$FB'$};
    \draw[proequal] (1) to (2);
    \draw[proequal] (5) to (4);
    \draw[a] (1) to node[la] {$f'$} (5);
    \draw[a] (2) to node[over,la] {$f$} (3);
    \draw[a] (3) to node[la] {$Fb$} (4);
    \node[la,xshift=-5pt] at ($(1)!0.5!(4)$) {$\psi$};
    \node[la,xshift=5pt] at ($(1)!0.5!(4)$) {$\vcong$};

    \node(6)[below of=2] {$I$};
    \node(7)[right of=6] {$FC$};
    \node (8)[right of=7] {$FC'$};
    \draw[proequal] (2) to (6);
    \draw[proarrow] (3) to node[la,swap] {$Fu$} (7);
    \draw[proarrow] (4) to node[la] {$Fu'$} (8);
    \draw[a] (6) to node[swap,la] {$g$} (7);
    \draw[a] (7) to node[swap,la] {$Fc$} (8);
    \node[la,xshift=-5pt] at ($(2)!0.5!(7)$) {$\gamma$};
    \node[la] at ($(3)!0.5!(8)$) {$F\beta$};

    \node(6)[right of=5] {$I$};
    \node at ($(8)!0.5!(6)$) {$=$};
    \node(1)[below of=6]{$I$};
    \node(2)[below of=1]{$I$};
    \node(3)[right of=2]{$FC$};
    \node(4)[right of=3]{$FC'$};
    \node(5)[above of=4]{$FC'$};
    \node(7)[above of=5]{$FB'$};
    \draw[proequal] (6) to (1);
    \draw[proarrow] (7) to node[la] {$Fu'$} (5);
    \draw[proequal] (1) to (2);
    \draw[proequal] (5) to (4);
    \draw[a] (6) to node[la] {$f'$} (7);
    \draw[a] (1) to node[la,over] {$g'$} (5);
    \draw[a] (2) to node[la,swap] {$g$} (3);
    \draw[a] (3) to node[la,swap] {$Fc$} (4);
    \node[la,xshift=-5pt] at ($(1)!0.5!(4)$) {$\varphi$};
    \node[la,xshift=5pt] at ($(1)!0.5!(4)$) {$\vcong$};
    \node[la] at ($(6)!0.5!(5)$) {$\gamma'$};
\end{diagram*}
\end{rome}

The double functor $\ainl{\dblslice I F}{\Pi}{\mathbbm 1\times \bB\cong \bB}$ is the projection onto the $\bB$-component.

If $\bB=\bA$ and $F=\id_\bA$, we write $\dblslice I\bA\coloneqq \dblslice{I}{\id_\bA}$.
\end{remark}

We now define the pseudo-comma $2$-category of a cospan of normal pseudo-functors, and also give an explicit description of the special case of a pseudo-slice $2$-category.

\begin{definition}
Let $G\colon \cC\to \cA$ and $F\colon \cB\to \cA$ be normal pseudo-functors. The \textbf{pseudo-comma $2$-category} $\slice G F$ is defined as the following pullback in \TwoCat.
\begin{diagram*}
  \node(1)[]{$\slice G F$};
  \node(2)[below of= 1]{$\cC\times \cB$};
  \node(3)[right of= 2, xshift=1cm]{$\cA\times \cA$};
  \node(4)[above of= 3]{$\Psd(\mathbbm{2},\cA)$};
  \draw[a](1)to node[swap,la] {$\pi$} (2);
  \draw[a](2)to node[swap,la]{$(G,F)$}(3);
  \draw[a](1)to (4);
  \draw[a](4) to node[la]{$(s,t)$}(3);
  \pullback{1};
\end{diagram*}
where $\mathbbm 2$ is the free $2$-category on a morphism and $\Psd(-,-)$ is the $2$-category described in \cref{def:inthom2Cat}.
\end{definition}

  Note that $\ainl{\slice GF}{\pi}{\cC\times \cB}$ is a \emph{strict} $2$-functor.

\begin{remark} \label{def:pseudoslice2}
  We give an explicit description of the pseudo-comma $2$-category in the case where $\cC=\mathbbm 1$ is the terminal category and $\ainl{\mathbbm 1}{G=I}{\cA}$ is an object in $\cA$. This is the $2$-category $\slice{I}{F}$, called a \textbf{pseudo-slice $2$-category}, whose
  \begin{rome}
  \item objects are pairs $(B,f)$ of an object $B\in \cB$ and a morphism $\ainl I f {FB}$ in $\cA$,
  \item morphisms $\ainl{(B,f)}{(b,\psi)}{(B',f')}$ comprise the data of a morphism $\ainl B b{B'}$ in $\cB$ and a $2$-isomorphism $\psi$ in $\cA$ of the form
\begin{diagram*}[,][3][node distance=1.8cm]
  \node(1)[]{$I$};
  \node(2)[right of =1]{$FB$};
  \node(3)[below of =2]{$FB'$};
  \draw[a](1) to node[la]{$f$} (2);
  \draw[a](1) to node(x)[swap,la]{$f'$} (3);
  \draw[a](2) to node[la]{$Fb$} (3);
  \celli[la]{x}{2}{$\psi$};
\end{diagram*}
  \item $2$-morphisms $\ninl {(b,\psi)}{\beta}{(c,\varphi)}$ comprise the data of a $2$-morphism $\ninl{b}{\beta}{c}$ in $\cB$ such that the following pasting equality holds in $\cA$.
\begin{diagram*}[][][node distance=1.8cm]
    \node(1)[]{$I$};
    \node(2)[right of= 1]{$FB$};
    \node(3)[below of= 2]{$FB'$};
    \draw[a](2)to node(a)[la,over]{$Fb$}(3);
    \draw[a, bend left=60](2)to node(b)[la,right]{$Fc$}(3);
    \draw[a](1)to node(f)[la,swap]{$f$}(3);
    \draw[a](1)to node(x')[la]{$f'$}(2);
    \cell[la,xshift=-.2cm][n][.45]{a}{b}{$F\beta$};
    \celli[la,shrink]{f}{2}{$\psi$};

    \node(4)[right of= a,xshift=.2cm]{$=$};

    \node(1)[right of= 2, xshift=.75cm]{$I$};
    \node(2)[right of= 1]{$FB$};
    \node(3)[below of= 2]{$FB'$};
    \draw[a](2)to node(a)[la]{$Fc$}(3);
    \draw[a](1)to node[la]{$f'$}(2);
    \draw[a](1)to node(y)[swap,la]{$f$}(3);
    \celli[la,shrink]{y}{2}{$\varphi$};
  \end{diagram*}
\end{rome}

The $2$-functor $\ainl{\slice I F}{\pi}{\mathbbm 1\times\cB\cong \cB}$ is the projection onto the $\cB$-component.

If $\cB=\cA$ and $F=\id_\cA$, we write $\slice I\cA\coloneqq \slice{I}{\id_\cA}$ with $I\in \cA$.
\end{remark}

\begin{remark}
Given the explications of \Cref{def:pseudoslicedouble,def:pseudoslice2}, we wish to draw the reader's attention to an important disparity between the double categories $\dblslice I{\bH F}$ and $\bH(\slice I F)$, for a $2$-functor $\ainl \cB F\cA$ and an object $I\in \cA$. While the latter double category has only trivial vertical morphisms, the former has all $2$-morphisms of \cA of the form $\ninl f\gamma g$ for $\ainl I {f,g}{FB}$ as vertical morphisms -- a far richer stock of information. This is symptomatic of a broader truth: the double category $\PPsd(\bH\mathbbm2,\bH\cA)$ has all $2$-morphisms of $\cA$ as vertical morphisms, while $\bH\Psd(\mathbbm 2,\cA)$ is the double category associated to its underlying horizontal $2$-category $\Psd(\mathbbm 2,\cA)=\cH\PPsd(\bH\mathbbm2,\bH\cA)$ and therefore has only trivial vertical morphisms.
\end{remark}

While \bH thus does not preserve pseudo-comma objects, the main result of this section says that the functors $\cH$ and $\SV$ do preserve pseudo-comma objects, in the sense that the $2$-category obtained by applying $\cH$ or $\SV$ to the pseudo-comma double category associated to a cospan is isomorphic to the pseudo-comma $2$-category of the image under $\cH$ or $\SV$ of the original cospan. This is the content of the following proposition and the rest of the section will be devoted to its proof.

\begin{proposition}\label{lem:h-sv-general-commas}
  Let $\ainl{\bC}{G}{\bA}$ and $\ainl{\bB}{F}{\bA}$ be normal pseudo-double functors. Then there are canonical isomorphisms of $2$-categories as in the following commutative squares.
  \begin{center}
    \begin{tikzpicture}
      \node(1)[]{$\cH(\dblslice{G}{F})$};
      \node(2)[right of= 1,xshift=1cm]{$\slice{\cH G}{\cH F}$};
      \node(3)[below of= 1]{$\cH(\bC\times\bB)$};
      \node(4)[below of= 2]{$\cH\bC\times\cH\bB$};
      \draw[a](1)to node[la]{$\iso$}(2);
      \draw[a](1)to node[la,swap]{$\cH\Pi$}(3);
      \draw[a](2)to node[la]{$\pi$}(4);
      \draw[a](3)to node[la,swap]{$\iso$}(4);
    \end{tikzpicture}\hspace{2cm}
    \begin{tikzpicture}
      \node(1)[]{$\SV(\dblslice{G}{F})$};
      \node(2)[right of= 1,xshift=1cm]{$\slice{\SV G}{\SV F}$};
      \node(3)[below of= 1]{$\SV(\bC\times\bB)$};
      \node(4)[below of= 2]{$\SV\bC\times\SV\bB$};
      \draw[a](1)to node[la]{$\iso$}(2);
      \draw[a](1)to node[la,swap]{$\SV\Pi$}(3);
      \draw[a](2)to node[la]{$\pi$}(4);
      \draw[a](3)to node[la,swap]{$\iso$}(4);
    \end{tikzpicture}
  \end{center}
\end{proposition}

To prove this we first show that the functors $\cH$ and $\SV$ behave well with respect to $\Psd$ and $\PPsd$.

\begin{lemma}\label{lem:annoying}
  For every $2$-category $\cC$ and every double category $\bA$, there is an isomorphism of $2$-categories
 \[ \Psd(\cC, \cH \bA)\iso\cH\PPsd(\bH\cC,\bA) \]
 natural in $\cC$ and $\bA$.
\end{lemma}

\begin{proof}
Recall from \cref{lem:HHadjunction} that $\bH\dashv \cH$ form an adjunction between $\TwoCat$ and $\DblCat$, so that we have an isomorphism at the level of objects which is natural in $\cC$ and $\bA$. By unpacking the definition of a horizontal pseudo-natural transformation (see \cref{def:horizontal-pseudo-natural-transformation}) between normal pseudo-double functors $\bH\cC\to \bA$, we can see that such data corresponds to that of a pseudo-natural transformation (see \cref{def:pseudo-natural-transformation}) between normal pseudo-functors $\cC\to \cH\bA$ as the squares in \cref{defn:pseudo-double-unit-squares} are all trivial given that all vertical morphisms of $\bH\cC$ are identities. Similarly, one can check that, using \cref{def:modification} and \cite[Definition 3.8.3]{Grandis}, the $2$-morphisms of these $2$-categories coincide.
\end{proof}

In order to give the next result, an analogous statement for $\SV$, we will make use of the following technical lemma.

\begin{lemma}\label{lem:morethanannoying}
For every $2$-category $\cC$ and every double category $\bA$, there are isomorphisms of $2$-categories
  \[\cH\PPsd(\bH\cC,\PPsd(\bV\mathbbm 2,\bA))\cong \cH \PPsd(\bV\mathbbm 2,\PPsd(\bH\cC,\bA))\]
   natural in $\cC$ and $\bA$.
\end{lemma}

\begin{proof}
  That this result holds is due to the appearance of $\bH\cC$ and $\bV\mathbbm2$ in its statement. Should we replace $\bV\mathbbm 2$ by a more general double category with non-trivial horizontal morphisms, or replace $\bH\cC$ by a more general double category with non-trivial vertical morphisms, the result would fail to hold in general.

  We show that we have an isomorphism on objects. A normal pseudo-double functor $\ainl{\bH\cC}\nu{\PPsd(\bV\mathbbm2,\bA)}$ assigns to each object of $C\in\cC$ a vertical morphism $\vainl{FC}{\nu}{GC}$ in $\bA$, to each morphism $\ainl{C}{c}{C'}$ in $\cC$ a square $\sq{\nu_c}{Fc}{Gc}{\nu_C}{\nu_{C'}}$ in $\bA$, to each $2$-morphism $\ninl{c}{\gamma}{c'}$ in $\cC$ two squares $F\gamma$ and $G\gamma$ in $\bA$ which satisfy the following pasting equality,
  \begin{diagram*}
    \node(1)[]{$FC$};
    \node(2)[below of= 1]{$FC$};
    \draw[proequal](1)to (2);
    \node(3)[right of= 1]{$FC'$};
    \node(4)[below of= 3]{$FC'$};
    \draw[proequal](3)to (4);
    \node (5)[below of= 2]{$GC$};
    \node (6)[right of= 5]{$GC'$};
    \draw[a](1)to node[la]{$Fc$}(3);
    \draw[a](2)to node[la]{$Fc'$}(4);
    \node[la] at ($(1)!0.5!(4)$) {$F\gamma$};
    \draw[proarrow](2) to node[la,swap]{$\nu_C$}(5);
    \draw[proarrow](4) to node[la]{$\nu_{C'}$}(6);
    \draw[a](5)to node[swap,la]{$Gc'$}(6);
    \node[la] at ($(2)!0.5!(6)$) {$\nu_{c'}$};

    \node(1)[right of= 4]{$GC$};
    \node at ($(4)!0.5!(1)$) {$=$};
    \node(2)[below of= 1]{$GC$};
    \draw[proequal](1)to (2);
    \node(3)[right of= 1]{$GC'$};
    \node(4)[below of= 3]{$GC'$};
    \draw[proequal](3)to (4);
    \draw[a](1)to node[la]{$Gc$}(3);
    \draw[a](2)to node[swap,la]{$Gc'$}(4);
    \node[la] at ($(1)!0.5!(4)$) {$G\gamma$};
    \node(5)[above of=1]{$FC$};
    \node(6)[right of=5]{$FC'$};
    \draw[proarrow](5) to node[swap,la]{$\nu_C$}(1);
    \draw[proarrow](6) to node[la]{$\nu_{C'}$}(3);
    \draw[a](5) to node[la]{$Fc$}(6);
    \node[la] at ($(5)!0.5!(3)$) {$\nu_c$};
\end{diagram*}
  and to each pair of composable morphisms $\ainl{C}{c}{C'}$ and $\ainl{C}{c'}{C''}$ in $\cC$, two vertically invertible squares $\phi_{c,c'}$ and $\psi_{c,c'}$ in $\bA$ which satisfy the following pasting equality.
\begin{diagram*}
    \node[](1) {$GC$};
    \node[right of=1](5) {$GC'$};
    \node[below of=1](2) {$GC$};
    \node[right of=5](3) {$GC''$};
    \node[below of=3](4) {$GC''$};
    \node[above of=1](1') {$FC$};
    \node[above of=5](5') {$FC'$};
    \node[above of=3](3') {$FC''$};
    \draw[proequal] (1) to (2);
    \draw[proequal] (3) to (4);
    \draw[proarrow] (1') to node[swap,la]{$\nu_C$} (1);
    \draw[proarrow] (5') to node[swap,la]{$\nu_{C'}$} (5);
    \draw[proarrow] (3') to node[la]{$\nu_{C''}$} (3);
    \draw[a] (1') to node[la]{$Fc$}(5');
    \draw[a] (5') to node[la]{$Fc'$}(3');
    \draw[a] (1) to node[la]{$Gc$}(5);
    \draw[a] (5) to node[la]{$Gc'$}(3);
    \draw[a] (2) to node [la,swap]{$G(c'c)$} (4);
    \node[la,xshift=-8pt] at ($(1)!0.5!(4)$) {$\psi_{c,c'}$};
    \node[la,xshift=8pt] at ($(1)!0.5!(4)$) {$\vcong$};

    \path (1') -- node[la,auto=false,xshift=-5pt]{$\nu_c$}(5);
    \path(5') --node[la,auto=false]{$\nu_{c'}$} (3);
    \node[right of=3'](1) {$FC$};
    \node[below of=1](2) {$FC$};
    \path(3) --node[auto=false]{$=$} (2);
    \node[right of=1](5) {$FC'$};
    \node[right of=5](3) {$FC''$};
    \node[below of=3](4) {$FC''$};
    \node[below of=2](2') {$GC$};
    \node[below of=4](4') {$GC''$};
    \draw[proequal] (1) to (2);
    \draw[proequal] (3) to (4);
    \draw[proarrow] (2) to node[swap,la]{$\nu_C$} (2');
    \draw[proarrow] (4) to node[la]{$\nu_{C''}$} (4');
    \draw[a] (1) to node[la]{$Fc$}(5);
    \draw[a] (5) to node[la]{$Fc'$}(3);
    \draw[a] (2) to node [la,over]{$F(c'c)$} (4);
    \draw[a] (2') to node [la,swap]{$G(c'c)$} (4');
    \node[la,xshift=-8pt] at ($(1)!0.5!(4)$) {$\phi_{c,c'}$};
    \node[la,xshift=8pt] at ($(1)!0.5!(4)$) {$\vcong$};

    \path (2) -- node[la,auto=false]{$\nu_{c'c}$}(4');
\end{diagram*}

One can check that not only is this data precisely the underlying data of two normal pseudo-double functors $\ainl{\bH\cC}{F,G}{\bA}$ together with a vertical natural transformation $\ninl F\nu G$ between them, but also that the various laws governing the compositors of $F,G$ hold. The two diagrams above already demonstrate that $\nu$ is natural.

We leave it to the reader to check that the morphisms and $2$-morphisms of the $2$-categories considered also coincide under this identification.
\end{proof}

\begin{lemma} \label{cor:alsoannoying}
  For every $2$-category $\cC$ and every double category $\bA$, there is an isomorphism of $2$-categories
 \[ \Psd(\cC, \SV \bA)\iso\SV\PPsd(\bH\cC,\bA) \]
 natural in $\cC$ and $\bA$.
\end{lemma}

\begin{proof}
We have the following isomorphisms
\begin{align*}
    \Psd(\cC, \SV\bA) & = \Psd(\cC, \cH\PPsd(\bV\mathbbm 2,\bA)) & \text{(definition of } \SV) \\
    & \cong \cH \PPsd(\bH\cC,\PPsd(\bV\mathbbm 2,\bA)) & \text{(\cref{lem:annoying})}  \\
    & \cong \cH \PPsd(\bV\mathbbm 2,\PPsd(\bH\cC,\bA)) & \text{(\cref{lem:morethanannoying})} \\
    & = \SV \PPsd(\bH\cC,\bA). & \text{(definition of } \SV)
\end{align*}
natural in $\cC$ and $\bA$.
\end{proof}

The proof of \cref{lem:h-sv-general-commas} now follows from these results and the fact that $\cH$ and $\SV$ are right adjoints, and therefore preserve limits.

\begin{proof}[\cref{lem:h-sv-general-commas}]
Let us consider the following diagram.
\begin{diagram*}
  \node(1)[]{$\slice {\cH G}{\cH F}$};
  \node(2)[below of= 1]{$\cH\bC\times \cH\bB$};
  \node(8)[left of=2,xshift=-1.5cm] {$\cH(\bC\times \bB)$};
  \node(7)[above of=8] {$\cH(\dblslice GF)$};
  \node(3)[right of= 2, xshift=1.5cm]{$\cH\bA\times \cH\bA$};
  \node(4)[above of= 3]{$\Psd(\mathbbm 2, \cH \bA)$};
  \node(5) [right of=4,xshift=1.5cm] {$\cH\PPsd(\bH\mathbbm{2},\bA)$};
  \node(6) [below of=5] {$\cH(\bA\times \bA)$};
  \draw[a] (4) to node[la] {$\iso$} (5);
  \draw[a] (3) to node[swap,la] {$\iso$} (6);
  \draw[a] (8) to node[swap,la] {$\iso$} (2);
  \draw[a] (5) to node[la] {$\cH(s,t)$} (6);
  \draw[a](1)to node[over,la] {$\pi$} (2);
  \draw[a](7)to node[swap,la] {$\cH\Pi$} (8);
  \draw[a](2)to node[swap,la]{$\cH(G\times F)$}(3);
  \draw[a](1)to (4);
  \draw[a](4) to node[la,over]{$(s,t)$}(3);
  \draw[a](7) to [bend left=20] (5);
  \draw[u](7) to node[la] {$\iso$} (1);
  \draw[a] (8) to [bend right=20] node[la,swap] {$\cH G\times \cH F$} (6);
  \node[la] at ($(1)!0.5!(4) +(0,.7cm)$) {(1)};
  \node[la] at ($(7)!0.5!(2)$) {(2)};
  \node[la] at ($(1)!0.5!(3)$) {(3)};
  \node[la] at ($(4)!0.5!(6)$) {(4)};
  \node[la] at ($(2)!0.5!(3) +(0,-.7cm)$) {(5)};
\end{diagram*}
First note that $\slice{\cH G}{\cH F}$ is a pullback of the commutative square (3), and, since $\cH$ preserves pullbacks, $\cH(\dblslice GF)$ is a pullback of the outer commutative square. The commutative square (4) is obtained in two steps. First, apply \cref{lem:annoying} to the $2$-categories $\mathbbm 1\sqcup \mathbbm 1$ and $\mathbbm 2$, respectively, and to the double category $\bA$, and use the naturality of these isomorphisms with respect to the $2$-functor $\mathbbm 1\sqcup \mathbbm 1\to \mathbbm 2$ given by the inclusion at the two endpoints. Second, apply the isomorphisms
\[ \Psd(\mathbbm 1\sqcup \mathbbm 1, \cH\bA)\cong \cH\bA\times \cH\bA \ \ \text{and} \ \ \cH\PPsd(\mathbbm 1\sqcup \mathbbm 1, \cH\bA)\cong \cH(\bA\times \bA). \]
Note that the bottom isomorphism $\cH\bA\times \cH\bA\cong \cH(\bA\times \bA)$ of the square (4) is the canonical one coming from the fact that $\cH$ preserves products. Similarly, we have a canonical isomorphism $\cH\bB\times \cH\bC\cong \cH(\bB\times \bC)$ and the diagram (5) commutes. By the universal property of pullbacks, we get an isomorphism $\cH(\dblslice FG)\iso \slice{\cH F}{\cH G}$ such that the diagrams (1) and (2) commute.

The argument is similar for the case of $\SV$ since this functor also preserves pullbacks and products,  and \cref{cor:alsoannoying} holds.
\end{proof}

\section{\texorpdfstring{$2$}{2}-dimensional initiality}
\label{sec:initial}

In \cref{sec:doublebiinitial} we introduce the new notion of a double bi-initial object in a double category, which we aim to compare with that of a bi-initial object in a $2$-category in \cref{sec:dblvs2biinitial}.

Double bi-initial objects are defined by requiring that the projection double functor from the pseudo-slice double category under this object is a double trivial fibration. While this notion might seem to involve a lot of data \emph{a priori}, in fact we show that there is a straightforward characterisation of double bi-initial objects: an object $I$ is double bi-initial if and only if there is a horizontal morphism to every object, and all square boundaries whose left vertical morphism is $\vid_I$ have a unique filler. We then show that a similar result holds for bi-initial objects in a $2$-category: an object $I$ is bi-initial if and only if there is a morphism to every object, and all parallel such morphisms $I\rightrightarrows C$ have a unique $2$-morphism filler which must therefore be invertible.

The main result then says that an object of a double category $\bA$ is double bi-initial if and only if its images in the $2$-category $\cH\bA$ and $\SV\bA$ are bi-initial. In fact, we improve upon this by showing that an object $I$ is bi-initial in $\cH\bA$ if its vertical identity $\vid_I$ is bi-initial in $\SV\bA$. That is, we show that double bi-initial objects in \bA may be successfully detected purely $2$-categorically as bi-initial objects of a suitable form in $\SV\bA$.

Finally we show that in the presence of double limits of vertical morphisms in $\bA$, called \emph{tabulators}, the reverse implication also holds: the object $\vid_I$ is bi-initial in $\SV\bA$ when $I$ is bi-initial in $\cH\bA$. Taken together these results show that, in the presence of tabulators, the characterisation of double bi-initial objects is now as good as one could hope for: a double bi-initial object in a double category $\bA$ is precisely a bi-initial object in the underlying horizontal $2$-category $\cH\bA$.

Both $2$-categories and double categories have several duals, but of interest is the opposite~$\cC^{\op}$ of a $2$-category $\cC$ and the horizontal opposite $\bA^{\op}$ of a double category $\bA$. These operations agree with one another under applications of the functors $\bH$, $\cH$, and $\SV$. In particular, later we will have interest in (double) \emph{bi-terminal} objects, which are simply (double) bi-initial objects in the (horizontal) opposite. Correspondingly, all the results of this section dualise to the setting of (double) bi-terminal objects.

\subsection{Double bi-initial objects} \label{sec:doublebiinitial}

Let us first give the definition of a double bi-initial object. Recall \cref{def:pseudoslicedouble}, where we described explicitly pseudo-slice double categories.

\begin{definition}\label{def:dblbi-initial}
Let $\bA$ be a double category. An object $I$ in $\bA$ is \textbf{double bi-initial} if the projection double functor $\ainl{\dblslice I \bA}{\Pi}{\bA}$ is a double trivial fibration.
\end{definition}

Although there appears to be a lot of data in this definition, in fact we will show that there is a simpler characterisation of double bi-initial objects.

\begin{proposition}\label{prop:simpler-dblbinitial}
Let $\bA$ be a double category. An object $I$ in $\bA$ is double bi-initial if and only if the following conditions hold:
\begin{rome}
\item for every object $A\in \bA$, there is a horizontal morphism $\ainl IfA$ in $\bA$,
\item[(ii')] for every vertical morphism $\vainl AuB$ and every pair of horizontal morphisms $\ainl IfA$ and $\ainl IgB$ in $\bA$, there is a unique square $\gamma$ in $\bA$ of the form
  \begin{diagram*}[.][4]
    \node(1)[]{$I$};
    \node(2)[below of= 1]{$I$};
    \draw[proequal](1)to(2);
    \node(3)[right of= 1]{$A$};
    \node(4)[below of= 3]{$B$};
    \draw[proarrow](3)to node[la]{$u$}(4);
    \draw[a](1)to node[la]{$f$}(3);
    \draw[a](2)to node[swap,la]{$g$}(4);
    \node[la] at ($(1)!0.5!(4)$) {$\gamma$};
\end{diagram*}
\end{rome}
\end{proposition}

In order to prove this, we first elaborate the content of \cref{def:dblbi-initial}.

\begin{remark}
By expanding the definition, the projection double functor $\ainl {\dblslice I \bA}\Pi\bA$ being a double trivial fibration is equivalent to the following conditions:
\begin{rrome}
\item\label{cond:exists-obj} for every object $A\in \bA$, there is a horizontal morphism $\ainl IfA$ in $\bA$,
  \item\label{cond:exists-horiz} for every tuple of horizontal morphisms $\ainl IfA$, $\ainl I{f'}{A'}$, and $\ainl Aa{A'}$ in~$\bA$, there is a vertically invertible square $\psi$ in $\bA$
    \begin{diagram*}[,][4]
      \node(1)[]{$I$};
      \node(2)[below of= 1]{$I$};
      \draw[proequal](1)to(2);
      \node (4)[right of=2]{$A$};
      \node(5)[right of= 4]{$A'$};
      \node(3)[above of= 5]{$A'$};
      \draw[proequal](3)to(5);
      \draw[a](1)to node[la]{$f'$}(3);
      \draw[a](2)to node[swap,la]{$f$}(4);
      \draw[a](4)to node[swap,la]{$a$}(5);
    \node[la,xshift=-5pt] at ($(1)!0.5!(5)$) {$\psi$};
    \node[la,xshift=5pt] at ($(1)!0.5!(5)$) {$\vcong$};
    \end{diagram*}
\item\label{cond:exists-vert} for every vertical morphism $\vainl AuB$ in $\bA$, there is a square $\gamma$ in $\bA$
  \begin{diagram*}[,][4]
    \node(1)[]{$I$};
    \node(2)[below of= 1]{$I$};
    \draw[proequal](1)to(2);
    \node(3)[right of= 1]{$A$};
    \node(4)[below of= 3]{$B$};
    \draw[proarrow](3)to node[la]{$u$}(4);
    \draw[a](1)to node[la]{$f$}(3);
    \draw[a](2)to node[swap,la]{$g$}(4);
    \node[la] at ($(1)!0.5!(4)$) {$\gamma$};
\end{diagram*}
    \item\label{cond:square} for every tuple of squares $\gamma$, $\gamma'$, and $\alpha$ in $\bA$
\begin{diagram*}
    \node(1)[]{$I$};
    \node(2)[below of= 1]{$I$};
    \draw[proequal](1)to(2);
    \node(3)[right of= 1]{$A$};
    \node(4)[below of= 3]{$B$};
    \draw[proarrow](3)to node[la]{$u$}(4);
    \draw[a](1)to node[la]{$f$}(3);
    \draw[a](2)to node[swap,la]{$g$}(4);
    \node[la] at ($(1)!0.5!(4)$) {$\gamma$};

    \node(1)[right of= 3]{$I$};
    \node(2)[below of= 1]{$I$};
    \draw[proequal](1)to(2);
    \node(3)[right of= 1]{$A'$};
    \node(4)[below of= 3]{$B'$};
    \draw[proarrow](3)to node[la]{$u'$}(4);
    \draw[a](1)to node[la]{$f'$}(3);
    \draw[a](2)to node[swap,la]{$g'$}(4);
    \node[la] at ($(1)!0.5!(4)$) {$\gamma'$};

    \node(1)[right of= 3]{$A$};
    \node(2)[below of= 1]{$B$};
    \draw[proarrow](1)to node[swap,la]{$u$}(2);
    \node(3)[right of= 1]{$A'$};
    \node(4)[below of= 3]{$B'$};
    \draw[proarrow](3)to node[la]{$u'$}(4);
    \draw[a](1)to node[la]{$a$}(3);
    \draw[a](2)to node[swap,la]{$b$}(4);
    \node[la] at ($(1)!0.5!(4)$) {$\alpha$};
\end{diagram*}
    and for every pair of vertically invertible squares in $\bA$
    \begin{diagram*}[,][4]
      \node(1)[]{$I$};
      \node(2)[below of= 1]{$I$};
      \draw[proequal](1)to(2);
      \node (4)[right of=2]{$A$};
      \node(5)[right of= 4]{$A'$};
      \node(3)[above of= 5]{$A'$};
      \draw[proequal](3)to(5);
      \draw[a](1)to node[la]{$f'$}(3);
      \draw[a](2)to node[swap,la]{$f$}(4);
      \draw[a](4)to node[swap,la]{$a$}(5);
    \node[la,xshift=-5pt] at ($(1)!0.5!(5)$) {$\psi$};
    \node[la,xshift=5pt] at ($(1)!0.5!(5)$) {$\vcong$};

    \node(1)[right of=3]{$I$};
      \node(2)[below of= 1]{$I$};
      \draw[proequal](1)to(2);
      \node (4)[right of=2]{$B$};
      \node(5)[right of= 4]{$B'$};
      \node(3)[above of= 5]{$B'$};
      \draw[proequal](3)to(5);
      \draw[a](1)to node[la]{$g'$}(3);
      \draw[a](2)to node[swap,la]{$g$}(4);
      \draw[a](4)to node[swap,la]{$b$}(5);
    \node[la,xshift=-5pt] at ($(1)!0.5!(5)$) {$\varphi$};
    \node[la,xshift=5pt] at ($(1)!0.5!(5)$) {$\vcong$};
    \end{diagram*}
  the following pasting equality holds in $\bA$.
\begin{diagram*}
    \node(1)[]{$I$};
    \node(2)[below of= 1]{$I$};
    \draw[proequal](1)to(2);
    \node (4)[right of=2]{$A$};
    \node(5)[right of= 4]{$A'$};
    \node(3)[above of= 5]{$A'$};
    \draw[proequal](3)to(5);
    \draw[a](1)to node[la]{$f'$}(3);
    \draw[a](2)to node[over,la]{$f$}(4);
    \draw[a](4)to node[over,la]{$a$}(5);
    \node[la,xshift=-5pt] at ($(1)!0.5!(5)$) {$\psi$};
    \node[la,xshift=5pt] at ($(1)!0.5!(5)$) {$\vcong$};
    \node(1')[below of= 2]{$I$};
    \node(2')[right of= 1']{$B$};
    \node (3')[right of= 2']{$B'$};
    \draw[proequal](2)to(1');
    \draw[proarrow](4) to node[swap,la]{$u$}(2');
    \draw[proarrow](5) to node[la]{$u'$}(3');
    \draw[a](1')to node[swap,la]{$g$}(2');
    \draw[a](2')to node[swap,la]{$b$}(3');
    \node[la] at ($(2)!0.5!(2')$) {$\gamma$};
    \node[la] at ($(4)!0.5!(3')$) {$\alpha$};

    \node(1)[right of=3]{$I$};
    \node at ($(3')!0.5!(1)$) {$=$};
    \node(2)[below of= 1]{$I$};
    \draw[proequal](1)to(2);
    \node(4)[right of=2, xshift=2cm]{$B'$};
    \node(3)[above of= 4]{$A'$};
    \draw[proarrow](3)to node[la]{$u'$}(4);
    \draw[a](1)to node[la]{$f'$}(3);
    \draw[a](2)to node[over,la]{$g'$}(4);
    \node[la] at ($(1)!0.5!(4)$) {$\gamma'$};
    \node(1')[below of= 2]{$I$};
    \node(2')[right of= 1']{$B$};
    \node (3')[below of= 4]{$B'$};
    \draw[proequal](2)to(1');
    \draw[proequal](4) to (3');
    \draw[a](1')to node[swap,la]{$g$}(2');
    \draw[a](2')to node[swap,la]{$b$}(3');
    \node[la,xshift=-5pt] at ($(2)!0.5!(3')$) {$\varphi$};
    \node[la,xshift=5pt] at ($(2)!0.5!(3')$) {$\vcong$};
\end{diagram*}
  \end{rrome}
\end{remark}

Now we are ready to prove the simpler characterisation of double bi-initial objects.

\begin{proof}[\cref{prop:simpler-dblbinitial}]
We first prove that if $I$ is a double bi-initial object in $\bA$, then conditions (i) and (ii') hold. It is clear that (i) holds by \cref{cond:exists-obj}. We prove (ii').

Let $\ainl I f A$ and $\ainl I g B$ be two horizontal morphisms in $\bA$, and let $\vainl A u B$ be a vertical morphism in $\bA$. By \cref{cond:exists-vert}, there is a square $\overline{\gamma}$ in $\bA$
\begin{diagram*}[.][2]
    \node(1)[]{$I$};
    \node(2)[below of= 1]{$I$};
    \draw[proequal](1)to(2);
    \node(3)[right of= 1]{$A$};
    \node(4)[below of= 3]{$B$};
    \draw[proarrow](3)to node[la]{$u$}(4);
    \draw[a](1)to node[la]{$\overline f$}(3);
    \draw[a](2)to node[swap,la]{$\overline g$}(4);
    \node[la] at ($(1)!0.5!(4)$) {$\overline \gamma$};
\end{diagram*}
By \cref{cond:exists-horiz} applied to $(\overline f,f, \id_A)$ and $(g,\overline g, \id_B)$, there are vertically invertible squares $\psi$ and $\varphi$ in $\bA$ as depicted below.
\begin{diagram*}
    \node(1)[]{$I$};
    \node(2)[below of= 1]{$I$};
    \draw[proequal](1)to(2);
    \node(3)[right of= 1]{$A$};
    \node(4)[below of= 3]{$A$};
    \draw[proequal](3)to (4);
    \draw[a](1)to node[la]{$f$}(3);
    \draw[a](2)to node[swap,la]{$\overline f$}(4);
    \node[la,xshift=-5pt] at ($(1)!0.5!(4)$) {$\psi$};
    \node[la,xshift=5pt] at ($(1)!0.5!(4)$) {$\vcong$};

    \node(1)[right of=3]{$I$};
    \node(2)[below of= 1]{$I$};
    \draw[proequal](1)to(2);
    \node(3)[right of= 1]{$B$};
    \node(4)[below of= 3]{$B$};
    \draw[proequal](3)to (4);
    \draw[a](1)to node[la]{$\overline g$}(3);
    \draw[a](2)to node[swap,la]{$g$}(4);
    \node[la,xshift=-5pt] at ($(1)!0.5!(4)$) {$\varphi$};
    \node[la,xshift=5pt] at ($(1)!0.5!(4)$) {$\vcong$};
\end{diagram*}
Then we set $\gamma$ to be the following composite
\begin{diagram*}
    \node(1)[]{$I$};
    \node(2)[below of= 1]{$I$};
    \draw[proequal](1)to(2);
    \node(3)[right of= 1]{$A$};
    \node(4)[below of= 3]{$B$};
    \draw[proarrow](3)to node[la]{$u$}(4);
    \draw[a](1)to node[la]{$f$}(3);
    \draw[a](2)to node[swap,la]{$g$}(4);
    \node[la] at ($(1)!0.5!(4)$) {$\gamma$};

    \node(2)[right of=3]{$I$};
    \node at ($(4)!0.5!(2)$) {$=$};
    \node(1)[above of=2]{$I$};
    \draw[proequal](1)to(2);
    \node(3)[right of= 1]{$A$};
    \node(4)[below of= 3]{$A$};
    \draw[proequal](3)to (4);
    \draw[a](1)to node[la]{$f$}(3);
    \draw[a](2)to node[over,la]{$\overline f$}(4);
    \node[la,xshift=-5pt] at ($(1)!0.5!(4)$) {$\psi$};
    \node[la,xshift=5pt] at ($(1)!0.5!(4)$) {$\vcong$};

    \node(5)[below of=2]{$I$};
    \node(6)[right of=5]{$B$};
    \draw[proequal](2) to (5);
    \draw[proarrow](4)to node[la]{$u$}(6);
    \draw[a](5)to node[la,over]{$\overline g$}(6);
    \node[la] at ($(2)!0.5!(6)$) {$\overline \gamma$};
    \node(7)[below of=5]{$I$};
    \node(8)[right of=7]{$B$};
    \draw[proequal](5)to (7);
    \draw[proequal](6)to (8);
    \draw[a](7)to node[swap,la]{$g$}(8);
    \node[la,xshift=-5pt] at ($(5)!0.5!(8)$) {$\varphi$};
    \node[la,xshift=5pt] at ($(5)!0.5!(8)$) {$\vcong$};
\end{diagram*}
which proves the existence.

Now suppose $\gamma'$ is another such square in $\bA$
\begin{diagram*}[.][2]
    \node(1)[]{$I$};
    \node(2)[below of= 1]{$I$};
    \draw[proequal](1)to(2);
    \node(3)[right of= 1]{$A$};
    \node(4)[below of= 3]{$B$};
    \draw[proarrow](3)to node[la]{$u$}(4);
    \draw[a](1)to node[la]{$f$}(3);
    \draw[a](2)to node[swap,la]{$g$}(4);
    \node[la] at ($(1)!0.5!(4)$) {$\gamma'$};
\end{diagram*}
By applying \cref{cond:square} to the squares $(\gamma,\gamma',\id_u)$ and to the vertically invertible squares $(\vid_f,\vid_g)$, we directly get that $\gamma=\gamma'$ from the pasting equality.

Now, conversely, it is straightforward to see that (ii') implies \cref{cond:exists-horiz,cond:exists-vert,cond:square}, since there is a unique square with each given boundary whose left vertical morphism is $\vid_I$. Note that the square in \cref{cond:exists-horiz} is indeed vertically invertible since every square from $\vid_I$ to another vertical identity is vertically invertible by (ii').
\end{proof}

\subsection{Double bi-initial objects vs bi-initial objects} \label{sec:dblvs2biinitial}

We now want to compare the notion of double bi-initial objects in a double category with the notion of bi-initial objects in related $2$-categories. We first recall the definition of a bi-initial object, which uses the notion of a pseudo-slice $2$-category as elaborated in \cref{def:pseudoslice2}.

\begin{definition} \label{def:bi-initial}
    Let $\cA$ be a $2$-category. An object $I\in \cA$ is \textbf{bi-initial} if the projection $2$-functor $\ainl{\slice I \cA}{\pi}{\cA}$ is a trivial fibration.
\end{definition}

\begin{proposition}\label{prop:simpler-bi-initial}
 Let $\cA$ be a $2$-category, and $I\in \cA$ be an object. Then the object $I$ is bi-initial in \cA if and only if, for every object $A\in \cA$, the unique functor $\cA(I,A)\xarr{\eqv}\mathbbm 1$ is part of an equivalence of categories. More precisely, this means that
 \begin{rome}
 \item for every object $A\in \cA$, there is a morphism $\ainl{I}{f}{A}$ in $\cA$,
 \item[(ii')] for every pair of morphisms $\ainl IfA$ and $\ainl I{f'}A$, there is a unique $2$-morphism $\ninl {f'}\alpha{f}$.
 \end{rome}
\end{proposition}

In order to prove this, we first elaborate the content of \cref{def:bi-initial}.

\begin{remark}
By expanding the definition, the projection $2$-functor $\ainl {\slice I \cA}{Pi}\cA$ being a trivial fibration is equivalent to the following conditions:
\begin{rrome}
  \item\label{cond:bi-in_obj} for every object $A\in \cA$, there is a morphism $\ainl{I}{f}{A}$ in $\cA$,
  \item\label{cond:bi-in_mor} for every tuple of morphisms $\ainl{I}{f}{A}$, $\ainl{I}{f'}{A'}$, and $\ainl{A}{a}{A'}$ in $\cA$, there is a $2$-isomorphism $\psi$ in $\cA$
 \begin{diagram*}[,][3][node distance=1.8cm]
  \node(1)[]{$I$};
  \node(2)[right of =1]{$A$};
  \node(3)[below of =2]{$A'$};
  \draw[a](1) to node[la]{$f$} (2);
  \draw[a](1) to node(x)[swap,la]{$f'$} (3);
  \draw[a](2) to node[la]{$a$} (3);
  \celli[la,shrink]{x}{2}{$\psi$};
\end{diagram*}
    \item \label{cond:bi-in_2mor} for every $2$-morphism $\ninl{a}{\alpha}{b}$ in $\cA$, and every pair of $2$-isomorphisms $\niinl f \psi{af'}$ and $\niinl f {\varphi}{bf'}$ the following pasting equality holds in $\cA$.
\begin{diagram*}[][][node distance=1.8cm]
\node(1)[]{$I$};
    \node(2)[right of= 1]{$A$};
    \node(3)[below of= 2]{$A'$};
    \draw[a](2)to node(a)[la,over]{$a$}(3);
    \draw[a, bend left=60](2)to node(b)[la,right]{$b$}(3);
    \draw[a](1)to node(f)[la,swap]{$f'$}(3);
    \draw[a](1)to node(x')[la]{$f$}(2);
    \cell[la,xshift=-.1cm][n][.42]{a}{b}{$\alpha$};
    \celli[la,shrink]{f}{2}{$\psi$};

    \node(4)[right of= a,xshift=.2cm]{$=$};

    \node(1)[right of= 2, xshift=.75cm]{$I$};
    \node(2)[right of= 1]{$A$};
    \node(3)[below of= 2]{$A'$};
    \draw[a](2)to node(a)[la]{$b$}(3);
    \draw[a](1)to node[la]{$f$}(2);
    \draw[a](1)to node(y)[swap,la]{$f'$}(3);
    \celli[la,shrink]{y}{2}{$\varphi$};
  \end{diagram*}
\end{rrome}
\end{remark}

\begin{proof}[\cref{prop:simpler-bi-initial}]
We first prove that if $I$ is a bi-initial object in \cA then conditions (i) and (ii') hold. It is clear that (i) holds by \cref{cond:bi-in_obj}. We prove (ii'). Given two morphisms $\ainl IfA$ and $\ainl I{f'}A$ in $\cA$, by applying \cref{cond:bi-in_mor} to the tuple $(f,f',\id_A)$, there is a unique $2$-isomorphism $\niinl {f'}\psi f$.

Now, conversely, it is straightforward to see that (ii’) implies \cref{cond:bi-in_mor,cond:bi-in_2mor}, since there is a unique $2$-morphism between every pair of morphisms from $I$ to any object. Note that every such $2$-morphism is in fact invertible by (ii’).
\end{proof}

Since pseudo-slices are special cases of pseudo-commas, \cref{lem:h-sv-general-commas} may be specialised in this context to give the following result.

\begin{corollary} \label{lem:iso_between_slices}
Let $\bA$ be a double category, and $I\in \bA$ be an object. Then there are canonical isomorphisms of $2$-categories as in the following commutative triangles.
\begin{diagram*}
\node[](1) {$\cH(\dblslice{I}{\bA})$};
\node[right of=1,xshift=1cm](3) {$\slice{I}{\cH\bA}$};
\node(2) at ($(1)!0.5!(3)-(0,1.5cm)$) {$\cH\bA$};
\draw[a] (1) to node[la,swap] {$\cH \Pi$} (2);
\draw[a] (1) to node[la] {$\iso$} (3);
\draw[a] (3) to node[la] {$\pi$} (2);

\node[right of=3,xshift=2cm](1) {$\SV(\dblslice{I}{\bA})$};
\node[right of=1,xshift=1cm](3) {$\slice{\vid_I}{\SV\bA}$};
\node(2) at ($(1)!0.5!(3)-(0,1.5cm)$) {$\SV\bA$};
\draw[a] (1) to node[la,swap] {$\SV \Pi$} (2);
\draw[a] (1) to node[la] {$\iso$} (3);
\draw[a] (3) to node[la] {$\pi$} (2);
\end{diagram*}
\end{corollary}

\begin{proof}
This directly follows from \cref{lem:h-sv-general-commas}, by taking $\bB=\mathbbm 1$, $\bC=\bA$, $\ainl{\mathbbm 1}{F=I}{\bA}$ and $\ainl{\bA}{G=\id_{\bA}}{\bA}$. Note that $\ainl{\SV(\mathbbm 1)=\mathbbm 1}{\SV I=\vid_I}{\SV\bA}$.
\end{proof}

With this result and the fact that double trivial fibrations are exactly the double functors whose images under $\cH$ and $\SV$ are trivial fibrations, we may give a $2$-categorical characterisation of double bi-initial objects by leveraging this fact as follows.

\begin{theorem}\label{thm:initial}
    Let $\bA$ be a double category, and $I\in \bA$ be an object. The following statements are equivalent.
    \begin{rome}
      \item The object $I\in \bA$ is double bi-initial.
      \item The corresponding objects $I\in \cH\bA$ and $\vid_I\in \SV\bA$ are bi-initial.
      \item The corresponding object $\vid_I\in \SV\bA$ is bi-initial.
    \end{rome}
\end{theorem}

\begin{proof}
  By definition, an object $I\in \bA$ is double bi-initial if and only if the projection double functor $\ainl{\dblslice{I}{\bA}}{\Pi}{\bA}$ is a double trivial fibration. By \cref{prop:dblvs2trivfib}, this is equivalent to saying that the induced $2$-functors $\cH \Pi$ and $\SV\Pi$ are trivial fibrations, or equivalently that $\SV\Pi$ alone is a trivial fibration. By \cref{lem:iso_between_slices}, this holds if and only if the projection $2$-functors $\ainl{\slice{I}{\cH\bA}}{\pi}{\cH\bA}$ and $\ainl{\slice{\vid_I}{\SV\bA}}{\pi}{\SV\bA}$ are trivial fibrations, or equivalently that $\ainl{\slice{\vid_I}{\SV\bA}}{\pi}{\SV\bA}$ alone is a trivial fibration. By definition of a bi-initial object, this holds if and only if the objects $I\in \cH\bA$ and $\vid_I\in \SV\bA$ are bi-initial, or equivalently that $\vid_I\in \SV\bA$ is bi-initial.
\end{proof}

\begin{remark}\label{REMARK}
 This theorem served as the initial motivation for the definition of the functor $\SV$ whose role is so central in this paper.

 Observe that double bi-initial objects in a double category $\bA$ have two aspects to their \emph{weak} universal properties: one concerning objects and one concerning vertical morphisms. The former is entirely horizontal in nature and so is completely captured by the underlying horizontal $2$-category $\cH\bA$. The latter, despite concerning vertical morphisms, does not in fact need the full strength of vertical composition in $\bA$ to be expressed. Indeed, except for vertical composition by squares with trivial boundary of the form of \cref{2morinSV}, this aspect of the weak universal property is also somehow horizontal. That is to say, the underlying \emph{horizontal} $2$-category $\cH\PPsd(\bV\mathbbm 2,\bA)$ is precisely the setting in which to capture this last data as it has vertical morphisms as objects and understands horizontal compositions of general squares. But this $2$-category $\cH\PPsd(\bV\mathbbm 2,\bA)$ is exactly our $\SV\bA$!
\end{remark}

\subsection{Double bi-initial objects and tabulators}

We have seen that double bi-initial objects may be detected through purely $2$-categorical means. In this section we show that a substantial simplification of the $2$-categorical criteria is possible when the double category in question has \emph{tabulators}. These correspond to double limits of vertical morphisms and were introduced by Grandis and Par\'e in \cite[\S 5.3]{GraPar1999}. In the presence of tabulators our result is as strong as possible: double bi-initial objects are precisely bi-initial objects in the underlying horizontal $2$-category.

\begin{definition} \label{def:tabulator}
  Let $\bA$ be a double category, and $\vainl{A}{u}{B}$ be a vertical morphism in~$\bA$. A \textbf{tabulator} of $u$ is a double limit of the double functor $u\colon \bV\mathbbm 2\to \bA$, where $\bV\mathbbm 2$ is the double category free on a vertical morphism. In other words, it is a pair $(\top u,\tau_u)$ of an object $\top u\in \bA$ together with a square $\sq{\tau_u}{p}{q}{\vid_{\top u}}{u}$ in $\bA$ satisfying the following universal properties.
  \begin{drome}
      \item \label{cond:firstup-tab} For every square $\sq{\gamma}{f}{g}{\vid_I}{u}$ in $\bA$, there is a unique horizontal morphism $t\colon I\to \top u$ in $\bA$ such that the following pasting equality holds.
\begin{diagram*}
\node[](1) {$I$};
\node[below of =1](2) {$I$};
\node[right of=1](3) {$A$};
\node[below of =3](4) {$B$};
\draw[a] (1) to node[above,la] {$f$} (3);
\draw[a] (2) to node[below,la] {$g$} (4);
\draw[proequal] (1) to (2);
\draw[proarrow] (3) to node[right,la] {$u$} (4);
\node[la] at ($(1)!0.5!(4)$) {$\gamma$};

\node[right of=3,xshift=.75cm](6) {$I$};
\node at ($(4)!0.5!(6)$) {$=$};
\node[below of=6](7) {$I$};
\node[right of =6](1) {$\top u$};
\node[below of =1](2) {$\top u$};
\node[right of=1](3) {$A$};
\node[below of =3](4) {$B$};
\draw[a] (6) to node[above,la] {$t$} (1);
\draw[a] (7) to node[below,la] {$t$} (2);
\draw[a] (1) to node[above,la] {$p$} (3);
\draw[a] (2) to node[below,la] {$q$} (4);
\draw[proequal] (6) to (7);
\draw[proequal] (1) to (2);
\draw[proarrow] (3) to node[right,la] {$u$} (4);
\node[la] at ($(6)!0.5!(2)$) {$\vid_t$};
\node[la] at ($(1)!0.5!(4)$) {$\tau_u$};
\end{diagram*}
    \item \label{cond:secondup-tab} For every tuple of squares $\sq{\gamma}{f}{g}{\vid_I}{u}$, $\sq{\gamma'}{f'}{g'}{\vid_{I'}}{u}$, $\sq{\theta_0}{f'}{f}{v}{\vid_A}$ and $\sq{\theta_1}{g'}{g}{v}{\vid_B}$ in $\bA$ satisfying the following pasting equality,
\begin{diagram*}
\node[](6) {$I'$};
\node[right of=6](7) {$A$};
\node[below of=6](1) {$I$};
\node[below of =1](2) {$I$};
\node[right of=1](3) {$A$};
\node[below of =3](4) {$B$};
\draw[a] (6) to node[above,la] {$f'$} (7);
\draw[a] (1) to node[over,la] {$f$} (3);
\draw[a] (2) to node[below,la] {$g$} (4);
\draw[proequal] (7) to (3);
\draw[proequal] (1) to (2);
\draw[proarrow] (3) to node[right,la] {$u$} (4);
\draw[proarrow] (6) to node[left,la] {$v$} (1);
\node[la] at ($(6)!0.5!(3)$) {$\theta_0$};
\node[la] at ($(1)!0.5!(4)$) {$\gamma$};

\node[right of=7,xshift=.75cm](6) {$I'$};
\node at ($(4)!0.5!(6)$) {$=$};
\node[right of=6](7) {$A$};
\node[below of=6](1) {$I'$};
\node[below of =1](2) {$I$};
\node[right of=1](3) {$B$};
\node[below of =3](4) {$B$};
\draw[a] (6) to node[above,la] {$f'$} (7);
\draw[a] (1) to node[over,la] {$g'$} (3);
\draw[a] (2) to node[below,la] {$g$} (4);
\draw[proequal] (3) to (4);
\draw[proequal] (6) to (1);
\draw[proarrow] (7) to node[right,la] {$u$} (3);
\draw[proarrow] (1) to node[left,la] {$v$} (2);
\node[la] at ($(6)!0.5!(3)$) {$\gamma'$};
\node[la] at ($(1)!0.5!(4)$) {$\theta_1$};
\end{diagram*}
   there is a unique square $\sq{\theta}{t'}{t}{v}{\vid_{\top u}}$, where $t\colon I\to \top u$ and $t'\colon I'\to \top u$ are the unique horizontal morphisms given by (i) applied to $\gamma$ and $\gamma'$ respectively, such that $\theta$ satisfies the following pasting equalities.
\end{drome}
\begin{diagram*}
\node[](1) {$I'$};
\node[below of =1](2) {$I$};
\node[right of=1](3) {$A$};
\node[below of =3](4) {$A$};
\draw[a] (1) to node[above,la] {$f'$} (3);
\draw[a] (2) to node[below,la] {$f$} (4);
\draw[proarrow] (1) to node[left,la] {$v$} (2);
\draw[proequal] (3) to (4);
\node[la] at ($(1)!0.5!(4)$) {$\theta_0$};

\node[right of=3](6) {$I'$};
\node at ($(4)!0.5!(6)$) {$=$};
\node[below of=6](7) {$I$};
\node[right of =6](1) {$\top u$};
\node[below of =1](2) {$\top u$};
\node[right of=1](3) {$A$};
\node[below of =3](4) {$A$};
\draw[a] (6) to node[above,la] {$t'$} (1);
\draw[a] (7) to node[below,la] {$t$} (2);
\draw[a] (1) to node[above,la] {$p$} (3);
\draw[a] (2) to node[below,la] {$p$} (4);
\draw[proequal] (3) to (4);
\draw[proequal] (1) to (2);
\draw[proarrow] (6) to node[left,la] {$v$} (7);
\node[la] at ($(6)!0.5!(2)$) {$\theta$};
\node[la] at ($(1)!0.5!(4)$) {$\vid_p$};

\node[right of=3,xshift=.5cm](1) {$I'$};
\node[below of =1](2) {$I$};
\node[right of=1](3) {$B$};
\node[below of =3](4) {$B$};
\draw[a] (1) to node[above,la] {$g'$} (3);
\draw[a] (2) to node[below,la] {$g$} (4);
\draw[proarrow] (1) to node[left,la] {$v$} (2);
\draw[proequal] (3) to (4);
\node[la] at ($(1)!0.5!(4)$) {$\theta_1$};

\node[right of=3](6) {$I'$};
\node at ($(4)!0.5!(6)$) {$=$};
\node[below of=6](7) {$I$};
\node[right of =6](1) {$\top u$};
\node[below of =1](2) {$\top u$};
\node[right of=1](3) {$B$};
\node[below of =3](4) {$B$};
\draw[a] (6) to node[above,la] {$t'$} (1);
\draw[a] (7) to node[below,la] {$t$} (2);
\draw[a] (1) to node[above,la] {$q$} (3);
\draw[a] (2) to node[below,la] {$q$} (4);
\draw[proequal] (3) to (4);
\draw[proequal] (1) to (2);
\draw[proarrow] (6) to node[left,la] {$v$} (7);
\node[la] at ($(6)!0.5!(2)$) {$\theta$};
\node[la] at ($(1)!0.5!(4)$) {$\vid_q$};
\end{diagram*}
 We say that $\bA$ \textbf{has tabulators} if there is a tabulator for each vertical morphism $u$ in $\bA$.
\end{definition}

\begin{theorem} \label{prop:tabulators-biinitial}
  Let \bA be a double category with tabulators, and $I\in \bA$ be an object. Then the following statements are equivalent.
  \begin{rome}
    \item The object $I$ is double bi-initial in $\bA$.
    \item The object $I$ is bi-initial in $\cH\bA$.
  \end{rome}
\end{theorem}

\begin{proof}
  If $I$ is double bi-initial in $\bA$, then, by \cref{thm:initial}, $I$ is bi-initial in $\cH\bA$.

 Now suppose that $I$ is bi-initial in $\cH\bA$. We prove that $I$ satisfies \cref{def:dblbi-initial} (i-iv). First note that \cref{cond:bi-in_obj} and (ii) applied to $I\in \cH\bA$ correspond to \cref{cond:exists-obj} and (ii) applied to $I\in\bA$. Therefore, it remains to show \cref{cond:exists-vert} and (iv). Let $\vainl AuB$ be a vertical morphism and let $(\top u,\tau_u)$ be the tabulator of $u$. By \cref{cond:bi-in_obj} applied to the object $\top u$, there is a horizontal morphism $t\colon I\to \top u$. By the first universal property of tabulators, we get a square in $\bA$
 \begin{diagram*}
\node[](1) {$I$};
\node[below of =1](2) {$I$};
\node[right of=1](3) {$A$};
\node[below of =3](4) {$B$};
\draw[a] (1) to node[above,la] {$f$} (3);
\draw[a] (2) to node[below,la] {$g$} (4);
\draw[proequal] (1) to (2);
\draw[proarrow] (3) to node[right,la] {$u$} (4);
\node[la] at ($(1)!0.5!(4)$) {$\gamma$};
\end{diagram*}
as desired. This proves \cref{cond:exists-vert}.

Now suppose that we have squares in $\bA$
\begin{diagram*}
    \node(1)[]{$I$};
    \node(2)[below of= 1]{$I$};
    \draw[proequal](1)to(2);
    \node(3)[right of= 1]{$A$};
    \node(4)[below of= 3]{$B$};
    \draw[proarrow](3)to node[la]{$u$}(4);
    \draw[a](1)to node[la]{$f$}(3);
    \draw[a](2)to node[swap,la]{$g$}(4);
    \node[la] at ($(1)!0.5!(4)$) {$\gamma$};

    \node(1)[right of= 3]{$I$};
    \node(2)[below of= 1]{$I$};
    \draw[proequal](1)to(2);
    \node(3)[right of= 1]{$A'$};
    \node(4)[below of= 3]{$B'$};
    \draw[proarrow](3)to node[la]{$u'$}(4);
    \draw[a](1)to node[la]{$f'$}(3);
    \draw[a](2)to node[swap,la]{$g'$}(4);
    \node[la] at ($(1)!0.5!(4)$) {$\gamma'$};

    \node(1)[right of= 3]{$A$};
    \node(2)[below of= 1]{$B$};
    \draw[proarrow](1)to node[swap,la]{$u$}(2);
    \node(3)[right of= 1]{$A'$};
    \node(4)[below of= 3]{$B'$};
    \draw[proarrow](3)to node[la]{$u'$}(4);
    \draw[a](1)to node[la]{$a$}(3);
    \draw[a](2)to node[swap,la]{$b$}(4);
    \node[la] at ($(1)!0.5!(4)$) {$\alpha$};
\end{diagram*}
and suppose $(\top u,\tau_u)$ and $(\top u',\tau_{u'})$ are tabulators for $u$ and $u'$ respectively. By the first universal property of tabulators, the squares $\gamma$ and $\gamma'$ uniquely correspond to horizontal morphisms $t\colon I\to \top u$ and $t'\colon I\to \top u'$ respectively. Moreover, the square $\alpha$ uniquely corresponds to a horizontal morphism $\top \alpha\colon \top u\to \top u'$. By applying \cref{cond:bi-in_mor} to $t\colon I\to \top u$, $t'\colon I\to \top u'$, and $\top \alpha\colon \top u\to \top u'$, we get a square $\theta$ in $\bA$ of the form
 \begin{diagram*}[.][4]
      \node(1)[]{$I$};
      \node(2)[below of= 1]{$I$};
      \draw[proequal](1)to(2);
      \node (4)[right of=2]{$\top u$};
      \node(5)[right of= 4]{$\top u'$};
      \node(3)[above of= 5]{$\top u'$};
      \draw[proequal](3)to(5);
      \draw[a](1)to node[la]{$t'$}(3);
      \draw[a](2)to node[swap,la]{$t$}(4);
      \draw[a](4)to node[swap,la]{$\top \alpha$}(5);
    \node[la,xshift=-5pt] at ($(1)!0.5!(5)$) {$\theta$};
    \node[la,xshift=5pt] at ($(1)!0.5!(5)$) {$\vcong$};
\end{diagram*}
By the second universal property of tabulators, this uniquely corresponds to squares $\psi\coloneqq \theta_0$ and $\varphi\coloneqq \theta_1$ satisfying the following pasting.
\begin{diagram*}
    \node(1)[]{$I$};
    \node(2)[below of= 1]{$I$};
    \draw[proequal](1)to(2);
    \node (4)[right of=2]{$A$};
    \node(5)[right of= 4]{$A'$};
    \node(3)[above of= 5]{$A'$};
    \draw[proequal](3)to(5);
    \draw[a](1)to node[la]{$f'$}(3);
    \draw[a](2)to node[over,la]{$f$}(4);
    \draw[a](4)to node[over,la]{$a$}(5);
    \node[la,xshift=-5pt] at ($(1)!0.5!(5)$) {$\psi$};
    \node[la,xshift=5pt] at ($(1)!0.5!(5)$) {$\vcong$};
    \node(1')[below of= 2]{$I$};
    \node(2')[right of= 1']{$B$};
    \node (3')[right of= 2']{$B'$};
    \draw[proequal](2)to(1');
    \draw[proarrow](4) to node[swap,la]{$u$}(2');
    \draw[proarrow](5) to node[la]{$u'$}(3');
    \draw[a](1')to node[swap,la]{$g$}(2');
    \draw[a](2')to node[swap,la]{$b$}(3');
    \node[la] at ($(2)!0.5!(2')$) {$\gamma$};
    \node[la] at ($(4)!0.5!(3')$) {$\alpha$};

    \node(1)[right of=3]{$I$};
    \node at ($(3')!0.5!(1)$) {$=$};
    \node(2)[below of= 1]{$I$};
    \draw[proequal](1)to(2);
    \node(4)[right of=2, xshift=2cm]{$B'$};
    \node(3)[above of= 4]{$A'$};
    \draw[proarrow](3)to node[la]{$u'$}(4);
    \draw[a](1)to node[la]{$f'$}(3);
    \draw[a](2)to node[over,la]{$g'$}(4);
    \node[la] at ($(1)!0.5!(4)$) {$\gamma'$};
    \node(1')[below of= 2]{$I$};
    \node(2')[right of= 1']{$B$};
    \node (3')[below of= 4]{$B'$};
    \draw[proequal](2)to(1');
    \draw[proequal](4) to (3');
    \draw[a](1')to node[swap,la]{$g$}(2');
    \draw[a](2')to node[swap,la]{$b$}(3');
    \node[la,xshift=-5pt] at ($(2)!0.5!(3')$) {$\varphi$};
    \node[la,xshift=5pt] at ($(2)!0.5!(3')$) {$\vcong$};
\end{diagram*}
Note that $\psi$ and $\varphi$ are the unique squares for the tuples $(f,f',a)$ and $(g,g',b)$ respectively; see \cref{prop:simpler-bi-initial} (ii'). This shows \cref{cond:square}.
\end{proof}

\begin{corollary} \label{cor:bi-in-in-HA-VA-tab}
  Let \bA be a double category with tabulators, and $I\in \bA$ be an object. Then the object $I$ is bi-initial in~$\cH\bA$ if and only if the corresponding object $\vid_I$ is bi-initial in $\SV\bA$.
\end{corollary}

\begin{proof}
  By \cref{prop:tabulators-biinitial}, $I$ is bi-initial in $\cH\bA$ if and only if $I$ is double bi-initial in $\bA$. By \cref{thm:initial}, this holds if and only if $\vid_I$ is bi-initial in $\SV\bA$.
\end{proof}

\section{Bi-representations of normal pseudo-functors}
\label{sec:bi-rep-pseudo}

In \cref{sec:generalcase} we state and prove our main result characterising bi-representations of normal pseudo-functors $\ainl{\cC^{\op}}{F}{\Cat}$ as various sorts of bi-initial objects, where $\Cat$ is the $2$-category of categories, functors, and natural transformations. We give two flavours of such a theorem, one stated in the language of double categories, and the other stated completely in terms of $2$-categories. The former  predictably relies on the double category of elements of $F$ construction, but in the latter case, we will define from the data of a normal pseudo-functor $F$ not only the 2-category of \emph{elements} of $F$, but also a $2$-category of \emph{morphisms} of $F$. Moreover, by specialising \cref{thm:initial}, we see that the latter $2$-category subsumes the former for this purpose: bi-representations of a normal pseudo-functor $\ainl{\cC^{\op}}{F}{\Cat}$ are precisely bi-initial objects of a particular form in the $2$-category of morphisms of $F$.

We then show in \cref{sec:tensorby2case} that, when the $2$-category $\cC$ has tensors by $\mathbbm 2$ and the normal pseudo-functor $\ainl{\cC^{\op}}{F}{\Cat}$ preserves them, the expected characterisation actually holds: bi-representations of $F$ are now precisely bi-initial objects in the $2$-category of elements of $F$.

\subsection{The general case} \label{sec:generalcase}

Let us begin by defining the central objects at issue.

\begin{definition} \label{def:birep}
  Let \cC be a 2-category, and $\ainl{\cC^{\op}}{F}{\Cat}$ be a normal pseudo-functor. A~\textbf{bi-representation} of $F$ is a pair $(I,\rho)$ of an object $I\in\cC$ and a pseudo-natural adjoint equivalence $\neinl{\cC(-,I)}{\rho_{-}}{F}$, i.e., an adjoint equivalence in the $2$-category $\Psd(\cC^{\op},\Cat)$.
\end{definition}

\begin{remark}
  Recall that an equivalence in a $2$-category can always be promoted to an \emph{adjoint} equivalence (see, e.g.~\cite[Lemma 2.1.11]{RiehlVerity}). Therefore, by requiring the pseudo-natural equivalence in \cref{def:birep} to be adjoint, we do not lose any generality while simultaneously making the data easier to handle in the forthcoming proofs.
\end{remark}

To the data of such a normal pseudo-functor $\cC^{\op}\to\Cat$ we will associate a \emph{double category of elements}. This double category will play an analogous role to the classical category of elements in detecting representations.

\begin{definition}\label{def:eell}
   Let \cC be a 2-category, and $\ainl{\cC^{\op}}{F}{\Cat}$ be a normal pseudo-functor. The \textbf{double category of elements} $\eell(F)$ of $F$ is defined to be the pseudo-slice double category $\dblslice{\mathbbm 1}{\bH F}$ induced by the cospan
   \[ \mathbbm 1\stackrel{\mathbbm 1}{\longrightarrow} \bH\Cat \xleftarrow{\bH F} \bH\cC^{\op} \ . \]

  More explicitly, it is the double category whose
\begin{drome}
  \item objects are pairs $(C,x)$ of an object $C\in\cC$ and a functor $\ainl{\mathbbm 1}x{FC}$, i.e., an object $x\in FC$,
  \item horizontal morphisms $\ainl{(C',x')}{(c,\psi)}{(C,x)}$ comprise the data of a morphism $\ainl Cc{C'}$ of \cC and a natural isomorphism $\psi$ of the form
    \begin{diagram*}[,][3][node distance=1.8cm]
  \node(1)[right of =1, xshift=.5cm]{$\mathbbm{1}$};
  \node(2)[right of =1]{$FC'$};
  \node(3)[below of =2]{$FC$};
  \draw[a](1) to node[la]{$x'$} (2);
  \draw[a](1) to node(x)[swap,la]{$x$} (3);
  \draw[a](2) to node[la]{$Fc$} (3);
  \celli[la,shrink]{x}{2}{$\psi$};
  \end{diagram*}
  i.e., an isomorphism $\aiinl {x}\psi{(Fc)x'}$ in $FC$,
\item vertical morphisms $\vainl{(C,x)}{\alpha}{(C,y)}$ are natural transformations $\ninl x\alpha y$ of functors $\ainl{\mathbbm 1}{x,y}{FC}$, i.e., morphisms $\ainl x\alpha y$ in $FC$,
\item\label{def:eellsquare} squares $\sq{\gamma}{(c,\psi)}{(d,\varphi)}{\alpha'}{\alpha}$ comprise the data of a $2$-morphism $\ninl c\gamma{d}$ of \cC, as displayed below-left, which satisfies the below-right pasting equality,
  \begin{diagram*}[][][node distance=1.8cm]
    \node(1)[]{$C$};
    \node(2)[below of= 1]{$C'$};
    \draw[a,bend right=30](1) to node(b)[la,swap]{$c$}(2);
    \draw[a,bend left=30](1) to node(a)[la]{$d$}(2);
    \cell[la][n][.5]{b}{a}{$\gamma$};

    \node(1)[right of= 1, xshift=1cm,minimum size=0.7cm]{$\mathbbm{1}$};
    \node(2)[right of= 1]{$FC'$};
    \node(3)[below of= 2]{$FC$};
    \draw[a](2)to node(a)[la,over]{$Fc$}(3);
    \draw[a, bend left=60](2)to node(b)[la,right]{$Fd$}(3);
    \draw[a](1)to node(f)[la,swap]{$x$}(3);
    \draw[a](1)to node(x')[la,over]{$x'$}(2);     \cell[la,xshift=-.1cm][n][.45]{a}{b}{$F\gamma$};
    \draw[a,bend left=60](1)to node(y')[la,pos=0.44]{$y'$}(2);
    \celli[la,shrink]{f}{2}{$\psi$};
    \cell[la,swap][n][0.4]{x'}{y'}{$\alpha'$};

    \node(4)[right of= a]{$=$};

    \node(1)[right of= 2, xshift=.75cm]{$\mathbbm{1}$};
    \node(2)[right of= 1]{$FC'$};
    \node(3)[below of= 2]{$FC$};
    \draw[a](2)to node(a)[la]{$Fd$}(3);
    \draw[a](1)to node[la]{$y'$}(2);
    \draw[a](1)to node(y)[over,la]{$y$}(3);
    \draw[a,bend right=50](1)to node(x)[swap,la]{$x$}(3);
    \cell[la][n][0.6]{x}{y}{$\alpha$};
    \celli[la,shrink][n][0.4]{y}{2}{$\varphi$};
  \end{diagram*}
  i.e., the following diagram in $FC$ is commutative.
  \begin{diagram*}[][][node distance=1.2cm]
    \node(1){$x$};
    \node(2)[below= of 1]{$(Fc)x'$};
    \node(3)[right= of 2]{$(Fc)y'$};
    \node(4)[right= of 3]{$(Fd)y'$};
    \node(5) at (1-|4) {$y$};
    \draw[a](1) to node[la]{$\alpha$} (5);
    \draw[a](1) to node[swap,la]{$\psi$} node[la]{$\iso$} (2);
    \draw[a](2) to node[la,swap]{$(Fc)\alpha'$} (3);
    \draw[a](3) to node[la,swap]{$(F\gamma)_{y'}$} (4);
    \draw[a](5) to node[la]{$\varphi$} node[swap,la]{$\iso$} (4);
  \end{diagram*}
\end{drome}
\end{definition}

Much like in the $1$-dimensional case, from a normal pseudo-functor $\ainl {\cC^{\op}}F\Cat$ we are able to construct the $2$-category of elements $\el(F)$ of $F$, but new here is the $2$-category of \emph{morphisms} of $F$. As we shall see, the joint properties of these $2$-categories may be leveraged to successfully characterise bi-representations.

\begin{definition}\label{def:el-mor}
  Let \cC be a $2$-category, and $\ainl {\cC^{\op}}{F}{\Cat}$ be a normal pseudo-functor. We define the following two $2$-categories associated to $F$.
  \begin{itemize}
  \item The \textbf{$2$-category of elements} $\el(F)$ of $F$ is defined to be $\cH\eell(F)$.
  \item The \textbf{$2$-category of morphisms} $\mor(F)$ of $F$ is defined to be $\SV\eell(F)$.
  \end{itemize}
\end{definition}

The ardently $2$-categorical reader may be dismayed by the foray into the realm of double categories to give the above definition. In the coming discussion we will find that we are able to comfortably re-seat these $2$-categories as the result of purely $2$-categorical considerations.

Observe that exponentiation by the category $\mathbbm{2}=\{0\to 1\}$ gives rise to the classical functor $\ainl {\textsf{Cat}}{\ar\coloneqq(-)^{\mathbbm{2}}}{\textsf{Cat}}$, the \emph{category of arrows} functors, where $\textsf{Cat}$ is the category of categories and functors.

\begin{definition}
  We define the functor $\ainl {\TwoCat}{\Ar}{\TwoCat}$ as follows. It sends a $2$-category $\cC$ to the $2$-category $\Ar \cC$ with the same objects as $\cC$ and hom-categories $\Ar \cC(C,C')\coloneqq\ar(\cC(C,C'))$ for each pair of objects $C,C'\in \cC$. That is, a morphism in $\Ar \cC$ is a $2$-morphism of $\cC$ and a $2$-morphism in $\Ar \cC$ is a commutative square of vertical composites of $2$-morphisms in $\cC$.

  Given a normal pseudo-functor $\ainl \cC F\cD$, we define the normal pseudo-functor $\Ar F$ to act as $F$ on objects and as $\ar F$ on hom-categories. The compositors of $\Ar F$ are given component-wise by the compositors of $F$.
\end{definition}

\begin{remark}\label{rem:arprops}
  The functor $\Ar$ is a shadow of our double categorical approach of the previous sections. Indeed, we have the equality of functors $\SV\bH=\Ar$ to complement $\cH\bH=\id_{\TwoCat}$.
\end{remark}

Recall that $\eell(F)$ was defined as the pseudo-slice double category $\dblslice {\mathbbm 1}{\bH F}$. This, coupled with the fact that $\bH$ and $\SV$ preserve slices allows us to give the following, purely $2$-categorical formulations of the $2$-categories of elements and morphisms of a normal pseudo-functor.

\begin{remark}
  If \cC is a $2$-category and $\ainl {\cC^{\op}}{F}{\Cat}$ is a normal pseudo-functor then, by \cref{lem:iso_between_slices,rem:arprops}, the $2$-categories $\el(F)$ and $\mor(F)$ are isomorphic to the pseudo-slice $2$-categories induced by the cospans
  \[ \mathbbm 1\stackrel{\mathbbm 1}{\longrightarrow} \Cat \stackrel{F}{\longleftarrow} \cC^{\op}  \ \ \text{and} \ \ \mathbbm 1\stackrel{\mathbbm 1}{\longrightarrow} \Ar\Cat \xleftarrow{\Ar F} \Ar\cC^{\op} \ , \]
  respectively. In particular, $\el(F)$ is the pseudo-type version of the usual $2$-category of elements of $F$.
\end{remark}

We are now in a position to give the central result of this paper, a $2$-dimensional analogue to the classical relationship between representations and initial objects. The equivalence between (i) and (iii) below gives the promised $2$-categorical account of the theorem, and while it may be derived directly, we will find that our work of the previous sections allows for a more efficient approach via (ii).

\begin{theorem} \label{thm:birep}
  Let \cC be a $2$-category, and $\ainl{\cC^{\op}}{(F,\phi)}{\Cat}$ be a normal pseudo-functor. The following statements are equivalent.
  \begin{trome}
  \item\label{hyp:birep} The normal pseudo-functor $F$ has a bi-representation $(I,\rho)$.
  \item\label{hyp:dblbi-initial} There is an object $I\in \cC$ together with an object $i\in FI$ such that $(I,i)$ is double bi-initial in~$\eell(F)$.
  \item There is an object $I\in \cC$ together with an object $i\in FI$ such that $(I,i)$ is bi-initial in $\el(F)$ and $(I,\id_{i})$ is bi-initial in $\mor(F)$.
  \item There is an object $I\in \cC$ together with an object $i\in FI$ such that $(I,\id_{i})$ is bi-initial in $\mor(F)$.
  \end{trome}
\end{theorem}

\begin{remark}
  Note that conditions (iii) and (iv) are inherently $2$-categorical statements. Let us think of an object of a $2$-category \cA as a normal pseudo-functor $\mathbbm 1\to\cA$. Recall that an object $\mathbbm 1\to\cA$ is bi-initial if and only if the projection $\slice \cA I \to \cA$ is a trivial fibration, where the slice and its projection are obtained as a certain pullback. Thus condition (iii) instructs us to construct the slice $\slice {\el(F)}{(I,i)}$ as well as the slice $\slice {\mor(F)}{(I,\id_i)}$, where the latter is obtained by considering the object given by the composite $\mathbbm 1\to\el(F)\to \mor(F)$, and ensures that both resulting projections are trivial fibrations. Similarly, condition (iv) concerns the latter pullback only, and ensures that its projection functor is a trivial fibration.
\end{remark}

First note that the equivalence of conditions (ii), (iii), and (iv) follows directly from \Cref{thm:initial}. The rest of this section will be devoted to the proof of the equivalence between conditions (i) and (ii). For this, we first introduce the following ``unique filler'' lemma to make efficient the proof of the forward implication.

\begin{lemma}\label{lem:sqboundary}
Under the assumptions of \Cref{hyp:birep}, for every pair of horizontal morphisms $\ainl{(I,i)}{(f,\psi)}{(C,x)}$ and $\ainl{(I,i)}{(g,\varphi)}{(C,y)}$, and every vertical morphism $\vainl{(C,x)}{\alpha}{(C,y)}$ in $\eell(F)$, there is a unique square in $\eell(F)$ of the below form.
     \begin{diagram*}[.][4]
        \node(1)[]{$(I,i)$};
        \node(2)[right of=1,xshift=0.5cm]{$(C,x)$};
        \node(3)[below of=1]{$(I,i)$};
        \node(4)[below of=2]{$(C,y)$};
        \draw[a](1)to node[la]{$(f,\psi)$}(2);
        \draw[a](3)to node[la,swap]{$(g,\varphi)$}(4);
        \draw[proequal](1)to(3);
        \draw[proarrow](2)to node[la]{$\alpha$}(4);
        \node[la] at ($(1)!0.5!(4)$) {$\gamma$};
      \end{diagram*}
\end{lemma}

\begin{proof}
  Let $(I,\rho)$ be a bi-representation of a normal pseudo-functor $\ainl{\cC^{\op}}{(F,\phi)}{\Cat}$, $\ainl{(I,i)}{(f,\psi)}{(C,x)}$ and $\ainl{(I,i)}{(g,\varphi)}{(C,y)}$ be horizontal morphisms in $\eell(F)$, and $\vainl{(C,x)}\alpha{(C,y)}$ be a vertical morphism in $\eell(F)$. Define $\upsilon$ to be the unique morphism of $FC$ fitting in the following diagram,
  \begin{diagram*}
    \node(1)[]{$(Ff)i$};
    \node(2)[right of= 1]{$x$};
    \node(3)[right of= 2]{$y$};
    \node(4)[right of= 3]{$(Fg)i$};
    \node(6)[below of= 1]{$\rho_{C}(f)$};
    \node(7)[below of= 4]{$\rho_{C}(g)$};
    \draw[a](1)to node[la]{$\psi^{\inv}$}(2);
    \draw[a](2)to node[la]{$\alpha$}(3);
    \draw[a](3)to node[la]{$\varphi$}(4);
    \draw[a](1)to node[la,swap]{$(\rho_{f})_{\id_{I}}$} node[la]{$\iso$}(6);
    \draw[a](4)to node[la]{$(\rho_{g})_{\id_{I}}$} node[la,swap]{$\iso$}(7);
    \draw[u](6)to node[la,swap]{$\upsilon$}(7);
  \end{diagram*}
  where $\niinl{(Ff)\rho_I}{\rho_f}{\rho_C \cC(I,f)}$ and $\niinl{(Fg)\rho_I}{\rho_g}{\rho_C \cC(I,g)}$ are the $2$-isomorphism components of $\rho$ at $f$ and $g$, respectively. By \cref{def:eellsquare}, a square $\sq{\gamma}{(f,\psi)}{(g,\varphi)}{\id_i}{\alpha}$ in~$\eell(F)$ is the data of a $2$-morphism $\ninl f\gamma{g}$ of \cC such that
   \[ (Ff)i
    \xrightarrow{(F\gamma)_{i}} (Fg)i \quad = \quad (Ff)i\stackrel{\psi^{\inv}}{\longrightarrow}
    x\stackrel{\alpha}{\longrightarrow}y \stackrel{\varphi}{\longrightarrow}(Fg)i.
  \]
  Therefore, we may deduce that this equation holds if and only if $(\rho_{g})_{\id_{I}}(F\gamma)_{i}=\upsilon(\rho_{f})_{\id_{I}}$, by definition of $\upsilon$. The left-hand composite of this equality appears as the result of evaluating the below-left pasting at $\id_{I}$, and this pasting is equal to the below-right pasting by pseudo-naturality of $\rho$.
  \begin{diagram*}[][][node distance=1.8cm]
    \node(1)[]{$\cC(I,I)$};
    \node(2)[right of= 1,xshift=0.5cm,minimum size=0.0cm]{$FI$};
    \node(3)[below of= 1]{$\cC(I,C)$};
    \node(4)[below of= 2,minimum size=0.0cm]{$FC$};
    \draw[a](1)to node[la]{$\rho_{I}$}(2);
    \draw[a](1)to node[la,swap]{$g_*$}(3);
    \draw[a](3)to node[la,swap]{$\rho_{C}$}(4);
    \draw[a](2)to node(Fac')[la,over,pos=0.5] {$Fg$} (4);
    \draw[a,bend left=65](2)to (4);
    \node(Fc)[la] at ($(Fac')+(1.1cm,0)$) {$Ff$};
    \celli[la,swap]{2}{3}{$\rho_{g}$};
    \cell[la,swap,xshift=-3pt][n][0.58]{Fc}{Fac'}{$F\gamma$};

    \node(5)[right of= 2,xshift=2.5cm]{$\cC(I,I)$};
    \node(6)[right of=5,xshift=0.5cm]{$FI$};
    \node(7)[below of=5]{$\cC(I,C)$};
    \node(8)[below of=6]{$FC$};
    \draw[a](5)to node[la]{$\rho_{I}$}(6);
    \draw[a](7)to node[la,swap]{$\rho_{C}$}(8);
    \draw[a](5)to node(F)[la,over]{$f_*$}(7);
    \draw[a](6)to node[la]{$Ff$}(8);
    \draw[a,bend right=60](5)to node(bad)[la,left]{$g_*$}(7);
    \coordinate (outer) at ($(bad.east)-(0,3pt)$);
    \coordinate (inner) at ($(outer)+(1.5cm,0)$);
    \cell[la,swap,xshift=2pt][n][0.4]{F}{bad}{$\gamma_*$};
    \celli[la,swap]{6}{7}{$\rho_{f}$};

    \node at ($(Fc.east)!0.5!(bad.west)$) {$=$};
  \end{diagram*}
  We deduce therefore that  \[ (\rho_{g})_{\id_{I}}(F\gamma)_{i}=\upsilon(\rho_{f})_{\id_{I}} \ \ \text{iff} \ \ (\rho_{C}(\gamma))(\rho_{f})_{\id_{I}}=\upsilon(\rho_{f})_{\id_{I}}\ \ \text{iff} \ \ \rho_{C}(\gamma)=\upsilon. \]
  All in all then, $\sq{\gamma}{(f,\psi)}{(g,\varphi)}{\id_i}{\alpha}$ is a square in $\eell(F)$ if and only if $\rho_{C}(\gamma)=\upsilon$. Since $\ainl{\cC(I,C)}{\rho_{C}}{FC}$ is an equivalence and is therefore fully faithful on morphisms, there is a unique such~$\gamma$.
\end{proof}

With this lemma established, the proof of the forward implication (i)$\Rightarrow$(ii) of \Cref{thm:birep} is readily given.

\begin{proof}[\Cref{thm:birep}, (i)$\Rightarrow$(ii)]
  Suppose (i), that is, we have a specified bi-representation $(I,\rho)$ of $F$. From this data we will select an object $i\in FI$ and demonstrate that $(I,i)$ is double bi-initial in $\eell(F)$. To begin, let us define $i\in FI$ as $i\coloneqq\rho_I(\id_I)$. We address each of conditions (i-iv) of \Cref{def:dblbi-initial} in turn.

  Let $(C,x)$ be an object of $\eell(F)$. Since $\ainl{\cC(C,I)}{\rho_C}{FC}$ is an equivalence and $x\in FC$, there is a morphism $\ainl{C}f{I}$ in $\cC$ together with a isomorphism $\aiinl {x}{\overline\psi}{\rho_C(f)}$ in $FC$. By post-composing with the inverse of $\aiinl {(Ff)i}{(\rho_f)_{\id_I}}{\rho_C(f)}$, arising from the $2$-isomorphism component $\niinl{(Ff)\rho_I}{\rho_f}{\rho_C \cC(f,I)}$ of $\rho$ at $f$, we find a horizontal morphism $\ainl {(I,i)}{(f,(\rho_f)_{\id_I}^{\inv}\overline\psi)}{(C,x)}$ in $\eell(F)$, and so we have established \Cref{cond:exists-obj}.

  The rest of conditions (ii-iv) each follow from applications of \Cref{lem:sqboundary} above, which we elaborate below. First, \Cref{cond:exists-horiz} grants us the existence of a boundary of $\eell(F)$ of the form depicted below, and charges us with finding a unique, vertically invertible filler.
  \begin{diagram*}
    \node(1)[]{$(I,i)$};
    \node(2)[below of= 1]{$(I,i)$};
    \node(3)[right of= 2, xshift=0.5cm]{$(C',x')$};
    \node(4)[right of= 3, xshift=0.5cm]{$(C,x)$};
    \node(5)[above of= 4]{$(C,x)$};
    \draw[a](1)to node[la]{$(f,\psi)$}(5);
    \draw[proequal](1)to(2);
    \draw[proequal](5)to(4);
    \draw[a](2)to node[la,swap]{$(f',\psi')$}(3);
    \draw[a](3)to node[la,swap]{$(c,\varphi)$}(4);
  \end{diagram*}
  By composing the bottom horizontal morphisms we see that \Cref{lem:sqboundary} supplies us with a unique filler for this square. That this filler is vertically invertible follows from considering the vertical opposite of the above square combined with further applications of \Cref{lem:sqboundary}.

  Next, \Cref{cond:exists-vert} grants us a vertical morphism $\vainl{(C,x)}\alpha{(C,y)}$ of $\eell(F)$ and demands the existence of a square from $\vidinl {(I,i)}$ to $\alpha$. By our construction of \Cref{cond:exists-obj} above, we may give horizontal morphisms $\ainl{(I,i)}{(f,\psi)}{(C,x)}$ and $\ainl{(I,i)}{(g,\varphi)}{(C,y)}$, and thus produce the below boundary in $\eell(F)$. An application of \Cref{lem:sqboundary} shows (iii).
  \begin{diagram*}
    \node(1)[]{$(I,i)$};
    \node(2)[right of=1,xshift=0.5cm]{$(C,x)$};
    \node(3)[below of=1]{$(I,i)$};
    \node(4)[below of=2]{$(C,y)$};
    \draw[a](1)to node[la]{$(f,\psi)$}(2);
    \draw[a](3)to node[la,swap]{$(g,\varphi)$}(4);
    \draw[proequal](1)to(3);
    \draw[proarrow](2)to node[la]{$\alpha$}(4);
  \end{diagram*}

  Finally, we must show that \Cref{cond:square} holds. That is, we must demonstrate that there is an equality of squares filling a fixed boundary. Fortunately we may apply \Cref{lem:sqboundary} to this boundary and so conclude the proof.
\end{proof}

We conclude this section by proving the reverse implication.

\begin{proof}[\Cref{thm:birep}, (ii)$\Rightarrow$(i)]
  Suppose (ii), that is, we have a double bi-initial object $(I,i)$ in $\eell(F)$. From this data we will construct equivalences $\ainl{\cC(C,I)}{\rho_{C}}{FC}$ for each $C\in \cC$ and then show that they assemble into a pseudo-natural transformation. By a standard result of $2$-categories, any equivalence is canonically rectifiable into an adjoint equivalence and so we do not trouble ourselves with additional work after giving $\rho$.

  For a fixed $C\in\cC$, let us define the functor $\ainl{\cC(C,I)}{\rho_{C}}{FC}$ on objects $\ainl CfI$ as $\rho_{C}(f)\coloneqq (Ff)i$ and on morphisms $\ninl f\gamma g$  as $\rho_{C}(\gamma)\coloneqq (F\gamma)_{i}$. As $F$ respects vertical composition of $2$-morphisms strictly, it is clear that $\rho_{C}$ is a functor by construction.

  With the functors $\rho_{C}$ defined, we now show that each of these functors is an equivalence. To that end, let us fix $C\in\cC$ and $x\in FC$. Observe that $(C,x)$ is an object of $\eell(F)$ so that, since $(I,i)$ is double bi-initial in $\eell(F)$, there is a horizontal morphism $\ainl{(I,i)}{(f,\psi)}{(C,x)}$ in $\eell(F)$. This is precisely the data of an object $f\in\cC(C,I)$ and an isomorphism $\aiinl{\rho_{C}(f)=(Ff)i}\psi x$, which shows that $\rho_{C}$ is essentially surjective on objects. To see that each $\rho_{C}$ is fully faithful on morphisms, let $\ainl C{f,g}I$ be objects in $\cC(C,I)$ and $\ainl{\rho_{C}(f)}{\alpha}{\rho_{C}(g)}$ be a morphism between their images in $FC$. This data is equivalently a pair of horizontal morphisms $\ainl{(I,i)}{(f,\id_{\rho_C(f)})}{(C,\rho_{C}(f))}$ and $\ainl{(I,i)}{(g,\id_{\rho_C(g)})}{(C,\rho_{C}(g))}$,
  since $\rho_C(f)=(Ff)i$ and $\rho_C(g)=(Fg)i$ by definition, together with a vertical morphism $\vainl{(C,\rho_{C}(f))}{\alpha}{(C,\rho_{C}(g))}$ in $\eell(F)$. Since $(I,i)$ is double bi-initial in $\eell(F)$, by \Cref{prop:simpler-dblbinitial} (ii'), there is a unique square in $\eell(F)$ of the form
  \begin{diagram*}[,][4]
  \node(1)[]{$(I,i)$};
    \node(2)[right of=1,xshift=1.25cm]{$(C,\rho_C(f))$};
        \node(3)[below of=1]{$(I,i)$};
        \node(4)[below of=2]{$(C,\rho_C(g))$};
        \draw[a](1)to node[la]{$(f,\id_{\rho_C(f)})$}(2);
        \draw[a](3)to node[la,swap]{$(g,\id_{\rho_C(g)})$}(4);
        \draw[proequal](1)to(3);
        \draw[proarrow](2)to node[la]{$\alpha$}(4);
        \node[la] at ($(1)!0.5!(4)$) {$\gamma$};
  \end{diagram*}
  that is, a unique $2$-morphism $\ninl f\gamma g$ such that $\rho_C(\gamma)=(F\gamma)_i=\alpha$. This shows fully faithfulness of $\rho_{C}$.

  Now that we have a collection of object-wise equivalences $\rho_{C}$ we seek to construct the data of the pseudo-naturality comparison natural isomorphisms $\niinl{(Fc)\rho_{C'}}{\rho_{c}}{\rho_{C}\cC(c,I)}$ depicted below, for each morphism $\ainl{C}c{C'}$ in $\cC$.
  \begin{diagram*}
    \node(1)[]{$\cC(C',I)$};
    \node(2)[right of=1,xshift=0.5cm]{$FC'$};
    \node(3)[below of=1]{$\cC(C,I)$};
    \node(4)[below of=2]{$FC$};
    \draw[a](1)to node[la]{$\rho_{C'}$}(2);
    \draw[a](3)to node[la,swap]{$\rho_{C}$}(4);
    \draw[a](1)to node[la,swap]{$\cC(c,I)$}(3);
    \draw[a](2)to node[la]{$Fc$}(4);
    \celli[la,swap]{2}{3}{$\rho_{c}$};
  \end{diagram*}
  For $f\in\cC(C',I)$ observe that $(Fc)\rho_{C'}(f)=(Fc)(Ff)i$ and $\rho_{C}\cC(c,I)(f)=F(fc)i$, so that we can set $\rho_{c}$ to be $(\phi_{c,-})_i$, the compositor of $F$ at $(c,-)$ evaluated at $i$. This satisfies all of the required properties of pseudo-naturality.
\end{proof}

In fact, we have additionally proven that bi-representations $(I,\rho)$ of a normal pseudo-functor are determined up to isomorphism by their values on $\id_I$. This conclusion may be seen as a special case of a suitable $2$-dimensional Yoneda lemma.

\begin{corollary}\label{cor:bi-rep-rectification}
  Let \cC be a $2$-category, and $\ainl{\cC^{\op}}{F}{\Cat}$ be a normal pseudo-functor. Suppose that $(I,\rho)$ is a bi-representation of $F$. Then there is a canonical bi-representation $(I,\overline\rho)$ of $F$ given by \[ \ainl{\cC(C,I)}{\overline \rho_C=(F-)(\rho_I(\id_I))}{FC}, \]
  for every $C\in \cC$. Moreover, we have that $\overline\rho\iso\rho$.
\end{corollary}

\begin{proof}
  The construction is given by tracing the proofs above of \cref{thm:birep} through (i)$\Rightarrow$(ii) and then (ii)$\Rightarrow$(i). Finally, the isomorphism $\overline\rho\iso\rho$ is given by the $2$-isomorphism components $\aiinl{(Ff)\rho_I(\id_I)}{(\rho_f)_{\id_I}}{\rho_C(f)}$ of $\rho$ itself evaluated at $\id_I$, for every $\ainl CfI$ in $\cC$.
\end{proof}

\begin{remark}
  In particular, when $F$ is a strict $2$-functor, without loss of generality a bi-representation of $F$ may be taken to be a $2$-natural adjoint equivalence. Indeed, the bi-representation constructed in \cref{cor:bi-rep-rectification} is $2$-natural.
\end{remark}

\subsection{The case in presence of tensors by \texorpdfstring{$\mathbbm 2$}{2}} \label{sec:tensorby2case}

Finally we explore a substantial improvement of \cref{thm:birep} which is possible when the $2$-category \cC has \emph{tensors}, defined below, which are preserved by $F$.

\begin{definition}
    Let $\cC$ be a $2$-category, $C\in \cC$ be an object, and $\mathcal A$ be a category.

    A \textbf{power of $C$ by} $\mathcal A$ is a weighted $2$-limit of the $2$-functor $C\colon \mathbbm 1\to \cC$ by the weight $\mathcal A\colon \mathbbm 1\to \Cat$. In other words, it is a pair $(\mathcal A\pitchfork C,\lambda)$ of an object $\mathcal A\pitchfork C\in \cC$ and a functor $\ainl {\mathcal A}{\lambda}{\cC(\mathcal A\pitchfork C,C)}$ such that, for every object $C'\in \cC$, pre-composition by $\lambda$ induces an isomorphism of categories
    \[ \aiinl{\cC(C',\mathcal A\pitchfork C)}{\lambda_*\circ \cC(-,C)}{\Cat(\mathcal A,\cC(C',C))}. \]
    We say that $\cC$ \textbf{has powers by $\mathcal A$} if there is a power of $C$ by $\mathcal A$ for each object $C\in \cC$.

    Dually, a \textbf{tensor of $C$ by} $\mathcal A$ is a power of $C$ by $\mathcal A$ in the opposite $2$-category $\cC^{\op}$. In other words, it is a pair $(C\otimes \mathcal A, \zeta)$ of an object $C\otimes \mathcal A\in \cC$ and a functor $\zeta\colon \mathcal A\to \cC(C,C\otimes \mathcal A)$ such that, for every object $C'\in \cC$, pre-composition by $\zeta$ induces an isomorphism of categories
    \[ \aiinl{\cC(C\otimes \mathcal A,C')}{\zeta^*\circ \cC(C,-)}{\Cat(\mathcal A,\cC(C,C'))} \ . \]
    We say that $\cC$ \textbf{has tensors by} $\mathcal A$ if there is a tensor of $C$ by $\mathcal A$ for each object $C\in \cC$.
\end{definition}

\begin{remark}
  Powers may be viewed as a lower-dimensional shadow of the double categorical notion of tabulators seen in \cref{def:tabulator}. Indeed, a power of an object $C\in \cC$ by the category $\mathbbm 2=\{0\to 1\}$ in a $2$-category $\cC$ is precisely a tabulator of the vertical identity $\vid_C$ in its associated horizontal double category $\bH \cC$; see \cite[Exercise 5.6.2 (c)]{Grandis}. In particular, tabulators in $\bH\cC^{\op}$ correspond to tensors by $\mathbbm 2$ in $\cC$.
\end{remark}

From the universal property of powers, it is straightforward to see that the $2$-category $\Cat$ has powers by any category $\mathcal A$ given by $\mathcal A\pitchfork \mathcal C\coloneqq\Cat(\mathcal A,\mathcal C)$. Given a $2$-category $\cC$ with tensors by a category $\mathcal A$, we then say that a normal pseudo-functor $\ainl{\cC^{\op}}{F}{\Cat}$ \textbf{preserves powers by} $\mathcal A$ if, for every object $C\in \cC$, we have an isomorphism of categories $F(C\otimes\mathcal A)\iso \Cat(\mathcal A, FC)$ which is natural with respect to the defining cones.

\begin{theorem}\label{thm:eeellll-tensors}
  Let $\cC$ be a $2$-category with tensors by $\mathbbm 2$, and $\ainl {\cC^{\op}}{F}{\Cat}$ be a normal pseudo-functor which preserves powers by $\mathbbm 2$. Then the following statements are equivalent.
  \begin{rome}
  \item The normal pseudo-functor $F$ has a bi-representation $(I,\rho)$.
  \item There is an object $I\in \cC$ together with an object $i\in FI$ such that $(I,i)$ is double bi-initial in $\eell(F)$.
  \item There is an object $I\in\cC$ together with an object $i\in FI$ such that $(I,i)$ is bi-initial in $\el(F)$.
  \end{rome}
  In particular, $(I,i)$ is bi-initial in $\el(F)$ if and only if $(I,\id_i)$ is bi-initial in~$\mor(F)$.
\end{theorem}

In order to prove this result we will make use of the following lemma.

\begin{lemma} \label{lem:eellF-tabulators}
  Let $\cC$ be a $2$-category with tensors by $\mathbbm 2$, and $\ainl {\cC^{\op}}{F}{\Cat}$ be a normal pseudo-functor which preserves powers by $\mathbbm 2$. Then the double category $\eell(F)$ has tabulators.
\end{lemma}

\begin{proof}
  Let $\vainl{(C,x)}{\alpha}{(C,y)}$ be a vertical morphism in $\eell(F)$ and let $(C\otimes \mathbbm 2, \zeta)$ be a tensor of $C$ by $\mathbbm 2$. Recall that $\zeta$ is a functor $\zeta\colon \mathbbm 2\to \cC(C,C\otimes \mathbbm 2)$, and therefore it corresponds to a $2$-morphism
\begin{diagram*}[.][2]
\node[minimum size=0.8cm](1) {$C$};
\node[right of=1,xshift=.5cm](2) {$C\otimes \mathbbm 2$};
\draw[bend left,a] (1) to node(a)[above,la] {$\zeta_0$} (2);
\draw[bend right,a] (1) to node(b)[below,la] {$\zeta_1$} (2);
\cell[la]{a}{b}{$\zeta$};
\end{diagram*}
  Moreover, the morphism $\alpha\colon x\to y$ in $FC$ is equivalently a functor $\alpha\colon  \mathbbm 2\to FC$ and therefore it corresponds to an object $\overline{\alpha}\in F(C\otimes \mathbbm 2)$ as $\Cat(\mathbbm 2,FC)\iso F(C\otimes \mathbbm 2)$. We set $\top u\coloneqq (C\otimes\mathbbm 2,\overline{\alpha})\in\eell(F)$ and $\tau_u$ to be the following square in $\eell(F)$.
\begin{diagram*}
\node[](1) {$(C\otimes \mathbbm 2, \overline{\alpha})$};
\node[below of =1](2) {$(C\otimes \mathbbm 2, \overline{\alpha})$};
\node[right of=1,xshift=1cm](3) {$(C,x)$};
\node[below of =3](4) {$(C,y)$};
\draw[a] (1) to node[above,la] {$(\zeta_0,\id_x)$} (3);
\draw[a] (2) to node[below,la] {$(\zeta_1,\id_y)$} (4);
\draw[proequal] (1) to (2);
\draw[proarrow] (3) to node[right,la] {$\alpha$} (4);
\node[la] at ($(1)!0.5!(4)$) {$\zeta$};
\end{diagram*}
We show that it satisfies the universal properties of tabulators of \cref{def:tabulator}. Let
\begin{diagram*}
\node[](1) {$(C',x')$};
\node[below of =1](2) {$(C',x')$};
\node[right of=1,xshift=1cm](3) {$(C,x)$};
\node[below of =3](4) {$(C,y)$};
\draw[a] (1) to node[above,la] {$(c,\psi)$} (3);
\draw[a] (2) to node[below,la] {$(d,\varphi)$} (4);
\draw[proequal] (1) to (2);
\draw[proarrow] (3) to node[right,la] {$\alpha$} (4);
\node[la] at ($(1)!0.5!(4)$) {$\gamma$};
\end{diagram*}
be a square in $\eell(F)$. By the universal property of tensors, the $2$-morphism $\ninl c\gamma d$ corresponds to a morphism $\overline{\gamma}\colon C\otimes \mathbbm 2\to C'$. Moreover, note that the pair $(\psi,\varphi)$ gives an isomorphism in $\Cat(\mathbbm 2,FC)$ from $\alpha$ to $(F\gamma)_{x'}$, and since $F$ preserves powers by $\mathbbm 2$, then $(\psi,\varphi)$ corresponds to an isomorphism $\overline{(\psi,\varphi)}\colon \overline{\alpha}\cong (F\overline{\gamma})_{x'}$ in $F(C\otimes \mathbbm 2)$. We get the required horizontal morphism $(\overline{\gamma},\overline{(\psi,\varphi)})\colon (C',x')\to (C\otimes \mathbbm 2,\overline{\alpha})$ for \cref{cond:firstup-tab}.

Similarly, \cref{cond:secondup-tab} follows from the fact that $2$-morphisms in $\cC$ of the form
\begin{diagram*}
\node[minimum size=0.8cm](2) {$C'$};
\node[left of=1,xshift=-.5cm](1) {$C\otimes \mathbbm 2$};
\draw[bend left,a] (1) to node(a)[above,la] {$\overline{\gamma}$} (2);
\draw[bend right,a] (1) to node(b)[below,la] {$\overline{\gamma'}$} (2);
\cell[la]{a}{b}{$\overline{\theta}$};
\end{diagram*}
uniquely correspond to $2$-morphisms $\theta_0,\theta_1$ in $\cC$ between morphisms $C\to C'$ such that $\gamma \theta_0=\theta_1\gamma'$, by the universal property of tensors.
\end{proof}

\begin{proof}[\cref{thm:eeellll-tensors}]
    First note that (i) and (ii) are equivalent by \cref{thm:birep}.

    To see that (ii) and (iii) are equivalent consider the following. By \cref{lem:eellF-tabulators}, the double category $\eell(F)$ admits tabulators. Thus, by \cref{prop:tabulators-biinitial}, an object $(I,i)$ is double bi-initial in $\eell(F)$ if and only if $(I,i)$ is bi-initial in $\el(F)$.

    Finally, that $(I,i)$ is bi-initial in $\el(F)$ if and only if $(I,\id_i)$ is bi-initial in $\mor(F)$ follows from \cref{cor:bi-in-in-HA-VA-tab}.
\end{proof}

\section{Applications to bi-adjunctions and weighted bi-limits}
\label{sec:applications}

Now that we have satisfied ourselves with the characterisation of \Cref{thm:birep} of bi-representations of normal pseudo-functor we focus now on two formal applications. In \Cref{sec:bi-adjunctions}, we will leverage some $2$-dimensional arguments to give a characterisation of bi-adjunctions in terms of bi-terminal objects in pseudo-slices. Then, in \Cref{sec:weighted-limits} we will connect to the counter-examples given in \cite{us} by proving a correct characterisation of bi-limits in terms of bi-terminal objects in pseudo-slices, specialising the supporting theorem in the same section about weighted bi-limits. In both sections we will additionally give improvements on these results by specialising \cref{thm:eeellll-tensors} when the $2$-categories at issue have tensors by $\mathbbm 2$, which in the case of weighted bi-limits subsumes a known special case.

\subsection{Bi-adjunctions}
\label{sec:bi-adjunctions}

We begin by introducing the notion of a bi-adjunction.

\begin{definition} \label{def:biadj}
  Let $\cC$ and $\cD$ be $2$-categories. A \textbf{bi-adjunction} between \cC and \cD comprises the data of normal pseudo-functors $\ainl{\cC}{L}{\cD}$ and $\ainl{\cD}{R}{\cC}$, and adjoint equivalences of categories \[ \aeinl{\cC(C,RD)}{\Phi_{C,D}}{\cD(LC,D)} \]
  pseudo-natural in $C\in \cC^{\op}$ and $D\in \cD$.
\end{definition}

\begin{remark}\label{rem:local-nature}
  We wish to draw the reader's attention to the following relationship between bi-adjunctions and bi-representations. If $L$ and $R$ are embroiled in a bi-adjunction, then in particular for each object $D\in\cD$ we may observe that we have a bi-representation
  \[ \neinl{\cC(-,RD)}{\Phi_{-,D}}{\cD(L-,D)} \]
  of the normal pseudo-functor $\cD(L-,D)$. In this sense, bi-adjunctions are ``locally'' bi-representations.
\end{remark}

The major goal of this section is to provide a converse to the above observation. That is, we will concentrate our efforts on establishing that being ``locally bi-represented'' is, in fact, enough to determine a right bi-adjoint in the sense of \Cref{def:biadj}. Such a result is of course expected by analogy to the ordinary categorical version. Giving such a formulation of bi-adjunctions in terms of bi-representations allows us to apply \Cref{thm:birep} and thereby give a characterisation of bi-adjunctions in terms of bi-terminal objects in pseudo-slices.

\begin{theorem}\label{lem:yonedastyle}
  Let \cC and \cD be $2$-categories, and $\ainl \cC {(L,\delta)} \cD$ be a normal pseudo-functor. The following statements are equivalent.
  \begin{rome}
  \item The normal pseudo-functor $\ainl\cC L \cD$ has a right bi-adjoint $\ainl \cD R \cC$.
  \item For all objects $D\in\cD$, there is a bi-representation $(RD,\Psi^D)$ of $\cD(L-,D)$, where $\neinl{\cC(-,RD)}{\Psi^{D}}{\cD(L-,D)}$ in $\Psd(\cC^{\op},\Cat)$ is a pseudo-natural adjoint equivalence.
  \end{rome}
\end{theorem}

In order to prove this theorem we will make use of some purely formal results about the nature of bi-representations and bi-adjunctions. These arguments depend crucially upon the apparatus of a $2$-dimensional Yoneda lemma -- see \cite[\S 8.3]{JohYau} for the bi-categorical account.

\begin{notation}
    Let $\cC$ be a $2$-category. We denote the Yoneda embedding $2$-functor by
    \[ \ainl{\cC}{\Yoneda}{\Psd(\cC^{\op},\Cat)} \ . \]
    This $2$-functor sends an object $C\in \cC$ to the $2$-functor $\ainl{\cC^{\op}}{\Yoneda_C\coloneqq \cC(-,C)}{\Cat}$, and acts in the obvious way on hom-categories.
\end{notation}

We will make extensive use of the full sub-$2$-category on the image of the Yoneda $2$-functor, but this $2$-category is isomorphic to the following.

\begin{definition}
  Let $\cC$ be a $2$-category. We define a $2$-category $\YC$ with the same objects as $\cC$ and whose hom-categories are given by
  \[ \YC(C,C')\coloneqq \Psd(\cC^{\op},\Cat)(\Yoneda_C,\Yoneda_{C'}) \]
  for all $C,C'\in \cC$. Composition operations are given by those of $\Psd(\cC^{\op},\Cat)$.
\end{definition}

\begin{remark}
  The $2$-category $\YC$ is isomorphic to the full sub-$2$-category of $\Psd(\cC^{\op},\Cat)$ on the objects of the form $\Yoneda_C$ for $C\in \cC$. Observe that we therefore have the following factorisation of $2$-functors
  \begin{diagram*}
    \node(1)[]{$\cC$};
    \node(2)[right of= 1,xshift=1.5cm]{$\Psd(\cC^{\op},\Cat)$};
    \node(3) at ($(1.east)!0.5!(2.west)-(0,1cm)$) {$\cC^\Yoneda$};
    \draw[a] (1) to node[la]{$\Yoneda$} (2);
    \draw[a] (1) to node[la,swap]{$\sY$} (3);
    \draw[a,right hook->] (3) to (2.south west);
  \end{diagram*}
  where $\sY$ is the identity on objects.

  We have avoided defining $\YC$ as ``the full sub-2-category of bi-representables'' as this is problematic inasmuch as objects are concerned: we will need to know \emph{which} object is associated to a given bi-representable functor, but a priori any such object is only defined up to equivalence. One way to solve this is to choose, for each bi-representable, a representing object, and the result is precisely our $2$-category $\YC$ above.
\end{remark}

\begin{remark}\label{lem:evYequiv}
    Let \cC be a $2$-category. Then the $2$-dimensional Yoneda lemma says that the normal pseudo-functor $\ainl{\cC}{\sY}{\YC}$ is the identity on objects and induces equivalences between the hom-categories
    \[ \cC(C,C')\eqv \Psd(\cC^{\op},\Cat)(\Yoneda_C,\Yoneda_{C'})=\YC(C,C'), \]
    for all objects $C,C'\in \cC$; see \cite[Lemma 8.3.12]{JohYau}.

    Consequently, one can construct a normal pseudo-functor $\ainl{\YC}{\evY}{\cC}$ together with pseudo-natural isomorphisms $\niinl{\id_{\cC}}\eta{\evY\sY}$ and $\niinl{\sY\evY}\varepsilon{\id_{\YC}}$. We may see that $\evY$ is the identity on objects, and acts on hom-categories $\YC(C,C')$ by first taking the $C$-component of the pseudo-natural transformation or modification, and then evaluating it at $\id_C\in \Yoneda_C(C)=\cC(C,C)$. In fact, closer inspection reveals that $\id_\cC=\evY\sY$.
\end{remark}

\begin{lemma}\label{lem:yon}
  Let $\cC$ and $\cD$ be $2$-categories. Suppose that $\ainl{\cD}{Q}{\Psd(\cC^{\op},\Cat)}$ is a normal pseudo-functor such that, for all objects $D\in \cD$, $QD=\Yoneda_{RD}$ for an object $RD\in \cC$. Then there is a normal pseudo-functor $\ainl{\cD}{R}{\cC}$ and a pseudo-natural isomorphism $\Yoneda R\iso Q$ in $\Psd(\cD,\Psd(\cC^{\op},\Cat))$.
\end{lemma}

\begin{proof}
  First note that the image of the normal pseudo-functor $\ainl{\cD}{Q}{\Psd(\cC^{\op},\Cat)}$ is contained in the full sub-$2$-category $\YC$ of $\Psd(\cC^{\op},\Cat)$. That is, we have the following factorisation
  \begin{diagram*}[,][3]
    \node(1)[]{$\cD$};
    \node(2)[right of= 1,xshift=1.5cm]{$\Psd(\cC^{\op},\Cat)$};
    \node(3) at ($(1.east)!0.5!(2.west)-(0,1cm)$) {$\cC^\Yoneda$};
    \draw[a] (1) to node[la]{$Q$} (2);
    \draw[a] (1) to node[la,swap]{$\overline{Q}$} (3);
    \draw[a,right hook->] (3) to (2.south west);
  \end{diagram*}
  where $\overline Q D=RD$, for all objects $D\in \cD$. Using this we define the normal pseudo-functor $\ainl{\cD}{R}{\cC}$ to be the composite
 \[ \cD\stackrel{\overline Q}{\longrightarrow} \YC \stackrel{\evY}{\longrightarrow}\cC. \] Observe that by \cref{lem:evYequiv} we have a pseudo-natural isomorphism $\sY R=\sY\evY \overline{Q}\iso\overline{Q}$ in~$\Psd(\cD,\YC)$. By post-composing with the inclusion $\cC^\Yoneda\hookrightarrow\Psd(\cC^{\op},\Cat)$ this gives a pseudo-natural isomorphism $\Yoneda R\iso Q$ in $\Psd(\cD,\Psd(\cC^{\op},\Cat))$.
\end{proof}

Finally we need a technical result relating pseudo-natural isomorphisms in the category $\Psd(\cD,\Psd(\cC^{\op},\Cat))$ to those in $\Psd(\cC^{\op}\times\cD,\Cat)$. While it is not generally true that these two categories are isomorphic, we may ``by hand'' show that in a special case pseudo-natural isomorphisms in the former may be recast as occurring in the latter.

\begin{lemma}\label{lem:QisotoR}
  Given normal pseudo-functors $\ainl \cC L\cD$ and $\ainl \cD R\cC$, as well as a pseudo-natural isomorphism $\niinl{\cD(L-,-)}{\theta}{\cC(-,R-)}$ in $\Psd(\cD,\Psd(\cC^{\op},\Cat))$, the data of $\theta$ gives a pseudo-natural isomorphism $\niinl{\cD(L-,-)}{\gamma}{\cC(-,R-)}$ in $\Psd(\cC^{\op}\times\cD,\Cat)$.
\end{lemma}

\begin{proof}
  In order to give $\gamma$ we must give its value at each pair $(C,D)\in\cC^{\op}\times\cD$ of objects as well as its $2$-morphism component on each morphism of such pairs, and demonstrate that suitable compatibility conditions hold.

  Observe that for each $D\in\cD$ we have a pseudo-natural isomorphism \[\theta_{D}\in\Psd(\cC^{\op},\Cat)(\cD(L-,D),\cC(-,RD))\ .\] Thus it makes sense to set the component $\ainl{\cD(LC,D)}{\gamma_{(C,D)}}{\cC(C,RD)}$ to be the isomorphism $(\theta_{D})_{C}$.

  For each $\ainl Dd{D'}$ observe that the pseudo-naturality of $\theta$ gives the following diagram of pseudo-natural isomorphisms and modifications.
  \begin{diagram}\label{diag:pseudomod}
    \node(1)[]{$\cD(L-,D)$};
    \node(2)[right= of 1.east,anchor=west]{$\cC(-,RD)$};
    \node(3)[below of=1]{$\cD(L-,D')$};
    \node(4) at (3-|2) {$\cC(-,RD')$};
    \draw[n](1)to node[la]{$\theta_{D}$}(2);
    \draw[n](3)to node[la,swap]{$\theta_{D'}$}(4);
    \draw[n](1)to node[la,swap]{$\cD(L-,d)$}(3);
    \draw[n](2)to node[la]{$\cC(-,Rd)$}(4);
    \celli[la][t]{2}{3}{$\theta_{d}$};
  \end{diagram}
  Using this and the normality of $L$ and $R$, on a morphism $\ainl {C'}cC$ of \cC we may define $\niinl{\cC(c,Rd)\gamma_{(C,D)}}{\gamma_{(c,d)}}{\gamma_{(C',D')}\cD(Lc,d)}$ to be either of the following two, equal pastings.
  \begin{diagram}\label{diag:gammadef}
    \node(1)[]{$\cD(LC,D)$};
    \node(2)[right=of 1.east,anchor=west]{$\cC(C,RD)$};
    \node(3)[below of=1]{$\cD(LC',D)$};
    \node(4) at (3-|2){$\cC(C',RD)$};
    \node(5)[below of=3]{$\cD(LC',D')$};
    \node(6) at (5-|2){$\cC(C',RD')$};

    \node(1')[right=of 2.east,anchor=west]{$\cD(LC,D)$};
    \node(2')[right=of 1'.east,anchor=west]{$\cC(C,RD)$};
    \node(3')[below of=1']{$\cD(LC,D')$};
    \node(4') at (3'-|2'){$\cC(C,RD')$};
    \node(5')[below of=3']{$\cD(LC',D')$};
    \node(6') at (5'-|2'){$\cC(C',RD')$};

    \draw[a](1)to node[la]{$(\theta_{D})_{C}$}(2);
    \draw[a](3)to node[la,over]{$(\theta_{D})_{C'}$}(4);
    \draw[a](1)to node[la,swap]{$\cD(Lc,D)$}(3);
    \draw[a](2)to node[la]{$\cC(c,RD)$}(4);
    \draw[a](4)to node[la,right]{$\cC(C',Rd)$}(6);
    \draw[a](3)to node[la,left]{$\cD(C',d)$}(5);
    \draw[a](5)to node[la,below]{$(\theta_{D'})_{C'}$}(6);

    \draw[a](1')to node[la]{$(\theta_{D})_{C}$}(2');
    \draw[a](3')to node[la,over]{$(\theta_{D'})_{C}$}(4');
    \draw[a](1')to node[la,swap]{$\cD(LC,d)$}(3');
    \draw[a](2')to node[la]{$\cC(C,Rd)$}(4');
    \draw[a](4')to node[la,right]{$\cC(c,RD')$}(6');
    \draw[a](3')to node[la,left]{$\cD(c,D')$}(5');
    \draw[a](5')to node[la,below]{$(\theta_{D'})_{C'}$}(6');

    \celli[la,swap]{2}{3}{$(\theta_{D})_{c}$};
    \celli[la,swap]{4}{5}{$(\theta_{d})_{C}$};
    \celli[la,swap]{2'}{3'}{$(\theta_{d})_{C}$};
    \celli[la,swap]{4'}{5'}{$(\theta_{D'})_{c}$};

    \path(4) -- node[auto=false]{$=$} (3');
  \end{diagram}
  This concludes the data of $\gamma$, and now we must check that it is pseudo-natural. That is, we must ensure that it is compatible with the compositors of $\cC(-,R-)$ and $\cD(L-,-)$, that its components on identities are themselves identities, and that it commutes appropriately with $2$-cell pasting.

  None of these calculations are especially enlightening, and amount to massaging pasting diagrams with the equality of \eqref{diag:gammadef}, appealing to the fact that $\theta_{d}$ is a modification as in \eqref{diag:pseudomod}, and finally noting that $\theta_{d}$ itself is pseudo-natural. Hence we choose to omit the details for brevity.
\end{proof}

Now we are in a position to give a proof of \Cref{lem:yonedastyle}.

\begin{proof}[\Cref{lem:yonedastyle}]
  First note that \cref{rem:local-nature} directly gives that (i) implies (ii). We show the other implication.

  Suppose (ii), that is, for all objects $D\in \cD$, we have a bi-representation $(RD,\Psi^D)$ of $\cD(L-,D)$, i.e., a pseudo-natural adjoint equivalence $\neinl {\cC(-,RD)}{\Psi^D}{\cD(L-,D)}$. We want to construct the data of a normal pseudo-functor $\ainl\cD R\cC$ and a pseudo-natural adjoint equivalence $\aeinl{\cC(-,R-)}{\Phi_{-,-}}{\cD(L-,-)}$ in $\Psd(\cC^{\op}\times \cD,\Cat)$. 
  
  For this, we will simultaneously construct a pseudo-functor $\ainl{\cC^{\op}\times\cD}{(Q,\phi)}{\Cat}$ such that $Q(C,D)=\cC(C,RD)=\Yoneda_{RD}(C)$ for all $(C,D)\in \cC^{\op}\times \cD$ along with a pseudo-natural adjoint equivalence $\neinl{Q}{\Gamma}{\cD(L-,-)}$ in $\Psd(\cC^{\op}\times \cD,\Cat)$. Note that while our construction of $Q$ below does not necessarily yield a \emph{normal} pseudo-functor, we may apply a normalisation argument such as \cite[Proposition 5.2]{LackPauli} to construct a normal pseudo-functor $Q^{n}$ which agrees with $Q$ on objects and a pseudo-natural isomorphism $\niinl {Q^n}{\nu}{Q}$. Note that there is a forgetful functor $\Psd(\cD,\Psd(\cC^{\op},\Cat))\to\Psd(\cC^{\op}\times\cD,\Cat)$, so that we can see $Q^n$, $Q$, and $\cD(L-,-)$ as objects in $\Psd(\cD,\Psd(\cC^{\op},\Cat))$ and \[ Q^n\underset{\nu}{\xnat{\iso}} Q \underset{\Gamma}{\xnat{\eqv}} \cD(L-,-) \] as an isomorphism in $\Psd(\cD,\Psd(\cC^{\op},\Cat))$.
  
  Then, by applying \cref{lem:yon} to $Q^n$, we may extract a normal pseudo-functor $\ainl \cD R \cC$ and a pseudo-natural isomorphism $\niinl{\cC(-,R-)}{\xi}{Q^{n}}$ in $\Psd(\cD,\Psd(\cC^{\op},\Cat))$. Finally \cref{lem:QisotoR} applied to $L$, $R$, and the isomorphism
  \[ \cC(-,R-)\underset{\xi}{\xnat{\iso}} Q^n\underset{\nu}{\xnat{\iso}} Q \underset{\Gamma}{\xnat{\eqv}} \cD(L-,-) \ \] gives a pseudo-natural isomorphism $\cC(-,R-)\iso\cD(L-,-)$ in $\Psd(\cC^{\op}\times\cD,\Cat)$ as desired.

  It remains to construct the pseudo-functor $\ainl{\cC^{\op}\times\cD}{Q}{\Cat}$ and pseudo-natural adjoint equivalence $\neinl{Q}{\Gamma}{\cD(L-,-)}$.

  On objects $C\in \cC$ and $D\in\cD$, we define $Q(C,D)\coloneqq\cC(C,RD)$ where $RD$ is a representing object which exists by assumption, and we define $\Gamma_{C,D}$ as the adjoint equivalence
  \[ \aeinl{Q(C,D)=\cC(C,RD)}{\Gamma_{C,D}\coloneqq \Psi_C^D}{\cD(LC,D)}. \]
  On morphisms $\ainl Cc{C'}$ in $\cC$ and $\ainl Dd{D'}$ in $\cD$, we must define $Q(c,d)$ and $\Gamma_{c,d}$ such that they fit in the following square:
  \begin{diagram*}[.][4][node distance=1.8cm]
    \node(1)[]{$\cC(C',RD)$};
    \node(2)[right of= 1,xshift=1.5cm]{$\cD(LC',D)$};
    \node(3)[below of= 1]{$\cC(C,RD')$};
    \node(4)[below of= 2]{$\cD(LC,D')$};
    \draw[a](1)to node[la]{$\Gamma_{C',D}$}(2);
    \draw[a](3)to node[la,swap]{$\Gamma_{C,D'}$}(4);
    \draw[a](1)to node[la,swap]{$Q(c,d)$}(3);
    \draw[a](2)to node[la]{$\cD(Lc,d)$}(4);
    \celli[la,swap,shrink][n]{2}{3}{$\Gamma_{c,d}$};
  \end{diagram*}
To do this we, use the equivalence data $(\Psi^{D}_{C},(\Psi_C^D)^{\inv},\eta_C^D,\varepsilon_C^D)$ and set $Q(c,d)$ to be the composite
  \begin{diagram*}[,][1]
    \node(1)[]{$\cC(C',RD)$};
    \node(2)[right of= 1,xshift=1.25cm]{$\cD(LC',D)$};
    \node(3)[right of= 2,xshift=1.5cm]{$\cD(LC,D')$};
    \node(4)[right of= 3,xshift=1.5cm]{$\cC(C,RD')$};
    \draw[a](1)to node[la]{$\Psi^{D}_{C'}$}(2);
    \draw[a](2)to node[la]{$\cD(Lc,d)$}(3);
    \draw[a](3)to node[la]{$(\Psi^{D'}_{C})^{\inv}$}(4);
    \end{diagram*}
    and $\Gamma_{c,d}$ to be the following pasting.
    \begin{diagram*}[.][5]
      \node(1)[]{$\cC(C',RD)$};
      \node(2)[right of= 1,xshift=1.5cm]{$\cD(LC',D)$};
      \node(3)[below of= 2,yshift=-.5cm]{$\cD(LC,D')$};
      \node(4)[below of= 3,yshift=-.5cm]{$\cD(LC,D')$};
      \node(5)[below of=1,yshift=-2.5cm] {$\cC(C,RD')$};
      \draw[a](1)to node[la]{$\Psi^{D}_{C'}$}(2);
      \draw[a](2)to node[la]{$\cD(Lc,d)$}(3);
      \draw[a](3)to node[la,swap,near start]{$(\Psi^{D'}_{C})^{\inv}$}(5);
      \draw[a](5)to node[la,swap,near end]{$\Psi^{D'}_{C}$}(4);
      \draw[d](3)to(4);
      \draw[a](1)to node[la,swap]{$Q(c,d)$}(5);
      \coordinate (t) at ($(3)!0.4!(4)$);
      \celli[la,shrink][n][0.35]{t}{5}{$(\varepsilon^{D'}_{C})^{\inv}$};
    \end{diagram*}
  Next, on $2$-morphisms $\ninl c\alpha{c'}$ in $\cC$ and $\ninl d\beta{d'}$ in $\cD$, we define $Q(\alpha,\beta)$ to be the following pasting:
  \begin{diagram*}[.][4]
    \node(1)[]{$\cC(C',RD)$};
    \node(2)[right of= 1,xshift=1.5cm]{$\cD(LC',D)$};
    \node(3)[right of= 2,xshift=2cm]{$\cD(LC,D')$};
    \node(4)[right of= 3,xshift=1.5cm]{$\cC(C,RD')$};
    \draw[a](1)to node[la]{$\Psi^{D}_{C'}$}(2);
    \draw[a,bend left=30](2)to node(a')[la]{$\cD(Lc,d)$}(3);
    \draw[a,bend right=30](2)to node(b')[la,swap]{$\cD(Lc',d')$}(3);
    \draw[a](3)to node[la]{$(\Psi^{D'}_{C})^{\inv}$}(4);
    \coordinate(a) at ($(a')-(0.65cm,0)$);
    \coordinate(b) at ($(b')-(0.65cm,0)$);
    \cell[la]{a}{b}{$\cD(L\alpha,\beta)$};
  \end{diagram*}
  With this definition of $Q$ on $2$-morphisms, we can directly check that $\Gamma$ is natural with respect to this assignment. More precisely, the following pasting equality holds
  \begin{diagram*}[][][node distance=1.8cm]
    \node(1)[]{$\cC(C',RD)$};
    \node(2)[right of= 1,xshift=1.5cm]{$\cD(LC',D)$};
    \node(3)[below of= 1]{$\cC(C,RD')$};
    \node(4)[below of= 2]{$\cD(LC,D')$};
    \draw[a](1)to node[la]{$\Gamma_{C',D}$}(2);
    \draw[a](1)to node[la,swap]{$Q(c',d')$}(3);
    \draw[a](3)to node[la,swap]{$\Gamma_{C,D'}$}(4);
    \draw[a,bend right=56](2)to node(b)[la,left]{}(4);
    \draw[a,bend left=56](2)to node(a)[la,right]{}(4);
    \coordinate(a') at ($(a)-(0,0.1cm)$);
    \coordinate(b') at ($(b)-(0,0.1cm)$);
    \coordinate(c) at ($(2)-(.3cm,0)$);
    \celli[la,swap,shrink]{c}{3}{$\Gamma_{c',d'}$};
    \cell[la,swap][n]{a'}{b'}{$\cD(L\alpha,\beta)$};

    \node(5)[right of= 2,xshift=2cm]{$\cC(C',RD)$};
    \node(6)[right of=5,xshift=1.5cm]{$\cD(LC',D)$};
    \node(7)[below of=5]{$\cC(C,RD')$};
    \node(8)[below of=6]{$\cD(LC,D')$};
    \draw[a](5)to node[la]{$\Gamma_{C',D}$}(6);
    \draw[a](7)to node[la,swap]{$\Gamma_{C,D'}$}(8);
    \draw[a,bend left=56](5)to node(c)[la]{} (7);
    \draw[a,bend right=56](5)to node(d)[la,swap]{} (7);
    \draw[a](6)to node[la]{$\cD(c,d)$}(8);
    \coordinate(c') at ($(c)-(0,0.1cm)$);
    \coordinate(d') at ($(d)-(0,0.1cm)$);
    \coordinate(j) at ($(7)+(.3cm,0)$);
    \cell[la,swap]{c'}{d'}{$Q(\alpha,\beta)$};
    \celli[la,swap,shrink]{6}{j}{$\Gamma_{c,d}$};

    \node at ($(a.east)!0.5!(d.west)$) {$=$};
  \end{diagram*}
  since both sides are given by the following pasting.
  \begin{diagram*}
      \node(1)[]{$\cC(C',RD)$};
      \node(2)[right of= 1,xshift=1.5cm]{$\cD(LC',D)$};
      \node(3)[below of= 2,yshift=-.5cm]{$\cD(LC,D')$};
      \node(4)[below of= 3,yshift=-.5cm]{$\cD(LC,D')$};
      \node(5)[below of=1, yshift=-2.5cm] {$\cC(C,RD')$};
      \draw[a](1)to node[la]{$\Psi^{D}_{C'}$}(2);
      \draw[a, bend left=50](2)to node(a)[la]{$\cD(Lc,d)$}(3);
      \draw[a, bend right=50](2)to node(b)[la,swap]{$\cD(Lc',d')$}(3);
      \draw[a](3)to node[la,swap,near start]{$(\Psi^{D'}_{C})^{\inv}$}(5);
      \draw[a](5)to node[la,swap,near end]{$\Psi^{D'}_{C}$}(4);
      \draw[d](3)to(4);
      \draw[a](1)to node[la,swap]{$Q(c',d')$}(5);
      \coordinate (t) at ($(3)!0.4!(4)$);
      \celli[la,shrink][n][0.35]{t}{5}{$(\varepsilon^{D'}_{C})^{\inv}$};
      \cell[la,swap][n][0.48]{a}{b}{$\cD(L\alpha,\beta)$};
    \end{diagram*}

  With that achieved, it remains to supply the data of the compositors and unitors of~$Q$ and verify the pseudo-naturality conditions of $\Gamma$ with respect to these. We start with the compositors. Let $\ainl Cc{C'}$ and $\ainl {C'}{c'}{C''}$ be composable morphisms in $\cC$ and  $\ainl Dd{D'}$ and $\ainl {D'}{d'}{D''}$ be composable morphisms in $\cD$. We define the $2$-isomorphism compositor $\niinl{Q(c,d')Q(c',d)}{\phi_{(c',d),(c,d')}}{Q(c'c,d'd)}$ as the below pasting.
  \begin{diagram*}
    \node(1)[]{$\cC(C'',RD)$};
    \node(2)[right of= 1,xshift=1.5cm]{$\cD(LC'',D)$};
    \node(3)[right of= 2,xshift=1.5cm]{$\cD(LC',D')$};
    \node(4)[right of= 3,xshift=1.5cm]{$\cC(C',RD')$};
    \node(5)[below of= 4,yshift=-.5cm]{$\cD(LC',D')$};
    \node(6)[below of= 5,yshift=-.5cm]{$\cD(LC,D'')$};
    \node(7)[below of= 6,yshift=-.5cm]{$\cD(LC,D'')$};
    \draw[a](1)to node[la]{$\Psi^{D}_{C''}$}(2);
    \draw[a](2)to node[la,yshift=2pt]{$\cD(Lc',d)$}(3);
    \draw[a](3)to node[la]{$(\Psi^{D'}_{C'})^{\inv}$}(4);
    \draw[a](4)to node[la]{$\Psi^{D'}_{C'}$}(5);
    \draw[a](5)to node[la]{$\cD(Lc,d')$}(6);
    \draw[a](6)to node[la]{$(\Psi^{D''}_{C})^{\inv}$}(7);
    \draw[d](3)to(5);
    \coordinate(a) at ($(3)!0.5!(5)$);
    \coordinate(b) at ($(2)!0.5!(6)$);
    \draw[a](2)to node[la,over,pos=0.5,xshift=-0.5cm]{$\cD(L(c'c),d'd)$}(6);
    \coordinate (4') at ($(4)-(0.3cm,0)$);
    \coordinate (a') at ($(a)-(0.3cm,0)$);
    \coordinate (b') at ($(b)-(0.3cm,0)$);
    \celli[la,shrink][n][0.6]{4'}{a'}{$\varepsilon^{D'}_{C'}$};
    \celli[la][n][0.4]{a'}{b'}{$\delta_{c,c'}^{*}$};
    \draw[a](1)to node[la,swap]{$Q(c'c,d'd)$}(7);
    \draw[a,bend left=30](1)to node[la]{$Q(c',d)$}(4);
    \draw[a,bend left=50](4.south east)to node[la]{$Q(c,d')$}(7.north east);
  \end{diagram*}
  From this definition, the definition of $Q$ on $2$-morphisms in terms of $L$, and the properties of the compositor $\delta$ of $L$, we may directly verify that $\phi$ is associative and is natural with respect to $2$-morphisms.

  We need to check that $\Gamma$ is compatible with the compositors $\phi$, namely that the following pasting equality holds.
  \begin{diagram*}
    \node(1)[]{$\cC(C'',RD)$};
    \node(2)[right of= 1,xshift=1.5cm]{$\cD(LC'',D)$};
    \node(3)[below of= 1]{$\cC(C',RD')$};
    \node(4)[below of= 2]{$\cD(LC',D')$};
    \node(5)[below of= 3]{$\cC(C,RD'')$};
    \node(6)[below of= 4]{$\cD(LC,D'')$};
    \draw[a](1) to node[la]{$\Gamma_{C'',D}$}(2);
    \draw[a](3)to node[la,yshift=-2pt]{$\Gamma_{C',D'}$}(4);
    \draw[a](5)to node(p1)[la,swap]{$\Gamma_{C,D''}$}(6);
    \draw[a](1)to node[la,over]{$Q(c',d)$}(3);
    \draw[a](3)to node[la,over]{$Q(c,d')$}(5);
    \draw[a,rounded corners](1.west) to ($(1)-(2.25cm,0)$) to node[la,over,near end]{$Q(c'c,d'd)$} node(a)[la,swap]{} ($(5)-(2.25cm,0)$) to (5.west);
    \draw[a](2)to node[la]{$\cD(Lc',d)$}(4);
    \draw[a](4)to node[la]{$\cD(Lc,d')$}(6);
    \celli[la,shrink,swap]{2}{3}{$\Gamma_{c',d}$};
    \celli[la,shrink,swap]{4}{5}{$\Gamma_{c,d'}$};
    \celli[la,swap,yshift=5pt][n][0.62]{3}{a}{$\phi_{(c',d),(c,d')}$};
    
    \node(7)[right of=2,xshift=1.3cm]{$\cC(C'',RD)$};
    \node(8)[right of=7,xshift=1.5cm]{$\cD(LC'',D)$};
    \node(9)[below of=7,yshift=-1.5cm]{$\cC(C,RD'')$};
    \node(10)[below of=8,yshift=-1.5cm]{$\cD(LC,D'')$};
    \draw[a](7)to node(b)[la,over,near end]{$Q(c'c,d'd)$}(9);
    \draw[a](8)to node(b)[auto=false]{} node[la,near end,over]{$\cD(L(c'c),d'd)$}(10);
    \node(11)[right of=b,xshift=.5cm] {$\cD(LC',D')$};
    \draw[a](8)to node[la]{$\cD(Lc',d)$}(11);
    \draw[a](11)to node[la]{$\cD(Lc,d')$}(10);
    \draw[a](7) to node(p2)[la]{$\Gamma_{C'',D}$}(8);
    \draw[a](9) to node[la,swap]{$\Gamma_{C,D''}$}(10);
    \celli[la,swap,shrink]{8}{9}{$\Gamma_{c'c,d'd}$};
    \celli[la,swap,xshift=-5pt][n][0.65]{11}{b}{$\delta^{*}_{c,c'}$};
    
    \coordinate (79) at ($(7)!0.5!(9)$);
    \node at ($(4.east)!0.4!(79)$) {$=$};
  \end{diagram*}
  By direct expansion of definitions, we see that both pastings reduce to the following pasting.
  \begin{diagram*}
   \node(1)[]{$\cC(C'',RD)$};
      \node(2)[right of= 1,xshift=1.5cm]{$\cD(LC'',D)$};
      \node(8)[right of =2,xshift=1.5cm,yshift=-1.5cm]{$\cD(LC',D')$};
      \node(3)[below of= 2,yshift=-1.5cm]{$\cD(LC,D'')$};
      \node(4)[below of= 3,yshift=-.5cm]{$\cD(LC,D'')$};
      \node(5)[below of=1, yshift=-3.5cm] {$\cC(C,RD'')$};
      \draw[a](1)to node[la]{$\Psi^{D}_{C''}$}(2);
      \draw[a](2)to node(a)[la,swap]{$\cD(L(c'c),d'd)$}(3);
      \draw[a](3)to node[la,swap,near start]{$(\Psi^{D''}_{C})^{\inv}$}(5);
      \draw[a](5)to node[la,swap,near end]{$\Psi^{D''}_{C}$}(4);
      \draw[d](3)to(4);
      \draw[a](1)to node[la,swap]{$Q(c'c,d'd)$}(5);
      \draw[a](2) to node[la]{$\cD(Lc',d)$} (8);
      \draw[a](8) to node[la]{$\cD(Lc,d')$} (3);
      \coordinate (t) at ($(3)!0.4!(4)$);
      \celli[la,shrink][n][0.35]{t}{5}{$(\varepsilon^{D''}_{C})^{\inv}$};
      \celli[la]{8}{a}{$\delta_{c,c'}^*$};
  \end{diagram*}
  To complete the proof it remains to deal with the unitors. Given objects $C\in \cC$ and $D\in \cD$, recall that $Q(\id_C,\id_D)$ is given by the following composite
  \begin{diagram*}[,][1]
    \node(1)[]{$\cC(C,RD)$};
    \node(3)[right of= 1,xshift=1.5cm]{$\cD(LC,D)$};
    \node(4)[right of= 3,xshift=1.5cm]{$\cC(C,RD)$};
    \draw[a](1)to node[la]{$\Psi^{D}_{C}$}(3);
    \draw[a](3)to node[la]{$(\Psi^{D}_{C})^{\inv}$}(4);
    \end{diagram*}
    since $\cD(L(\id_C),\id_D)=\id_{\cD(LC,D)}$ by normality of $L$. We define the $2$-isomorphism unitor $\phi_{(C,D)}$ as the unit
    \begin{diagram*}[.][3]
    \node(1)[]{$\cC(C,RD)$};
    \node(2)[right of=1,xshift=1.5cm]{$\cC(C,RD)$};
    \coordinate(d) at ($(1)!0.5!(2)$);
    \node[below of=d](3) {$\cD(LC,D)$};
    \draw[a] (1) to node[la,swap]{$\Psi_C^D$} (3);
    \draw[a] (3) to node[la,swap]{$(\Psi_C^D)^{\inv}$} (2);
    \draw[d] (1) to (2);
    \celli[la][n][0.4]{d}{3}{$\eta_C^D$};
    \end{diagram*}
    From this definition and the triangle equalities for $(\eta_C^D,\varepsilon_C^D)$, we may directly verify that $\phi$ is satisfies the unitality conditions and that $\Gamma$ is compatible with the unitor $\phi$. This completes the constructions of $Q$ and $\Gamma$ and proves the theorem.
\end{proof}

We have now successfully shown that bi-adjunctions $(L,R,\Phi)$ are equivalently families of bi-representations of $\cD(L-,D)$, indexed by the objects $D\in \cD$. In the remainder of the section our goal is to combine \Cref{thm:birep} and \Cref{lem:yonedastyle} in order to obtain the following characterisation of bi-adjunctions in terms of bi-terminal objects in different pseudo-slices. Recall that (double) bi-terminal are defined as (double) bi-initial objects in the (horizontal) opposite.

In the statement of the theorem below, the pseudo-slice double categories $\dblslice{\bH L}D$ are given by the following cospan in $\DblCat$ \[ \bH\cC\xrightarrow{\bH L} \bH\cD\stackrel{D}{\longleftarrow}\mathbbm 1 \ , \]
 and the pseudo-slice $2$-categories $\slice L D$ and $\slice {\Ar L} D$ are given by the following cospans in $\TwoCat$
\[ \cC\stackrel{L}{\longrightarrow} \cD\stackrel{D}{\longleftarrow}\mathbbm 1 \ \ \text{and} \ \ \Ar \cC\xrightarrow{\Ar L} \Ar \cD\stackrel{D}{\longleftarrow}\mathbbm 1 \ , \]
respectively, for objects $D\in \cD$.

\begin{theorem}\label{thm:biadjoint}
  Let $\cC$ and $\cD$ be $2$-categories, and $\ainl \cC{L} \cD$ be a normal pseudo-functor. The following statements are equivalent.
  \begin{rome}
  \item The normal pseudo-functor $\ainl\cC L \cD$ has a right bi-adjoint $\ainl \cD R \cC$.
  \item For all objects $D\in\cD$, there is an object $RD\in \cC$ together with a morphism $\ainl{LRD}{\varepsilon_D}{D}$ in $\cD$ such that $(RD,\varepsilon_D)$ is double bi-terminal in $\dblslice {\bH L} D$.
  \item For all objects $D\in\cD$, there is an object $RD\in \cC$ together with a morphism $\ainl{LRD}{\varepsilon_D}{D}$ in $\cD$ such that $(RD,\varepsilon_D)$ is bi-terminal in $\slice L D$ and $(RD,\id_{\varepsilon_D})$ is bi-terminal in $\slice{\Ar L} D$.
  \item For all objects $D\in\cD$, there is an object $RD\in \cC$ together with a morphism $\ainl{LRD}{\varepsilon_D}{D}$ in $\cD$ such that $(RD,\id_{\varepsilon_D})$ is bi-terminal in $\slice{\Ar L} D$.
  \end{rome}
\end{theorem}

The missing components for the proof of this theorem are canonical isomorphisms of double categories $\eell(\cD(L-,D))\iso (\dblslice {\bH L} D)^{\op}$, as well as related canonical isomorphisms for the $2$-categories $\el(\cD(L-,D))$ and $\mor(\cD(L-,D))$. This is the content of the following results.

\begin{lemma}\label{lem:eell-ld}
  Let $\cC$ and $\cD$ be $2$-categories, $\ainl{\cC}{L}{\cD}$ be a normal pseudo-functor, and $D\in \cD$ be an object. There is a canonical isomorphism of double categories as in the following commutative triangle.
  \begin{diagram*}
    \node(1)[] {$\eell(\cD(L-,D))$};
    \node(2)[right of= 1,xshift=2cm]{$(\dblslice {\bH L} D)^{\op}$};
    \node(3)[] at ($(1)!0.5!(2)-(0,1.5cm)$) {$\bH\cC^{\op}$};
    \draw[a](1) to node[la]{$\iso$}(2);
    \draw[a](2)to node[la]{$\Pi^{\op}$}(3);
    \draw[a](1)to node[la,swap]{$\Pi$}(3);
  \end{diagram*}
\end{lemma}

\begin{proof}
  We describe the data of the double category $\eell(\cD(L-,D))$. Then, by a straightforward comparison with the data in the double category $\dblslice{\bH L}{D}$, which is the dual construction to the double category described in \Cref{def:pseudoslicedouble} with the double functor $F=\bH L$ being horizontal, we can see that the isomorphism above canonically holds.

  An object in $\eell(\cD(L-,D))$ is a pair $(C,f)$ of objects $C\in \cC$ and $f\in \cD(LC,D)$, i.e., a morphism $\ainl{LC}{f}{D}$ in $\cD$. A horizontal morphism $\ainl{(C',f')}{(c,\psi)}{(C,f)}$ in~$\eell(\cD(L-,D))$ comprises the data of a morphism $\ainl{C}{c}{C'}$ in $\cC$ and an isomorphism $\aiinl{f}{\psi}{\cD(Lc,D)f'}$ in $\cD(LC,D)$, i.e., a $2$-isomorphism in~\cD
  \begin{diagram*}[.][3][node distance=1.8cm]
  \node(1)[]{$LC$};
  \node(2)[below of =1]{$LC'$};
  \node(3)[right of =2]{$D$};
  \draw[a](1) to node[la,swap]{$Lc$} (2);
  \draw[a](1) to node(x)[la]{$f$} (3);
  \draw[a](2) to node[la,swap]{$f'$} (3);
  \celli[la]{x}{2}{$\psi$};
  \end{diagram*}
  Note that this corresponds to a morphism $\ainl{(C,f)}{(c,\psi)}{(C',f')}$ in $\dblslice{\bH L}{D}$, and it is the reason why we need to take the horizontal opposite $(\dblslice{\bH L}{D})^{\op}$. A vertical morphism $\vainl{(C,f)}{\alpha}{(C,g)}$ in $\eell(\cD(L-,D))$ is a morphism $\ainl{f}{\alpha}{g}$ in $\cD(LC,D)$, i.e., a $2$-morphism $\ninl{f}{\alpha}{g}$ between morphisms $\ainl{LC}{f,g}{D}$ in $\cD$. Finally, a square $\sq{\gamma}{(c,\psi)}{(d,\varphi)}{\alpha'}{\alpha}$ is a $2$-morphism $\ninl{c}{\gamma}{d}$ in $\cC$ satisfying the pasting equality in \Cref{def:eellsquare}, which can be translated into the following pasting equality in \cD.
\[ \begin{tikzpicture}[node distance=1.8cm,baseline=(2.base)]
  \node(1)[]{$LC$};
  \node(2)[below of= 1]{$LC'$};
  \node(3)[right of= 2]{$D$};
  \draw[a](1)to node(a)[over,la]{$Lc$}(2);
  \draw[a, bend right=60](1)to node(b)[swap,la]{$Ld$}(2);
  \cell[la,swap][n][.4]{a}{b}{$L\gamma$};
  \draw[a](1)to node(f)[la]{$f$}(3);
  \draw[a](2)to node(f')[over,la]{$f'$}(3);
  \draw[a,bend right=60](2)to (3);
  \node(g')[swap,la] at ($(f')-(0,1cm)$) {$g'$};

  \coordinate (f'') at ($(f')-(2pt,0)$);
  \coordinate (g'') at ($(g')-(2pt,0)$);
  \cell[right,la,yshift=5pt][n][.4]{f''}{g''}{$\alpha'$};
  \celli[la][n][.4]{f}{2}{$\psi$};

  \node(4)[right of= a, xshift=.5cm]{$=$};
  \node(1)[right of=1, xshift=2cm]{$LC$};
  \node(2)[below of= 1]{$LC'$};
  \node(3)[right of= 2]{$D$};
  \draw[a](1)to node(a)[swap,la]{$Lc$}(2);
  \draw[a, bend left=50](1)to node(f)[la]{$f$}(3);
  \draw[a](1) to node(g)[over,la]{$g$}(3);
  \draw[a](2)to node(g')[swap,la]{$g'$}(3);
  \celli[la][n][.4]{g}{2}{$\varphi$};
  \cell[la][n][.6]{f}{g}{$\alpha$};
\end{tikzpicture}
\qedhere\]
\end{proof}

\begin{corollary}\label{lem:el-mor-ld}
 Let $\cC$ and $\cD$ be $2$-categories, $\ainl{\cC}{L}{\cD}$ be a normal pseudo-functor, and $D\in \cD$ be an object. There are canonical isomorphisms of $2$-categories as in the following commutative triangles.
  \begin{diagram*}
      \node(1)[] {$\el(\cD(L-,D))$};
      \node(2)[right of= 1,xshift=2cm]{$(\pslice L D)^{\op}$};
      \node(3)[] at ($(1)!0.5!(2)-(0,1.5cm)$) {$\cC^{\op}$};
      \draw[a](1) to node[la]{$\iso$}(2);
      \draw[a](2)to node[la]{$\pi^{\op}$}(3);
      \draw[a](1)to node[la,swap]{$\pi$}(3);

      \node(1)[right of=2,xshift=2cm] {$\mor(\cD(L-,D))$};
      \node(2)[right of= 1,xshift=2cm]{$(\pslice {\Ar L} D)^{\op}$};
      \node(3)[] at ($(1)!0.5!(2)-(0,1.5cm)$) {$\Ar\cC^{\op}$};
      \draw[a](1) to node[la]{$\iso$}(2);
      \draw[a](2)to node[la]{$\pi^{\op}$}(3);
      \draw[a](1)to node[la,swap]{$\pi$}(3);
    \end{diagram*}
\end{corollary}

\begin{proof}
This follows directly from the definitions of $\el$ and $\mor$, \Cref{lem:eell-ld,lem:h-sv-general-commas}.
\end{proof}

The proof of \cref{thm:biadjoint} now follows in a straightforward manner.

\begin{proof}[\Cref{thm:biadjoint}]
  By \Cref{thm:birep,lem:yonedastyle,lem:eell-ld}, we see that (i) and (ii) are equivalent. The equivalences of (ii), (iii), and (iv) follow from \Cref{thm:birep,lem:eell-ld,lem:el-mor-ld}.
\end{proof}

\begin{remark}
  Although we have proven this result by means of formal arguments involving a reformulation of a $2$-dimensional Yoneda lemma (see \cref{lem:evYequiv}), these details are not a necessary feature of the proof of this theorem. For the reader for whom such devices are unfamiliar or otherwise constitute a significant detour, we note here that a pleasingly direct (if somewhat lengthy) proof of this theorem is possible and follows entirely similar lines to the proof of \Cref{thm:birep}.
\end{remark}

Much as in the case of \cref{thm:eeellll-tensors}, we may improve \cref{thm:biadjoint} by assuming that $\cC$ has tensors by $\mathbbm 2$ and that $L$ preserves them, i.e., for every object $C\in \cC$, there is a tensor of $LC$ by $\mathbbm 2$ in $\cD$ and we have an isomorphism $L(C\otimes \mathbbm 2)\cong (LC)\otimes \mathbbm 2$ in~$\cD$ natural with respect to the defining cones.

\begin{theorem}\label{thm:biadjtensor}
  Let $\cC$ and $\cD$ be $2$-categories, and $\ainl \cC{L} \cD$ be a normal pseudo-functor. Suppose that $\cC$ has tensors by $\mathbbm 2$ and that $L$ preserves them. Then the following statements are equivalent.
  \begin{rome}
    \item The normal pseudo-functor $L\colon \cC\to \cD$ has a right bi-adjoint $R\colon \cD\to \cC$.
    \item For all objects $D\in \cD$, there is an object $RD\in \cC$ together with a morphism $\ainl{LRD}{\varepsilon_D}{D}$ in $\cD$ such that $(RD,\varepsilon_D)$ is bi-terminal in $\slice L D$.
  \end{rome}
\end{theorem}

\begin{remark}
  Note that tensors are a special case of a weighted $2$-colimit construction. Therefore, if a $2$-category $\cC$ has tensors by $\mathbbm 2$ and $L\colon \cC\to \cD$ is a left bi-adjoint, it preserves in particular all tensors by $\mathbbm 2$. In this way, this additional hypothesis on $L$ is entirely anodyne in the following sense: given an $L$ which we suspect to be a left bi-adjoint, in order to apply the above theorem we would need to know that $L$ preserves tensors by $\mathbbm 2$, but this should be part of a ``background-check'' on $L$ in the first place.
\end{remark}

\begin{proof}[\cref{thm:biadjtensor}]
  By \cref{thm:biadjoint,thm:eeellll-tensors}, it is enough to show that the normal pseudo-functor $\cD(L-,D)\colon \cC^{\op}\to \Cat$ preserves powers by $\mathbbm 2$. This is indeed the case since it follows from the fact that $L$ preserves tensors by $\mathbbm 2$ that
  \[ \cD(L(C\otimes \mathbbm 2),D)\cong \cD((LC)\otimes \mathbbm 2, D)\cong \Cat(\mathbbm 2, \cD(LC,D)). \qedhere\]
\end{proof}

\subsection{Weighted bi-limits}
\label{sec:weighted-limits}

The primary and indeed motivating application of this theory is to the notion of $2$-dimensional limits. In \cite[Counter-example 2.12]{us}, we give an example of a $2$-terminal object in the slice $2$-category of cones over a $2$-functor $\ainl \cI F\cC$ that does not give a $2$-limit of $F$. This also gives a counter-example of a bi-terminal object in the pseudo-slice $2$-category of cones over $F$ which is not a bi-limit of $F$, as explained in \cite[\S 5]{us}. However, we show in \cite[Proposition 2.13]{us} that when $\cC$ has tensors by $\mathbbm 2$, then $2$-limits and $2$-terminal objects in the slice do correspond precisely. But we deferred the corresponding result for bi-limits to this document.

With this in mind, we now apply \cref{thm:birep} to the case of (weighted) bi-limits in order to obtain a correct characterisation in terms of bi-terminal objects. We further prove the deferred results for (weighted) bi-limits involving tensors by $\mathbbm 2$, which are obtained as a direct application of \cref{thm:eeellll-tensors}.

Let us begin by recalling the definition of a weighted bi-limit.

\begin{definition}\label{def:weighted-limit}
  Let $\cI$ and $\cC$ be $2$-categories, and let $\ainl \cI F{\cC}$ and $\ainl \cI W\Cat$ be normal pseudo-functors. A \textbf{weighted bi-limit of $F$ by $W$} is a pair $(X,\lambda)$ of an object $X\in\cC$ together with a pseudo-natural transformation $\ninl{W}{\lambda}{\cC(X,F-)}$  in $\Psd(\cI,\Cat)$, such that, for every object $C\in \cC$, pre-composition by $\lambda$ induces an adjoint equivalence of categories
  \[ \aeinl{\cC(C,X)}{\lambda^*\circ\cC(-,F)}{\Psd(\cI,\Cat)(W,\cC(C,F-))},\]
  where $\ainl {\cC^{\op}}{\cC(-,F)}{\Psd(\cI,\Cat)}$ is the normal pseudo-functor sending an object $C\in \cC$ to the normal pseudo-functor $ \ainl{\cI}{\cC(C,F-)}{\Cat}$.
\end{definition}

\begin{remark} \label{rem:bi-limit-bi-rep}
  Note that a weighted bi-limit $(X,\lambda)$ induces a $2$-natural adjoint equivalence
  \[ \neinl{\cC(-,X)}{\lambda^*\circ\cC(-,F)}{\Psd(\cI,\Cat)(W,\cC(-,F))}, \]
  so that we can see that weighted bi-limits are, in particular, bi-representations of the $2$-functor $\ainl{\cC^{\op}}{\Psd(\cI,\Cat)(W,\cC(-,F))}{\Cat}$. Conversely, if we have a bi-representation $(X,\rho)$, with $\neinl{\cC(-,X)}{\rho}{\Psd(\cI,\Cat)(W,\cC(-,F))}$ a pseudo-natural adjoint equivalence in $\Psd(\cC^{\op}, \Cat)$, we may set $\ninl{W}{\lambda\coloneqq \rho_X(\id_X)}{\cC(X,F-)}$. Then by \Cref{cor:bi-rep-rectification}, $\overline \rho=\lambda^*\circ\cC(-,F)$ is a $2$-natural adjoint equivalence, that is, a weighted bi-limit of $F$ by~$W$.
\end{remark}

We now aim to apply \cref{thm:birep} to this setting in order to obtain a characterisation of weighted bi-limits in terms of bi-initial objects in different pseudo-slices.

In the statement of the theorem below, the pseudo-slice double category $\dblslice W {\bH\cC(-,F)}$ is given by the following cospan in $\DblCat$
\[ {\mathbbm 1}\xrightarrow{W}{\bH\Psd(\cI,\Cat)}\xleftarrow{\bH\cC(-,F)}{\bH\cC^{\op}} \ , \]
and the pseudo-slice $2$-categories $\slice W {\cC(-,F)}$ and $\slice W {\Ar\cC(-,F)}$ are given by the following cospans in $\TwoCat$
\[ {\mathbbm 1}\xrightarrow{W}{\Psd(\cI,\Cat)}\xleftarrow{\cC(-,F)}{\cC^{\op}} \ \ \text{and} \ \ {\mathbbm 1}\xrightarrow{W}{\Ar\Psd(\cI,\Cat)}\xleftarrow{\Ar\cC(-,F)}{\Ar\cC^{\op}} \ , \]
respectively.

\begin{theorem} \label{thm:weighted}
  Let $\cI$ and $\cC$ be $2$-categories, and let $\ainl \cI F{\cC}$ and $\ainl \cI W\Cat$ be normal pseudo-functors. The following statements are equivalent.
  \begin{rome}
  \item There is a weighted bi-limit $(X,\lambda)$ of $F$ by $W$.
  \item There is an object $X\in \cC$ together with a pseudo-natural transformation $\ninl W{\lambda}{\cC(X,F-)}$ in $\Psd(\cI,\Cat)$ such that $(X,\lambda)$ is double bi-initial in $\dblslice W {\bH\cC(-,F)}$.
  \item There is an object $X\in \cC$ together with a pseudo-natural transformation $\ninl W{\lambda}{\cC(X,F-)}$ in $\Psd(\cI,\Cat)$ such that $(X,\lambda)$ is bi-initial in $\slice W {\cC(-,F)}$ and $(X,\id_{\lambda})$ is bi-initial in $\slice W {\Ar\cC(-,F)}$.
  \item There is an object $X\in \cC$ together with a pseudo-natural transformation $\ninl W{\lambda}{\cC(X,F-)}$ in $\Psd(\cI,\Cat)$ such that $(X,\id_{\lambda})$ is bi-initial in $\slice W {\Ar\cC(-,F)}$.
  \end{rome}
\end{theorem}

\begin{remark}
  At a cursory reading it may surprise readers to learn that weighted \emph{bi-limits} are characterised as somehow \emph{bi-initial} rather than bi-terminal objects. However, such a statement belies their true nature. When we unravel definitions, we see that the double bi-initiality in the pseudo-slice double category $\dblslice W {\bH\cC(-,F)}$ is expressed over $\bH\cC^{\op}$, and its presence is indicative of a ``mapping in'' property for the limiting object in \cC -- precisely as one might expect from bi-limits.
\end{remark}

The proof of \cref{thm:weighted} is deferred to the end of the section, as we need to establish some technical results (\cref{lem:eell-weighted,cor:elp-weighted}) relating the double category $\eell(\Psd(\cI,\Cat)(W,\cC(-,F)))$ to the pseudo-slice double category $\dblslice W {\bH\cC(-,F)}$, and similarly so for $\el(\Psd(\cI,\Cat)(W,\cC(-,F))$ and $\mor(\Psd(\cI,\Cat)(W,\cC(-,F)))$.

Assuming \cref{thm:weighted}, we now specialise \cref{thm:eeellll-tensors} to the weighted bi-limit case. Here we only need to assume that the $2$-category $\cC$ has tensors by $\mathbbm 2$ as these are preserved automatically in this special case.

\begin{theorem}\label{thm:tensors-weighted}
  Let $\cI$ and $\cC$ be $2$-categories, and let $\ainl \cI F{\cC}$ and $\ainl \cI W\Cat$ be normal pseudo-functors. Suppose that $\cC$ has tensors by $\mathbbm 2$. Then the following statements are equivalent.
  \begin{rome}
  \item There is a weighted bi-limit $(X,\lambda)$ of $F$ by $W$.
  \item There is an object $X\in \cC$ together with a pseudo-natural transformation $\ninl W{\lambda}{\cC(X,F-)}$ in $\Psd(\cI,\Cat)$ such that $(X,\lambda)$ is bi-initial in $\slice W {\cC(-,F)}$.
  \end{rome}
\end{theorem}

\begin{proof}
  By \cref{thm:weighted,thm:eeellll-tensors}, it is enough to show that the normal pseudo-functor $\ainl{\cC^{\op}}{\Psd(\cI,\Cat)(W,\cC(-,F))}{\Cat}$ preserves powers by $\mathbbm 2$. Indeed we have that
  \begin{align*}
    \Psd(\cI,\Cat)(W,\cC(C\otimes \mathbbm 2,F-))&  \cong \Psd(\cI,\Cat)(W,\Cat(\mathbbm 2,\cC(C,F-)))\\ & \cong \Cat(\mathbbm 2, \Psd(\cI,\Cat)(W,\cC(C,F-))
  \end{align*}
  as powers in $\Psd(\cI,\Cat)$ are given by point-wise powers in $\Cat$.
\end{proof}

In the special case where the weight $W$ is constant at the terminal category, i.e., $W=\Delta \mathbbm 1$, the characterisation of weighted bi-limits by $\Delta\mathbbm 1$, called \emph{conical bi-limits}, takes a more familiar form.

In the statement of the corollary below, the pseudo-slice double category $\dblslice {\bH\Delta}F$ is given by the following cospan in $\DblCat$
\[ {\bH\cC}\xrightarrow{\bH\Delta}{\bH\Psd(\cI,\cC)}\xleftarrow{F}{\mathbbm 1} \ , \]
and the pseudo-slice $2$-categories $\slice \Delta F$ and $\slice {\Ar\Delta}F$ are given by the following cospans in $\TwoCat$
\[ {\cC}\xrightarrow{\Delta}{\Psd(\cI,\cC)}\xleftarrow{F}{\mathbbm 1} \ \ \text{and} \ \ {\Ar\cC}\xrightarrow{\Ar\Delta}{\Ar\Psd(\cI,\cC)}\xleftarrow{F}{\mathbbm 1} \ , \]
respectively.

\begin{corollary}\label{thm:conicalbilim}
  Let $\cI$ and $\cC$ be $2$-categories, and $\ainl \cI F{\cC}$ be a normal pseudo-functor. The following statements are equivalent.
  \begin{rome}
  \item There is a bi-limit $(X,\lambda)$ of $F$.
  \item There is an object $X\in \cC$ together with a pseudo-natural transformation $\ninl {\Delta X}{\lambda}{F}$ in $\Psd(\cI,\cC)$ such that $(X,\lambda)$ is double bi-terminal in $\dblslice {\bH\Delta} F$.
  \item There is an object $X\in \cC$ together with a pseudo-natural transformation $\ninl {\Delta X}{\lambda}{F}$ in $\Psd(\cI,\cC)$ such that $(X,\lambda)$ is bi-terminal in $\slice \Delta F$ and $(X,\id_{\lambda})$ is bi-terminal in $\slice {\Ar\Delta} F$.
  \item There is an object $X\in \cC$ together with a pseudo-natural transformation $\ninl {\Delta X}{\lambda}{F}$ in $\Psd(\cI,\cC)$ such that $(X,\id_{\lambda})$ is bi-terminal in $\slice {\Ar\Delta} F$.
  \end{rome}
\end{corollary}

The proof is deferred to the end of the section, where we prove the needed technical results (\cref{lem:conical,cor:2conical}) relating pseudo-slices of weighted cones for the weight $W=\Delta \mathbbm 1$ to the usual pseudo-slices of cones.

\begin{remark}
  As we already mentioned, we show in \cite[\S 5]{us} that the data of a bi-limit of $F$ is not fully captured by a bi-terminal object in the usual pseudo-slice $2$-category $\slice \Delta F$ of cones. Statement (iv) above shows that by ``shifting'' the pseudo-slice $\slice \Delta F$ to the pseudo-slice $\slice {\Ar \Delta} F$ whose \emph{objects} are modifications between cones, we can successfully capture the additional data we require.
\end{remark}

In particular, by comparing \cref{thm:conicalbilim} with the characterisation of bi-adjunctions of \Cref{thm:biadjoint}, we can see bi-limits as a right bi-adjoint. Namely:

\begin{remark}
  Let $\cI$ and $\cC$ be $2$-categories. If every normal pseudo-functor $\ainl{\cI}{F}{\cC}$ has a bi-limit, then this bi-limit construction extends to a right bi-adjoint to the diagonal $2$-functor $\ainl{\cC}{\Delta}{\Psd(\cI,\cC)}$.
\end{remark}

Assuming \cref{thm:conicalbilim} and specialising \cref{thm:tensors-weighted} to the case $W=\Delta \mathbbm 1$, we obtain the promised results of \cite[Proposition 5.5]{us}.

\begin{corollary} \label{cor:bilimtensor}
   Let $\cI$ and $\cC$ be $2$-categories, and $\ainl \cI F{\cC}$ be a normal pseudo-functor. Suppose that $\cC$ has tensors by $\mathbbm 2$. Then the following statements are equivalent.
  \begin{rome}
  \item There is a bi-limit $(X,\lambda)$ of $F$.
  \item There is an object $X\in \cC$ together with a pseudo-natural transformation $\ninl {\Delta X}{\lambda}{F}$ in $\Psd(\cI,\cC)$ such that $(X,\lambda)$ is bi-terminal in $\slice \Delta F$.
  \end{rome}
\end{corollary}

\begin{proof}
  This follows directly from \cref{thm:conicalbilim} and \cref{thm:tensors-weighted} applied to $W=\Delta\mathbbm 1$.
\end{proof}

The rest of this section will be devoted to the technical lemmas supporting the proofs of \Cref{thm:weighted,thm:conicalbilim} which give general characterisations of weighted bi-limits and conical bi-limits.

\begin{lemma}\label{lem:eell-weighted}
  Let $\cI$ and $\cC$ be $2$-categories, and let $\ainl \cI F{\cC}$ and $\ainl \cI W\Cat$ be normal pseudo-functors. There is a canonical isomorphism of double categories as in the following commutative triangle.
  \begin{diagram*}
    \node(1)[] {$\eell(\Psd(\cI,\Cat)(W,\cC(-,F)))$};
    \node(2)[right of= 1,xshift=3.5cm]{$\dblslice {W} {\bH\cC(-,F)}$};
    \node(3)[] at ($(1)!0.5!(2)-(0,1.5cm)$) {$\bH\cC^{\op}$};
    \draw[a](1) to node[la]{$\iso$}(2);
    \draw[a](2)to node[la]{$\Pi$}(3);
    \draw[a](1)to node[la,swap]{$\Pi$}(3);
  \end{diagram*}
\end{lemma}

\begin{proof}
  We describe the data of the double category $\eell(\Psd(\cI,\Cat)(W,\cC(-,F)))$. By a straightforward comparison with the data described in \cref{def:pseudoslicedouble} of the pseudo-slice double category $\dblslice {W} {\bH\cC(-,F)}$, we will see that the isomorphism above canonically holds.

  An object in $\eell(\Psd(\cI,\Cat)(W,\cC(-,F)))$ consists of a pair $(C,\kappa)$ of an object $C\in \cC$ and a pseudo-natural transformation $\ninl{W}{\kappa}{\cC(C,F-)}$ in $\Psd(\cI,\Cat)$. A horizontal morphism $\ainl{(C',\kappa')}{(c,\Psi)}{(C,\kappa)}$ in $\eell(\Psd(\cI,\Cat)(W,\cC(-,F)))$ consists of a morphism $\ainl{C}{c}{C'}$ in $\cC$ together with an invertible modification $\Psi$ in $\Psd(\cI,\Cat)$ of the form
  \begin{diagram*}[.][3][node distance=1.8cm]
  \node(1)[]{$W$};
  \node(2)[right of =1,xshift=.25cm]{$\cC(C',F-)$};
  \node(3)[below of =2]{$\cC(C,F-)$};
  \draw[n](1) to node[la]{$\kappa'$} (2);
  \draw[n](1) to node(x)[swap,la]{$\kappa$} (3);
  \draw[n](2) to node[la]{$\cC(c,F-)$} (3);
  \celli[la,shrink][t][0.4]{x}{2}{$\Psi$};
  \end{diagram*}
  A vertical morphism $\vainl{(C,\kappa)}{\Theta}{(C,\mu)}$ in $\eell(\Psd(\cI,\Cat)(W,\cC(-,F)))$ is a modification $\minl{\kappa}{\Theta}{\mu}$ between pseudo-natural transformations $\ainl{W}{\kappa,\mu}{\cC(C,F-)}$ in $\Psd(\cI,\Cat)$. Finally, a square $\sq{\gamma}{(c,\Psi)}{(d,\Phi)}{\Theta'}{\Theta}$ is a $2$-morphism $\ninl{c}{\gamma}{d}$ in $\cC$ satisfying the pasting equality in \Cref{def:eellsquare}, which can be translated into a pasting equality for the modification $\cC(\gamma,F-)$ in $\Psd(\cI,\Cat)$.
\end{proof}

\begin{corollary} \label{cor:elp-weighted}
  Let $\cI$ and $\cC$ be $2$-categories, and let $\ainl \cI F{\cC}$ and $\ainl \cI W\Cat$ be normal pseudo-functors. There are canonical isomorphisms of $2$-categories as in the following commutative triangles.
\begin{center}
\begin{tikzpicture}
      \node(1)[] {$\el(\Psd(\cI,\Cat)(W,\cC(-,F)))$};
      \node(2)[right of= 1,xshift=3cm]{$\pslice W {\cC(-,F)}$};
      \node(3)[] at ($(1)!0.5!(2)-(0,1.5cm)$) {$\cC^{\op}$};
      \draw[a](1) to node[la]{$\iso$}(2);
      \draw[a](2)to node[la]{$\pi$}(3);
      \draw[a](1)to node[la,swap]{$\pi$}(3);
\end{tikzpicture}
\begin{tikzpicture}
      \node(1')[below of=1, yshift=-1cm] {$\mor(\Psd(\cI,\Cat)(W,\cC(-,F)))$};
      \node(2)[right of= 1',xshift=3.5cm]{$\pslice W {\Ar \cC(-,F)}$};
      \node(3)[] at ($(1')!0.5!(2)-(0,1.5cm)$) {$\Ar\cC^{\op}$};
      \draw[a](1') to node[la]{$\iso$}(2);
      \draw[a](2)to node[la]{$\pi$}(3);
      \draw[a](1')to node[la,swap]{$\pi$}(3);
    \end{tikzpicture}
    \end{center}
\end{corollary}

\begin{proof}
  This follows directly from the definitions of $\el$ and $\mor$, \Cref{lem:eell-weighted,lem:h-sv-general-commas}.
\end{proof}

With \Cref{lem:eell-weighted,cor:elp-weighted} above established we may now give a direct proof of \Cref{thm:weighted}.

\begin{proof}[\Cref{thm:weighted}]
Recall from \Cref{rem:bi-limit-bi-rep} that a weighted bi-limit of $F$ by $W$ is equivalently a bi-representation of the $2$-functor $\Psd(\cI,\Cat)(W,\cC(-,F))$. Then the result is obtained as a direct application of \Cref{thm:birep} using \Cref{lem:eell-weighted,cor:elp-weighted}.
\end{proof}

In the conical case we may simplify the pseudo-slices above through the below computations to obtain a proof of \cref{thm:conicalbilim}.

\begin{lemma} \label{lem:conical}
  Let $\cI$ and $\cC$ be $2$-categories, and $\ainl \cI F\cC$ be a normal pseudo-functor. There is a canonical isomorphism of double categories as in the following commutative triangle.
  \begin{diagram*}
    \node(1)[] {$\dblslice {\Delta \mathbbm 1}{\bH\cC(-,F)}$};
    \node(2)[right of= 1,xshift=2cm]{$(\dblslice{\bH\Delta} F)^{\op}$};
    \node(3) at ($(1)!0.5!(2)-(0,1.5cm)$) {$\bH\cC^{\op}$};
    \draw[a](1)to node[la]{$\iso$}(2);
    \draw[a](1)to node[la,swap]{$\Pi$}(3);
    \draw[a](2)to node[la]{$\Pi^{\op}$}(3);
  \end{diagram*}
\end{lemma}

\begin{proof}
  This follows from the fact that, given an object $C\in \cC$, a pseudo-natural transformation $\ninl{\Delta\mathbbm 1}{\kappa}{\cC(C,F-)}$ in $\Psd(\cI,\Cat)$ corresponds to a pseudo-natural transformation $\ninl{\Delta C}{\kappa}{F}$ in $\Psd(\cI,\cC)$.
\end{proof}

\begin{corollary} \label{cor:2conical}
  Let $\cI$ and $\cC$ be $2$-categories, and $\ainl \cI F\cC$ be a normal pseudo-functor. There are canonical isomorphisms of $2$-categories as in the following commutative triangles.
  \begin{diagram*}
      \node(1)[] {$\slice{\Delta\mathbbm 1}{\cC(-,F)}$};
      \node(2)[right of= 1,xshift=2cm]{$(\slice \Delta F)^{\op}$};
      \node(3)[] at ($(1)!0.5!(2)-(0,1.5cm)$) {$\cC^{\op}$};
      \draw[a](1) to node[la]{$\iso$}(2);
      \draw[a](2)to node[la]{$\pi^{\op}$}(3);
      \draw[a](1)to node[la,swap]{$\pi$}(3);

      \node(1)[right of=2,xshift=2cm] {$\slice{\Delta\mathbbm 1}{\Ar\cC(-,F)}$};
      \node(2)[right of= 1,xshift=2cm]{$(\pslice {\Ar \Delta} F)^{\op}$};
      \node(3)[] at ($(1)!0.5!(2)-(0,1.5cm)$) {$\Ar\cC^{\op}$};
      \draw[a](1) to node[la]{$\iso$}(2);
      \draw[a](2)to node[la]{$\pi^{\op}$}(3);
      \draw[a](1)to node[la,swap]{$\pi$}(3);
    \end{diagram*}
\end{corollary}

\begin{proof}
  This follows directly from \Cref{lem:conical} and \Cref{lem:h-sv-general-commas}.
\end{proof}

Finally we obtain a straightforward proof of \Cref{thm:conicalbilim}.

\begin{proof}[\Cref{thm:conicalbilim}]
  This result is obtained by applying \Cref{thm:weighted} to the special case where $W=\Delta \mathbbm 1$ and using \Cref{lem:conical,cor:2conical}.
\end{proof}

\bibliographystyle{plain}
\bibliography{references}

\end{document}